\documentclass[a4paper,11pt]{article}
\usepackage{graphicx} 
\usepackage[english]{babel}
\usepackage[utf8]{inputenc}
\usepackage{amsmath}
\usepackage{amsthm}
\usepackage{amssymb}
\usepackage{url}
\usepackage{hyperref}
\usepackage{mathabx}
\usepackage{biblatex}
\usepackage{enumerate}
\usepackage{appendix}
\usepackage{float}
\usepackage[colorinlistoftodos]{todonotes}
\usepackage[skip=5pt, indent=20pt]{parskip}
\usepackage{setspace}
\usepackage{multirow}
\usepackage[font=small,labelfont=bf]{caption}
\usepackage{geometry}
\newtheorem{lemma}{Lemma}[section]
\newtheorem{theorem}{Theorem}[section]
\newtheorem{remark}{Remark}[section]
\newtheorem{corollary}{Corollary}[section]

\allowdisplaybreaks[4]

\title{High-Moment Stability and Error Analysis of a Fully Discrete LDG-IMEX Method for High Dimensional Nonlinear Stochastic Convection-Diffusion Equations}
\author{    
            Yiming Chen\thanks{
			Department of Mathematics, The Ohio State University, Columbus, OH 43210,
			USA. Email: \href{mailto:chen.11042@osu.edu}{chen.11042@osu.edu}. }
            \and
            Yunzhang Li\thanks{
			Research Institute of Intelligent Complex Systems, Fudan University, Shanghai 200433, P.R. China. Email: \href{mailto:li_yunzhang@fudan.edu.cn}{li\_yunzhang@fudan.edu.cn}.}
            \and 
			Yulong Xing\thanks{
			Department of Mathematics, The Ohio State University, Columbus, OH 43210,
			USA. Email: \href{mailto:xing.205@osu.edu}{xing.205@osu.edu}.}             
		    }
\date{\vspace{-2ex}}

\begin{document}
\pagenumbering{arabic}

\maketitle

\begin{abstract}
A fully discrete local discontinuous Galerkin (LDG) method coupled with an implicit-explicit (IMEX) Euler time discretization is presented and analyzed for a class of high dimensional nonlinear stochastic convection-diffusion equations driven by multiplicative $\mathcal Q$-Wiener noise. 
The model allows nonlinear leading coefficients, nonlinear convection terms, dissipative source terms, and gradient-dependent noise. 
The diffusion operator is treated implicitly through the LDG formulation, while the nonlinear convection, lower-order drift, and stochastic terms are evaluated explicitly.
The main contribution is a high-moment stability and error analysis for the fully discrete scheme. 
A central difficulty is that the nonlinear terms lead to pathwise growth factors that cannot be controlled uniformly on the full sample space. To provide the stability and error estimate, we introduce recursively defined nested subsets adapted to the numerical solution.
Under the stated stochastic parabolicity and refinement conditions, we prove that these subsets have probabilities converging to one.
On these subsets, the numerical solution satisfies high-moment stability, and the fully discrete error converges with order arbitrarily close to $r+1$ in space and $1/2$ in time. 
We also derive a pathwise error estimate by combining the high-moment error bound with a discrete Kolmogorov argument. 
Numerical experiments for stochastic Burgers' and Allen-Cahn equations confirm the theoretical rates and demonstrate the robustness of the proposed method for nonlinear stochastic models.
\end{abstract}
\vspace{2ex}




\begin{section}{Introduction}
Stochastic convection-diffusion equations arise naturally in the modeling of transport, diffusion, and reaction processes subject to random perturbations. 
They appear in applications such as fluid flow, heat transfer, porous-media transport, and phase-field dynamics, where uncertainty may enter through external forcing, unresolved scales, or random material properties. 
Compared with deterministic convection-diffusion equations, stochastic models introduce additional analytical and numerical difficulties: solutions are generally not differentiable in time, pathwise maximum principles may fail, and nonlinear drift or noise terms can generate growth factors that are difficult to control uniformly on the full sample space. 
These issues become more pronounced when the diffusion matrix depends on the solution, the convective fluxes have polynomial growth, the drift contains a nonlinear dissipative reaction, and the multiplicative noise depends on the solution gradient.
In this paper, we propose and analyze a fully discrete local discontinuous Galerkin (LDG) method coupled with an implicit-explicit (IMEX) time discretization for such nonlinear stochastic convection-diffusion equations.

We consider the following nonlinear stochastic convection-diffusion equation with a solution-dependent leading diffusion matrix, driven by multiplicative $\mathcal Q$-Wiener noise:
\begin{gather}
\label{spde:conv-diff}
    \begin{cases}
        \mathrm{d}u = \left\{\nabla \!\cdot\! [A(\cdot,u)\nabla u] - \nabla \!\cdot\! F(u) + \psi(\cdot,u,\nabla u) \right\}\mathrm{d}t + g(\cdot,u,\nabla u) \mathrm{d}W_t,  & (\omega,x,t)\in \Omega\times \mathcal{D} \times (0,T], \\
        u(\cdot,0) = u_0,  & (\omega,x) \in \Omega \times \mathcal{D}.
    \end{cases}
\end{gather}
where $\mathcal{D}=\mathbb{T}^d$ is the $d$-dimensional torus with periodic boundary conditions, $A = [a_{ij}(\omega,x,t,u)]$ is a symmetric, uniformly positive definite matrix, and $F(u) = (f_1(u),\ldots,f_d(u))$ denotes the vector of nonlinear convection flux functions. 
The lower-order drift term $\psi$ and the noise coefficient $g$ may depend on both the solution and its gradient.
Let $\{e_m\}_{m\geq1}$ be a set of orthonormal bases of $L^2(\mathcal{D})$ and $\mathcal{Q}$ be a symmetric Hilbert-Schmidt operator on $L^2(\mathcal{D})$ with finite trace, where $\mathcal{Q} e_m = \gamma_m e_m$, $\gamma_m > 0$, and $\text{Tr}(\mathcal{Q}) := \sum_{m=0}^{\infty} \gamma_m < \infty$. The space-time $\mathcal Q$-Wiener process $\{W_t\}_{t\in[0,T]}$ is defined as 
\[ 
    W_t = W_t(t,x,\omega) = \sum_{m=1}^{\infty} \sqrt{\gamma_m} e_m(x)\mathcal{B}_m(t,\omega) \quad x\in \mathcal{D}, 
\]
where $\{ \mathcal{B}_m\}_{m\geq 1}$ is a set of independent standard Brownian motions defined on the filtered probability space $(\Omega, \mathcal{F},\{ \mathbb{F}_t\}_{0 \leq t \leq T}, \mathbb{P})$.

The well-posedness and regularity of parabolic stochastic partial differential equations (SPDEs) have been extensively studied by semigroup methods \cite{Chow2007_Parabolic_SPDEs, Prato_Zabczyk_2014}, variational approaches \cite{Krylov_Rozovskii_1981, Prevot_Rockner2007}, and $L^p$-regularity theory \cite{Krylov_SIAMNu1996, Krylov_AMS1999}.
Existence, uniqueness, and regularity results are available for linear and semilinear equations \cite{Chow2007_Parabolic_SPDEs, Prato_Zabczyk_2014}, as well as for various quasilinear problems \cite{Debussche_SIAMMA2015, DuLiu_2019Cauchy, Krylov_AMS1999, Pardoux_Peng_1994}. 
These theories provide the functional framework and solution regularity required for the stability and convergence analysis of numerical approximations.
Since analytic solutions are rarely available, efficient and accurate numerical methods are essential for practical simulation of SPDEs. 
We refer to \cite{LiShuTang2021_ESAIM, YangZhaoZhao2023_NMPDE,ondrejat_Prohl_Walkington_2023} for further references on finite difference, finite element, and spectral methods for parabolic SPDEs.

The discontinuous Galerkin (DG) methods considered in the paper are high-order finite element methods with discontinuous piecewise polynomial approximation spaces. 
They were originally introduced for neutron transport problems \cite{ReedHill1973} and were later developed extensively for hyperbolic conservation laws, especially through the Runge-Kutta DG framework \cite{CockburnShu_I_M2AN_1991, cockburnShu_II_Mathcomp1989, cockburnShu_III_JCP1989, cockburnHouShu_IV_Mathcomp1990, CockburnShu_V_JCP1998}. 
The LDG method was later introduced for convection-diffusion equations by rewriting the higher-order spatial operators as first-order systems \cite{CockburnShu_SIAMNu1998}.
DG and LDG methods are attractive for convection-dominated and nonlinear problems because of their local conservation, high-order accuracy, flexibility for complex geometries, suitability for $h$-$p$ adaptivity, and efficient parallel implementation. 
The LDG methods have been successfully applied to many deterministic nonlinear problems, such as Navier-Stokes equation \cite{Cockburn_Mathcomp2005_NSE}, wave equation \cite{ChouShuXing_JCP2014, GuoXing_JSC2021}, convection-diffusion equation \cite{Castillo_Cockburn_Mathcomp2002} and KdV-type equations \cite{XuShu_CMAME2007, YanShu_SIAMNu2002, KX2016}. 
Fully discrete stability and error estimates for DG and LDG methods have also been developed for deterministic nonlinear problems;
for example, Zhang and Shu analyzed explicit Runge-Kutta DG schemes for scalar conservation laws \cite{ZhangShu2004, ZhangShu2010}, and Wang \textit{et al.} studied LDG schemes with IMEX Runge-Kutta time stepping for nonlinear convection-diffusion equations \cite{WangShuZhang2016AMC}.
For problems with nonlinear convection and stiff diffusion, IMEX time discretizations are particularly useful because they treat the diffusion term implicitly while allowing nonlinear convection terms to be handled explicitly.

In recent years, there has been a growing interest in designing DG and finite element methods for nonlinear stochastic partial differential equations, including stochastic Navier-Stokes equation \cite{FengVo_CiCP2024}, nonlinear wave equations \cite{HongHouSun_JCP2022, Li_Wu_Xing_2022}, the stochastic Allen-Cahn equation \cite{YangZhaoZhao_CiCP2024}, the stochastic Cahn-Hilliard equation \cite{LiQinMingWang_CMA2018, zhouLi_CiCP2022_LDG} and the stochastic Boussinesq equations \cite{Vo_Boussinesq_Multiplicative2025}. 
Among these works, Li \textit{et al.} studied one-dimensional DG, ultra-weak DG and LDG methods for stochastic conservation laws \cite{LiShuTang2020_SISC}, stochastic KdV equation \cite{LiShuTang_JSC2020} and stochastic parabolic equation \cite{LiShuTang2021_ESAIM}, respectively. 
In these studies, the analysis is mainly carried out for semi-discrete methods, and optimal error estimates are obtained in the semilinear setting. 
However, a complete fully discrete high-moment error analysis for nonlinear SPDEs remains substantially more challenging.
At the fully discrete level, the low temporal regularity of stochastic solutions must be handled simultaneously with the spatial projection error and stochastic increments. 
Moreover, polynomially growing convective fluxes and reaction terms produce random growth factors involving pathwise norms of the exact and numerical solutions.
Solution-dependent diffusion and gradient-dependent noise introduce additional coupling among the errors in the solution, its gradient, and the LDG auxiliary flux variables. 
Because these random growth factors cannot generally be bounded uniformly on the full sample space, a direct application of the deterministic error argument may not work in a straightforward way.
These difficulties motivate the development of a fully discrete LDG analysis that can handle nonlinear diffusion, nonlinear convection and source terms, and gradient-dependent multiplicative noise. 

The present paper extends the quasilinear fully discrete analysis in \cite{chen_semi-linear_2d} to a high-dimensional fully discrete LDG-IMEX-Euler approximation of nonlinear stochastic convection-diffusion equations on Cartesian meshes. 
Compared with the semilinear model, the equation studied here includes solution-dependent leading coefficients, nonlinear convection terms, and a dissipative nonlinear source term. 
These additional nonlinearities generate pathwise growth factors that are absent from the semilinear error analysis and cannot, in general, be controlled uniformly on the full sample space.
The main contributions of this paper are fourfold. 
First, we introduce novel recursively defined localized subsets of the sample space to control the nonlinear growth factors appearing in the fully discrete stability and error analyses.
Subset and stopping-time type arguments have been used in the SPDE numerical literature to facilitate discrete Gr\"onwall estimates or to recover pathwise regularity on high-probability events \cite{Carelli_Prohl_SIAMNu2012, Li_Wu_Xing_2022, Vo_Boussinesq_Multiplicative2025}. 
In the present setting, the relevant growth factors depend on the numerical solution itself, and hence the subsets must be constructed recursively in time.
This construction keeps the indicators compatible with the adaptedness of the numerical solution and the martingale estimates. 
On these localized subsets, the nonlinear growth factors can be controlled pathwise, and we prove that the subsets converge to the full sample space in probability as the discretization parameters tend to zero.
Second, we establish fully discrete high-moment stability estimates for the nonlinear LDG-IMEX-Euler scheme on the localized subsets. 
The proof is nontrivial because the leading diffusion matrix is solution-dependent, while the lower-order drift and multiplicative noise may be driven by the solution gradient. 
The energy argument must absorb and control the gradient-dependent drift and stochastic terms using the coercivity of the implicit LDG diffusion operator. 
This leads to a stochastic parabolicity condition on the leading coefficient matrix.
Third, we prove high-moment error estimates for the fully discrete scheme in the general nonlinear setting. 
On the localized subsets, the resulting error bound is arbitrarily close to the optimal order $r+1$ in space and order $1/2$ in time. 
The proof must control the spatial projection error, the temporal discretization error, the stochastic error, and the nonlinear convection and source terms in a unified argument. 
In two dimensions, additional LDG coupling terms involving the auxiliary variables and projection errors do not vanish. 
We derive auxiliary estimates relating the errors in the numerical flux and gradient variables in order to close the high-moment error estimate.
Fourth, we derive a pathwise error estimate for the fully discrete method by applying a discrete Kolmogorov lemma to the high-moment error bound.
The framework covers stochastic Burgers-type equations with nonlinear convection and gradient-dependent noise, as well as stochastic Allen--Cahn-type equations through the dissipative reaction term.
We present the analysis in two space dimensions to clearly demonstrate the main ideas, and the framework can be extended to higher-dimensional Cartesian meshes.
The treatment of more general meshes and fully implicit or higher-order time discretizations is left for future work.

To facilitate the analysis, we use the following notation throughout the paper. 
Let $(\cdot, \cdot)$ represent the standard $L^2$ inner product over the domain $\mathcal{D}$. 
For a non-negative integer $r$, the spatial $L^2$ and Sobolev $H^r$ norms with respect to $x \in \mathbb{R}^d$ are designated as $\| \cdot\|$ and $\| \cdot\|_r := \| \cdot\|_{H^r}$, respectively. 
For a stochastic process $\phi:\Omega\times[0,T]\to H^r(\mathcal D)$, we use the norms
$$\| \phi\|_{L^{2}(\Omega \times [0,T]; H^{r})}^2 := \mathbb{E}\left[ \int_0^T \| \phi(t) \|_{H^r}^2 \mathrm{d}t\right] \quad \mathrm{and} \quad \| \phi\|_{L^{2}(\Omega,L^{\infty}[0,T; H^{r}])}^2 := \mathbb{E}\left[ \sup_{t \in [0,T]} \| \phi(t) \|_{H^r}^2 \right].$$ 
We make the assumption that the quantity $K := \sum_{i=0}^{\infty} \gamma_i \| e_i(x)\|_{\infty}^2$ remains bounded throughout this work. 
For any predictable stochastic process $\Phi(t)$ taking values in $\mathcal{L}_2^r := \mathcal L_2\bigl(\mathcal Q^{1/2}L^2(\mathcal D); H^r(\mathcal D)\bigr)$, its corresponding operator norm is given by $\| \Phi(t)\|_{\mathcal{L}_2^r}^2 := \sum_{m=1}^{\infty} \gamma_m \| \Phi e_m\|_r^2$. 
In this context, we require the sequence of basis functions $\{ e_i(x)\}_{i=0}^{\infty}$ to reside in $H^r(\mathcal{D})$. 
Lastly, we use $C>0$ to denote a generic constant independent of the temporal and spatial mesh parameters which may vary from line to line.  
In the two-dimensional analysis, we take $\mathcal{D} = [0,2\pi]^2$ with periodic boundary conditions.

The rest of the paper is organized as follows. Section \ref{sec:2d-scheme-prelim} introduces the two-dimensional fully-discrete LDG method coupled with an IMEX-Euler time discretization. It also provides preliminary results necessary for the analysis. Section \ref{sec:stability-2d} establishes high-moment stability estimates for the fully discrete scheme on localized subsets of the sample space and proves that these subsets converge to the full sample space in probability.
The high moment error estimates and the corresponding pathwise error estimate are derived in Section \ref{sec:error-estimate-2d}. Numerical tests are presented in Section \ref{sec:numerical-test} to verify the theoretical stability condition and convergence rate, and illustrate the performance of the proposed method for nonlinear stochastic problems.
Section \ref{sec:conclusion} summarizes the main conclusions.
Some technical proofs are provided in Appendix \ref{appendixA}.

\end{section}

\begin{section}{IMEX-LDG Scheme and Preliminaries}
\label{sec:2d-scheme-prelim}

In this section, we present the fully discrete scheme for the two-dimensional (2D) stochastic convection-diffusion equation. Some preliminary lemmas are also proposed for the subsequent stability and error analyses. We start by considering the two-dimensional version of the model \eqref{spde:conv-diff} on the periodic domain $\mathcal{D}$, which takes the form
\begin{gather}
\label{spde:conv-diff-2d}
    \begin{cases}
        \mathrm{d}u =  \nabla \!\cdot\! \bigl(A(\cdot,u)\nabla u \bigr) \mathrm{d}t - f_1(u)_x \,\mathrm{d}t - f_2(u)_y \,\mathrm{d}t \\
        \quad\quad + \psi(\cdot,u,u_x,u_y) \mathrm{d}t + g(\cdot,u,u_x,u_y) \mathrm{d}W_t, \quad & (\omega,x,y,t) \in \Omega \times \mathcal{D} \times (0,T]; \\
        u(\cdot,x,y,0) = u_0(x,y),  & (\omega,x,y) \in \Omega \times \mathcal{D}.
    \end{cases}
\end{gather}
The leading diffusion matrix $A = \{ a_{ij}(\cdot,x,y,t,u)\}$ may depend nonlinearly on the solution $u$. The nonlinear convection terms are given by $f_1(u)_x$ and $f_2(u)_y$. The term $\psi(\cdot,x,y,t,u,u_x,u_y)$ represents a general lower-order drift including convection or nonlinear source term specified below. The stochastic perturbation is represented by the multiplicative noise term $g(\cdot,x,y,t,u,u_x,u_y) \mathrm{d}W_t$, where $W_t$ is a $\mathcal{Q}$-Wiener process. We make the following assumptions on these terms in \eqref{spde:conv-diff-2d}:
\begin{enumerate}[(i)]
    \item (Initial condition) The initial condition $u_0 \in H^1(\mathcal{D})$.
    \item ({Leading coefficient}) The leading coefficient matrix $A=\{ a_{ij}\}(\cdot,x,y,t,u)$ is infinitely differentiable in each variable. Moreover, there exist $\alpha, \Lambda>0$, and $\alpha_1, \alpha_2 \geq 0$, such that for $i,j = 1,2$,
    \begin{align*}
        &\xi^{\top} A(\omega,x,y,t,u)\xi \geq \alpha | \xi |^2, \qquad |a_{ij}(\omega,x,y,t,u)|^2 \leq \Lambda, \\
        &|a_{ij}(\omega,x,y,t,u)-a_{ij}(\omega,x,y,t',u')| \leq \alpha_1|u-u'| + \alpha_2|t-t'|^{\frac{1}{2}}(1+|u|+|u'|), 
    \end{align*}
    for any $\xi \in \mathbb{R}^2$, any $(\omega,x,y,t,t',u,u') \in \Omega \times \mathcal{D} \times [0,T]^2 \times \mathbb{R}^2$.

    \item (Flux function) The nonlinear flux functions $f_1,f_2$ are locally Lipschitz continuous and may have polynomial growth, i.e., there exists $p \geq 1$, such that, for $i=1,2$, and all $y,z \in \mathbb{R}$, 
    \begin{align*}
        &|f_i'(y)| \leq C( 1 + |y|^p), \\
        &|f_i'(y) - f_i'(z)| \leq C( 1+|y|^{p-1}+ |z|^{p-1}) |y-z|, \\
        &|f_i(y) - f_i(z)| \leq C( 1+|y|^p+ |z|^p) |y-z|.
    \end{align*}
    
    \item (Lower-order drift) We consider $\psi$ of the form $\psi = \psi_1 - Cu^{\lambda}$, where $\lambda \geq 3$ is an odd integer, $C>0$, and $\psi_1$ satisfies
    \begin{align*}
        &|\psi_1(\omega,x,y,t,u,v_1,v_2)|^2 \leq B_2^2(1+|u|^2) + B_3^2 (|v_1|^2+|v_2|^2), \\
        &|\psi_1(\omega,x,y,t,u,v_1,v_2)-\psi_1(\omega,x,y,t,u',v_1',v_2')| \leq B_1(|u-u'| + |v_1-v_1'| + |v_2-v_2'|), \\
        &|\psi_1(\omega,x,y,t,u,v_1,v_2)-\psi_1(\omega,x,y,t',u,v_1,v_2)| \leq B_4|t-t'|^{\frac{1}{2}}(1 + |u| + |v_1| + |v_2|), 
    \end{align*}
    for some nonnegative constants $B_1,B_2,B_3,B_4$, and for any $(\omega,x,y,t,t',u,u',v_1,v_1',v_2,v_2') \in \Omega \times \mathcal{D} \times [0,T]^2 \times \mathbb{R}^6$.
    
    \item (Noise Coefficient) There exist nonnegative constants $D_1,\ldots,D_5$ such that
    \begin{align*}
        &|g(\omega,x,y,t,u,v_1,v_2)|^2 \leq D_3^2(1+|u|^2) + D_4^2(|v_1|^2+|v_2|^2), \\
        &|g(\omega,x,y,t,u,v_1,v_2)-g(\omega,x,y,t,u',v_1',v_2')| \leq D_1|u-u'| + D_2(|v_1-v_1'|+|v_2-v_2'|), \\
        &|g(\omega,x,y,t,u,v_1,v_2)-g(\omega,x,y,t',u,v_1,v_2)| \leq D_5|t-t'|^{\frac{1}{2}}(1 + |u| + |v_1|+|v_2|),
    \end{align*}
    for any $(\omega,x,y,t,t',u,u',v_1,v_1',v_2,v_2') \in \Omega \times \mathcal{D} \times [0,T]^2 \times \mathbb{R}^6$.
\end{enumerate}
The Lipschitz continuity and linear growth assumptions on $\psi_1$ and
$g$ in (iv) and (v) are standard conditions used to guarantee the existence and uniqueness of solutions of the semi-discrete DG scheme. Compared with \cite{chen_semi-linear_2d,LiShuTang2021_ESAIM}, we include nonlinear convection terms $f_1,f_2$ and source term $-Cu^{\lambda}$, so that \eqref{spde:conv-diff-2d} includes nonlinear stochastic convection-diffusion equations and Allen--Cahn-type equations.

\begin{subsection}{Fully-Discrete Numerical Scheme}

Consider a quasi-uniform rectangular partition of the domain $\mathcal{D}$, denoted by 
$\mathcal{T}_h = \bigl\{ I_i \times J_j = [x_{i-1/2}, x_{i+1/2}] \times [y_{j-1/2}, y_{j+1/2}],\, i = 1,2,..N_x;\, \, j = 1,2,...N_y \bigr\}$. We define the cell center $(x_i,y_j)$ by $x_i = (x_{i-1/2}+ x_{i+1/2})/2$ and $y_j = (y_{j-1/2} + y_{j+1/2})/2$. The local element widths in the $x$- and $y$-directions are given as $h_{x,i} = x_{i+1/2}- x_{i-1/2}$, $h_{y,j} = y_{j+1/2} - y_{j-1/2}$ with $h_x = \max_i h_{x,i}$, $h_y = \max_j h_{y,j}$, and $h = \max(h_x,h_y)$ being the maximum mesh size. By the quasi-uniformity assumption, there exists a positive constant bounding the ratio $\min(h_{x,i},h_{y,j})/h$ from below for all $i,j$ as $h \to 0$. 

Let $H^1(\mathcal{T}_h)$ be the mesh-dependent broken Sobolev space. For a nonnegative integer $r$, let $P^r(I_i)$ be the space of one-dimensional polynomials of degree at most $r$ on $I_i$, and let $ Q^r(I_i \times J_j):= P^r(I_i) \otimes P^r(J_j)$ be the two-dimensional tensor product polynomial space on the cell $I_i \times J_j$. The discontinuous finite element space of tensor-product polynomials of degree at most $r$ in each variable is defined as 
\[ \mathbb{V}_h = \{ v(x,y): v|_{I_i \times J_j} \in Q^r(I_i \times J_j), \quad i = 1,2,..N_x;\, j = 1,2,...N_y\}.  \]
For an arbitrary function $v \in \mathbb{V}_h$, we write $v_{i+1/2,y}^{\pm}=v(x_{i+1/2}^\pm,y)$ and $v_{x,j+1/2}^{\pm}=v(x,y_{j+1/2}^\pm)$ to denote its trace values at the cell boundaries along the $x$- and $y$-axes, respectively. The discontinuous jumps across these interfaces are then expressed as $[v]_{i+1/2, y} = v_{i+1/2,y}^+ - v_{i+1/2,y}^-$ for $y \in J_j$, and $[v]_{x,j+1/2} = v_{x,j+1/2}^+ - v_{x,j+1/2}^-$ for $x \in I_i$. We capture these discontinuities globally using the jump semi-norm, defined by $\| [u]\|_{\Gamma_h}^2 := \sum_{i,j} \left(\int_{J_j} [u]_{i+1/2,y}^2 \, \mathrm{d}y + \int_{I_i} [u]_{x,j+1/2}^2\, \mathrm{d}x \right)$. 

For the time discretization, we partition the interval $[0,T]$ using a sequence of points $0=t_0<t_1<\cdots<t_{N_T}=T$, with the local time step $k_n = t_{n}-t_{n-1}$. For simplicity of presentation, we will proceed under the assumption of a uniform time step $k_n \equiv k = T/N_T$, though all analysis remain valid for variable time steps. 
We also denote the Wiener increment by $\Delta W_n := W(t_{n+1})-W(t_n)$.

To present the LDG method for approximating the SPDE \eqref{spde:conv-diff-2d}, we first rewrite the model into the following first order system: 
\begin{gather}
\label{eq:auxiliary}
    \begin{cases}
        \mathrm{d}u = \left(w_{1,x} + w_{2,y} - f_1(u)_x - f_2(u)_y \right)\mathrm{d}t \\
        \qquad + \psi(\cdot,x,y,t,u,u_x,u_y) \mathrm{d}t + g(\cdot,x,y,t,u,u_x,u_y) \mathrm{d}W_t,  \\
        v_1(x,y,t) = u_x(x,y,t),    \\
        v_2(x,y,t) = u_y(x,y,t),    \\
        w_1(x,y,t) = a_{11}(x,y,t,u) v_1(x,y,t) + a_{12}(x,y,t,u) v_2(x,y,t),   \\
        w_2(x,y,t) = a_{21}(x,y,t,u) v_1(x,y,t) + a_{22}(x,y,t,u) v_2(x,y,t),  \\
        u(x,y,0) = u_0(x,y).
    \end{cases}
\end{gather}
We will omit the explicit dependence on $\omega$ whenever it is clear from the context. 
The fully-discrete scheme, with the LDG spatial discretization and the IMEX-Euler time discretization, for the above system is given by: for any $\omega \in \Omega$, given $u_h^n(\omega,\cdot) \in \mathbb{V}_h$, find $(u_h^{n+1}, v_{1,h}^{n+1}, v_{2,h}^{n+1}, w_{1,h}^{n+1},w_{2,h}^{n+1})(\omega,\cdot) \in (\mathbb{V}_h)^5$, such that for any test functions $(r_h,z_h, p_h,q_h,\phi_h) \in (\mathbb{V}_h)^5$, we have
\begin{align}
    (u_h^{n+1} - u_h^n,r_h)_{ij} &= kH_{i,j}^+(w_{1,h}^{n+1},w_{2,h}^{n+1},r_h) + k\tilde{H}_{i,j}(f_1,f_2,u_h^n,r_h) \notag \\
    &+ k(\psi(x,y,t_n,u_h^n,v_{1,h}^n,v_{2,h}^n),r_h)_{ij} + (g(x,y,t_n,u_h^n,v_{1,h}^n,v_{2,h}^n) \Delta W_n, r_h)_{ij} ,  \label{eq:fully-discrete-1-2d} \\
    (v_{1,h}^{n+1},p_h)_{ij} &= L_{i,j}^{x-}(u_h^{n+1},p_h), \quad (v_{2,h}^{n+1},q_h)_{ij} = L_{i,j}^{y-}(u_h^{n+1},q_h),  \label{eq:fully-discrete-2-2d} \\
    (w_{1,h}^{n+1},z_h)_{ij} &= (a_{11}(x,y,t_{n+1},u_h^{n+1})v_{1,h}^{n+1}, z_h)_{ij} + (a_{12}(x,y,t_{n+1},u_h^{n+1})v_{2,h}^{n+1}, z_h)_{ij},  \label{eq:fully-discrete-3-2d} \\
    (w_{2,h}^{n+1},\phi_h)_{ij} &= (a_{21}(x,y,t_{n+1},u_h^{n+1})v_{1,h}^{n+1}, \phi_h)_{ij} + (a_{22}(x,y,t_{n+1},u_h^{n+1})v_{2,h}^{n+1}, \phi_h)_{ij},  \label{eq:fully-discrete-4-2d}
\end{align}
where $(\cdot, \cdot)_{ij}$ denotes the $L^2$ inner product on $I_i \times J_j$, and  
\begin{align*}
    H_{i,j}^+(w_{1,h},w_{2,h},r_h) &= -\iint_{I_i \times J_j} w_{1,h}(r_h)_x + w_{2,h}(r_h)_y \, \mathrm{d}x \mathrm{d}y + \int_{J_j} ({w}_{1,h}^+ r_h^-)_{i+\frac{1}{2},y} - ({w}_{1,h}^+ r_h^+)_{i-\frac{1}{2},y} \, \mathrm{d}y \\
    &+ \int_{I_i} ({w}_{2,h}^+ r_h^-)_{x,j+\frac{1}{2}} - ({w}_{2,h}^+ r_h^+)_{x,j-\frac{1}{2}} \, \mathrm{d}x, \\
    \tilde{H}_{i,j}(f_1,f_2,u_h^n,r_h) &= \iint_{I_i \times J_j} f_1(u_h^n)(r_h)_x + f_2(u_h^n)(r_h)_y \, \mathrm{d}x \mathrm{d}y
    - \int_{J_j} (\hat{f}_{1}^n \,r_h^-)_{i+\frac{1}{2},y} - (\hat{f}_{1}^n \,r_h^+)_{i-\frac{1}{2},y} \, \mathrm{d}y \\
    &- \int_{I_i} (\hat{f}_{2}^n \,r_h^-)_{x,j+\frac{1}{2}} - (\hat{f}_{2}^n \,r_h^+)_{x,j-\frac{1}{2}} \, \mathrm{d}x, \\
    L_{i,j}^{x-}(u_h,p_h) &= -\iint_{I_i \times J_j} u_h(p_h)_x \, \mathrm{d}x \mathrm{d}y + \int_{J_j} ({u}_{h}^- p_h^-)_{i+\frac{1}{2},y} - ({u}_{h}^- p_h^+)_{i-\frac{1}{2},y} \, \mathrm{d}y, \\
    L_{i,j}^{y-}(u_h,q_h) &= -\iint_{I_i \times J_j} u_h(q_h)_y \, \mathrm{d}x \mathrm{d}y + \int_{I_i} ({u}_{h}^- q_h^-)_{x,j+\frac{1}{2}} - ({u}_{h}^- q_h^+)_{x,j-\frac{1}{2}} \, \mathrm{d}x.
\end{align*}
The numerical flux for the convection term takes the form $(\hat{f}_1^n)_{i+\frac{1}{2},y} = \hat{f}_1\left( u_h^n(x_{i+\frac{1}{2}}^-,y),u_h^n(x_{i+\frac{1}{2}}^+,y) \right)$, and $(\hat{f}_2^n)_{x,j+\frac{1}{2}} = \hat{f}_2 \left(u_h^n(x,y_{j+\frac{1}{2}}^-),u_h^n(x,y_{j+\frac{1}{2}}^+) \right)$, where $\hat{f}_1$ and $\hat{f}_2$ are monotone numerical fluxes associated with the physical fluxes $f_1$ and $f_2$. In this paper, we choose $\hat{f}$ to be the local Lax-Friedrich flux:
\begin{equation}
\label{eq:LF-flux}
    \hat{f}(u^-,u^+) = \frac{1}{2}\left(f(u^-) + f(u^+) \right) - \frac{1}{2} \left( \max |f'(u)| \right) [u].
\end{equation}
To discretize the diffusion term, we employ the alternating flux realized through the pair $H_{i,j}^+$ in \eqref{eq:fully-discrete-1-2d} and $L_{i,j}^{x-}, L_{i,j}^{y-}$ in \eqref{eq:fully-discrete-2-2d}. Other choices of numerical fluxes such as the generalized numerical fluxes studied in \cite{SunShuXing_JCP2022, SunXing_Mathcomp2021,CX2024} can also be applied.

For the initial condition $u_h^0=\mathcal{P}u_0$, we employ a sufficiently accurate approximation of the exact initial datum $u_0$, for example, the local Gauss-Radau projections $\mathcal{P}^-$ defined in Section \ref{subsec:prelim-2d}. The global form of the fully-discrete scheme can be obtained by summing \eqref{eq:fully-discrete-1-2d}-\eqref{eq:fully-discrete-4-2d} over all  $i,j$. For simplicity, we set $L_{i,j}^-(u_h,p_h,q_h) := L_{i,j}^{x-}(u_h,p_h) + L_{i,j}^{y-}(u_h,q_h)$, $H^+ := \sum_{i,j} H_{i,j}^+$, $\tilde{H} := \sum_{i,j} \tilde{H}_{i,j}$, $L^- := \sum_{i,j} L_{i,j}^-$. 

\end{subsection}

\begin{subsection}{Preliminaries}
\label{subsec:prelim-2d}    
In this subsection, we propose a few lemmas to be used in the stability and error analysis of the fully discrete scheme \eqref{eq:fully-discrete-1-2d}-\eqref{eq:fully-discrete-4-2d} in Sections \ref{sec:stability-2d} and \ref{sec:error-estimate-2d}. 

The following three elementary inequalities will be utilized in our numerical analysis, where \eqref{ineq:convex-2} follows from the mean value theorem, and \eqref{ineq:convex-1}, \eqref{ineq:convex-3} can be derived by applying the weighted Jensen's convexity inequality on $x^q$, $q \geq 1$. 
\begin{lemma}
\label{lem:generalized-discrete-Holder}
Let $a,b,c$ be three nonnegative real numbers. Given $q > 1$, for any $\epsilon >0$, the following inequalities hold
\begin{align}
    &(a+b)^q \leq C_1(\epsilon)a^q + (1+\epsilon)b^q,  \label{ineq:convex-1} \\
    & \left|a^q - b^q \right| \leq q|a-b|\max(a^{q-1}, \,b^{q-1}), \label{ineq:convex-2} \\
    &(a+b+c)^q \leq C_2(\epsilon)a^q + (2^{q-1}+\epsilon)(b^q + c^q),   \label{ineq:convex-3}
\end{align}
where $C_1(\epsilon) = \left(1-(1+\epsilon)^{\frac{1}{1-q}} \right)^{1-q}$ and $C_2(\epsilon) = \left(1-2 \left(2^{q-1}+\epsilon \right)^{\frac{1}{1-q}} \right)^{1-q}$. For $q=1$, the inequalities hold with $C_1(\epsilon)=C_2(\epsilon)=1$.
\end{lemma}

Let $\mathcal{P}: L^2(\mathcal{T}_h) \to \mathbb{V}_h$ be the standard $L^2$ projection onto the piecewise polynomial space. Additionally, we introduce two Gauss-Radau projections $\mathcal{P}^{\pm}: L^2(\mathcal{T}_h) \to \mathbb{V}_h$, defined as tensor products of the one-dimensional Gauss-Radau projections: $\mathcal{P}^{\pm} = \mathcal{P}_x^{\pm} \otimes \mathcal{P}_y^{\pm}$.
The subscripts indicate the variables with respect to which the one-dimensional projections are applied. 
The one-dimensional projections $\mathcal{P}_x^{\pm}$ are given as 
\begin{align}  
    &\int_{I_i} (\mathcal{P}_x^{-}u - u) v \, \mathrm{d}x = 0, \quad \forall v \in P^{r-1}(I_i), \quad \text{and} \quad (\mathcal{P}_x^{-}u)_{i+\frac{1}{2}}^- = u_{i+\frac{1}{2}}^- \,, \label{def:gauss-radau-minus-1d} \\
    &\int_{I_i} (\mathcal{P}_x^{+}u - u) v \, \mathrm{d}x = 0, \quad \forall v \in P^{r-1}(I_i), \quad \text{and} \quad (\mathcal{P}_x^{+}u)_{i-\frac{1}{2}}^+ = u_{i-\frac{1}{2}}^+, \label{def:gauss-radau-plus-1d} 
\end{align}
for $i=1,2,...,N_x$. The projections $\mathcal{P}_y^\pm$ are defined similarly.  
We refer the reader to \cite{chen_semi-linear_2d,MengShuWu_MathComp2015} for the explicit definitions of $\mathcal{P}^{\pm}$. 
The approximation result for the projection error is given by the following Lemma. 
\begin{lemma}[{Lemma 3.2 in \cite{Cockburn_SIAMNu2001}}]
\label{lem:proj-property-2d}
Let $K = I_i \times J_j$, $\Pi = \mathcal{P}$ or $\mathcal{P}^{\pm}$. For any $u \in H^{r+1}(K)$ with $\eta = \Pi u - u$,
\begin{align}
\label{ineq:proj-property-1-2d}
\| \eta\|_{L^2(K)} + h^{\frac{1}{2}}\| \eta\|_{L^2( \partial K)} + h\| \eta\|_{H^1(K)} \leq Ch^{r+1}\|u\|_{H^{r+1}(K)},
\end{align}
where $\| v\|_{L^2(\partial K)}^2 := \| v(\cdot,y_{j-\frac{1}{2}}^+)\|_{I_i}^2 + \| v(\cdot,y_{j+\frac{1}{2}}^-)\|_{I_i}^2 + \| v(x_{i-\frac{1}{2}}^+, \cdot)\|_{J_j}^2 + \| v(x_{i+\frac{1}{2}}^-, \cdot)\|_{J_j}^2$ is the $L^2$ norm on the boundary.
Moreover, when $r \geq 1$, we have $\| \eta\|_{L^{\infty}(K)} \leq Ch^{r}\|u\|_{H^{r+1}(K)}$. 
\end{lemma}

For spatial dimensions above one, the terms involving the projection error in \eqref{ineq:superconvergence-1}, \eqref{ineq:superconvergence-2} do not vanish in general. We thus make use of the following super-convergence property given below. For details, we refer to \cite{ChengMengZhang_MathComp2016,MengShuWu_MathComp2015}. 
\begin{lemma}
\label{lem:superconvergence}
Let $(w,w_1,w_2) \in \mathbb{V}_h^3$, 
for any $u,v \in H^{r+2}(\mathcal{D})$, we have
\begin{align}
    &|L^-(\mathcal{P}^-u - u, w_1,w_2)| \leq Ch^{r+1} \| u\|_{{r+2}}(\| w_1\|+\| w_2\|), \label{ineq:superconvergence-1} \\
    &|H^+(\mathcal{P}^+ u - u, \mathcal{P}^+ v - v, w)| \leq Ch^{r+1} (\| u\|_{{r+2}} + \| v\|_{{r+2}}) \| w\|.  \label{ineq:superconvergence-2}
\end{align}
\end{lemma}

For the operators $H^+$, $L^-$ and $\tilde{H}$, we summarize the following properties of the LDG spatial discretization, the alternating flux and monotone flux. The proof directly follows from the one-dimensional case in \cite{WangShuZhang_SIAMNu2015, WangShuZhang2016AMC} and is omitted. 
\begin{lemma}
\label{lem:num-flux-2D}
For any $(v,w_1,w_2) \in \mathbb{V}_h^3$, we have 
\begin{align}
    & |L^-(v,w_1,w_2)| \leq \left(\| \nabla v\| + \mu h^{-\frac{1}{2}} \| [v]\|_{\Gamma_h} \right) \left(\| w_1\| + \| w_2\| \right), \label{ineq:alternating-flux-2d} \\
    & H^+(w_1,w_2,v) + L^-(v,w_1,w_2) = 0,  \quad  \tilde{H}(f_1,f_2,v,v) \leq 0, \notag
\end{align}
where $\mu$ is the inverse inequality constant. 
\end{lemma}

Next, we present a lemma that establishes an important relationship between the gradient of the numerical solution, and the LDG approximation of the gradient. It was first proved in \cite{WangShuZhang_SIAMNu2015} for the one-dimensional case and extended to $P^r$ rectangular and triangular elements in \cite{WangWangZhangShu_ESAIM2016}. We provide a proof in Appendix \ref{appendix-num-sol-2d} for $Q^r$ elements considered in this paper. Note that the result in Lemma \ref{lem:num-sol-2D} is also true for solutions of the semi-discrete LDG scheme. 
\begin{lemma}
\label{lem:num-sol-2D}
    Suppose $(u_h, v_{1,h}, v_{2,h}) \in \mathbb{V}_h^3$ are solutions of the fully-discrete scheme \eqref{eq:fully-discrete-2-2d}, then 
    \begin{align*}
        &\| \nabla u_h\| + \mu h^{-\frac{1}{2}} \| [u_h]\|_{\Gamma_h} \leq C_{\mu} (\| v_{1,h}\| + \| v_{2,h}\|), \\
        &\| \nabla \xi_u\| + \mu h^{-\frac{1}{2}}\|[\xi_u]\|_{\Gamma_h} \leq C_{\mu} (\| \xi_{v_1}\| + \| \xi_{v_2}\| +h^{r+1}\|u\|_{r+2}),
    \end{align*}
    where $C_{\mu}$ only depends on the inverse constant $\mu$. Here, $\xi_u=\mathcal{P}^-u-u_h$ and  $\xi_{v_i}=\mathcal{P}v_i-v_{i,h}$
    , $i=1,2$ are the errors to be studied in Section \ref{sec:error-estimate-2d}.
\end{lemma}

\begin{lemma}
\label{lem:nonlinear-convec-2D}
    Let 
    $$D(f_1,f_2;u_h,z_h,r_h) := \tilde{H}(f_1,f_2,u_h,r_h) - \tilde{H}(f_1,f_2,z_h,r_h).$$ 
    We assume $f_1, f_2$ satisfy hypothesis (iii). For any $(u_h,z_h,r_h) \in \mathbb{V}_h^3$, the following inequalities hold
    \begin{align*}
        &|\tilde{H}(f_1,f_2,u_h,r_h)| \leq C(1+\| u_h\|_{\infty}^p) \left(\| \nabla u_h\| + {\mu} h^{-\frac{1}{2}}\|[u_h]\|_{\Gamma_h} \right) \| r_h\|, \\
        &|D(f_1,f_2;u_h,z_h,r_h)| \leq C_f \| u_h-z_h\| \left(\| \nabla r_h\| + {\mu} h^{-\frac{1}{2}}\|[r_h]\|_{\Gamma_h} \right),
    \end{align*}
    where $C_f = C(1 + \| u_h\|_{\infty}^p +\| z_h\|_{\infty}^p)$. 
\end{lemma}
Lemma \ref{lem:nonlinear-convec-2D} estimates the nonlinear convection term in the numerical scheme, and its proof is provided in Appendix \ref{appendix-nonlinear-convec-2d}. The result is similar to \cite[Lemma 2.5]{WangShuZhang2016AMC}, while we do not assume that $f'$ is globally bounded here. As a result of Lemmas \ref{lem:num-sol-2D} and \ref{lem:nonlinear-convec-2D}, we get the following corollary. 
\begin{corollary}
\label{coro:nonlinear-convec-2D}
For fully discrete numerical solutions $(u_h^n, u_h^{n+1},v_{1,h}^n,v_{2,h}^n,v_{1,h}^{n+1},v_{2,h}^{n+1})$ given by the scheme \eqref{eq:fully-discrete-1-2d}-\eqref{eq:fully-discrete-4-2d}, we have
\begin{align*}
    &|\tilde{H}(f_1,f_2,u_h^{n},u_h^{n+1}-u_h^n)| 
    \leq C_{\mu}^{*} (1+\| u_h^{n}\|_{\infty}^p) \left( \| v_{1,h}^{n}\|+\| v_{2,h}^{n}\| \right) \| u_h^{n+1}-u_h^n\|, \\
    &|D(f_1,f_2;u(t),u_h^{n},\xi_u^{n+1})| 
    \leq C_{f_h}\| u(t)-u_h^{n}\| \left(\| \xi_{v_1}^{n+1}\| + \| \xi_{v_2}^{n+1}\| +h^{r+1}\|u^{n+1}\|_{r+2} \right),
\end{align*}
where $C_{\mu}^{*}$ is a fixed constant that only depends on $\mu$, and $C_{f_h} = C\left(1 +\| u(t)\|_{\infty}^p +\| u_h^n\|_{\infty}^p \right)$. 
\end{corollary}

\medskip
In what follows, we introduce three lemmas tailored to the stochastic components of our numerical analysis. The discrete Kolmogorov lemma below helps to derive pathwise error estimates. It is a discrete analog to the classic Kolmogorov continuity theorem \cite{DaPrato_Zabczyk_1992}, which provides the conditions ensuring pathwise regularity (H\"older continuity) of a continuous stochastic process. The proof is provided in \cite{chen_semi-linear_2d}. Let $\{ X(t)\}_{t \in [0,T]}$ be a stochastic process taking values in a separable Banach space with norm $\| \cdot\|$ and $0 = t_0 < t_1 < \cdots < t_N = T$ be a uniform partition of the interval $[0,T]$, with $k = {T}/{N}$.

\begin{lemma}[Discrete Kolmogorov Theorem]
\label{lem:Kolmogorov}
     Suppose there is a discrete stochastic process $\{X_n^N \}_{n=0}^N$ such that there exist constants $\nu,\beta >0$, $C>0$, and a positive integer $N_0$ satisfying for every $N \geq N_0$,
    \[ \mathbb{E}\left[ \max_{0 \leq n \leq N} \| X(t_n)-X_n^N\|^{\nu} \right] \leq C \left(\frac{1}{N} \right)^{1+\beta}.   \]
    Then, for any $0 \leq \gamma < {\beta}/{\nu}$, there exists a random variable $Z(\omega,\gamma)$ with $\mathbb{E}[|Z|^{\nu}] < \infty$, such that 
    \begin{align}
    \label{ineq:kolmogorov}
        \max_{0 \leq n \leq N} \|X(t_n,\omega)-X_n^N(\omega)\| \leq Z(\omega,\gamma) \left(\frac{1}{N} \right)^{\gamma}, \qquad \forall N \geq N_0.
    \end{align} 
\end{lemma}

The classical Burkholder-Davis-Gundy inequality bounds the supremum of a local martingale by its quadratic variation. Both the continuous and discrete versions are utilized in the numerical analysis in Sections \ref{sec:stability-2d} and \ref{sec:error-estimate-2d}. We provide a discrete version \cite{ondrejat_Prohl_Walkington_2023} below. 
\begin{lemma}[Discrete Burkholder-Davis-Gundy]
\label{lem:BDG}
    Let $(\Omega, \mathcal{F}, \mathbb{P})$ be a probability space with (discrete) filtration $\{\mathcal{F}^n\}_{n=0}^{N}$. Let $\{X^n\}_{n=0}^{N}$ with $X^0 \equiv 0$ be a $\{\mathcal{F}^n\}$-martingale taking values in a separable Hilbert space $H$. Then for each $p \geq 1$ there exist constants $0 < 1/C_b' < C_b$ such that
    \begin{align*}
        \frac{1}{C_b'} \mathbb{E} \left[ \left( \sum_{n=1}^{N} \| X^n - X^{n-1} \|_H^2 \right)^{p/2} \right] &\leq \mathbb{E} \left[ \max_{0 \leq n \leq N} \| X^n \|_H^p \right] 
        \leq C_b \mathbb{E} \left[ \left( \sum_{n=1}^{N} \| X^n - X^{n-1} \|_H^2 \right)^{p/2} \right]. 
    \end{align*}
\end{lemma}
\begin{remark}[Continuous and Discrete BDG Inequalities for Discrete Martingales as It\^o Integrals]
\label{rmk:BDG}
    In our fully-discrete numerical analysis, the relevant discrete martingales are It\^o integrals of the form $I_m = \sum_{n=0}^m G^n \Delta W_n$, where $G^n$ depends on the numerical solutions at time step $t_n$. Then as a result of the continuous and discrete BDG inequalities, for $p \geq 1$, $N \leq N_T$ ,
    \begin{align*}
        \frac{1}{C_b'}\mathbb{E}\left[ \left( \sum_{n=0}^{N} \| G^n \Delta W_n \|^2 \right)^{p/2} \right] \leq \mathbb{E} \left[ \max_{0 \leq m \leq N} \left\| \sum_{n=0}^m G^n \Delta W_n  \right\|^p \right] \leq C_b \mathbb{E}\left[ \left( kK\sum_{n=0}^{N} \| G^n \|^2 \right)^{p/2} \right],
    \end{align*}
    where the left inequality follows from Lemma \ref{lem:BDG} and the right inequality is derived from the continuous BDG inequality by viewing the discrete martingale $I_m$ as an It\^o integral using piecewise constant interpolation. 
\end{remark}

Lastly, we show a high-moment temporal H\"older continuity for the strong solution $u$ with respect to the spatial $H^{m}$ norm. The proof is similar to that in \cite{chen_semi-linear_2d} and is omitted here. Note that Lemma \ref{lem:Holder-High-Moments} and Corollary \ref{coro:Holder-High-Moments} remain valid in higher spatial dimensions.

\begin{lemma}[Higher Moment H\"older Continuity Estimate]
\label{lem:Holder-High-Moments}
    Let $u$ be the strong solution of \eqref{spde:conv-diff-2d}, and let $q \geq 1$. Suppose $u \in L^{2q}\left(\Omega,L^{\infty}[0,T; H^{m+1}] \right)$, $w_1,w_2,f_1(u), f_2(u) \in L^{2q}\left(\Omega \times [0,T]; H^{m+1} \right)$, $\psi(\cdot,u,\nabla u) \in L^{2q}\left(\Omega \times [0,T]; H^{m} \right)$, where $m$ is a nonnegative integer. In addition, we assume
    \begin{equation}\label{Holder-High-Moments-assumption}
        \| g(\omega,\cdot,t,u,\nabla u)\|_{\mathcal{L}_2^m}^2 \leq C\left(1 + \| u\|_m^{2} + \| \nabla u\|_m^{2} \right) \quad \forall (w,t) \in \Omega \times [0,T],
    \end{equation} 
    then 
    \[ \mathbb{E}\left[\| u(t)-u(s)\|_{m}^{2q} \right] \leq C(t-s)^q \quad \forall \, 0 \leq s < t \leq T. \]
\end{lemma}

As an immediate consequence, we have the following corollary. 
\begin{corollary}
\label{coro:Holder-High-Moments}
    Let $u$ be the strong solution of \eqref{spde:conv-diff-2d}, let $q \geq 1$ and $m$ be a nonnegative integer. If $w_1,w_2,f_1(u), f_2(u) \in L^{2q}(\Omega \times [0,T]; H^{m+1})$, $\psi(\cdot,u,\nabla u) \in L^{2q}(\Omega \times [0,T]; H^{m})$, and if $g(\cdot,u,\nabla u) \in L^{2q}(\Omega \times [0,T]; \mathcal{L}_2^m)$ satisfies \eqref{Holder-High-Moments-assumption}, then for any uniform partition $t_n = nT/N_T$, $n = 0,1,...N_T$, there exists $C$ independent of $k$, such that
    \[ \frac{1}{k^{q-1}}\sum_{n=0}^{N_T-1} \mathbb{E}\left[ \| u(t_{n+1}) - u(t_n)\|_m^{2q} \right] \leq C. \]
\end{corollary}

\end{subsection}

\end{section}

\begin{section}{High Moment Stability Estimate}
\label{sec:stability-2d}
In this section, we establish second- and higher-moment stability estimates for the numerical solution $u_h$ of the scheme \eqref{eq:fully-discrete-1-2d}-\eqref{eq:fully-discrete-4-2d} on subsets of the sample space. 

For numerical analysis of nonlinear stochastic PDEs, direct moment estimates over the full probability space are often obstructed by random growth factors arising from nonglobally Lipschitz nonlinearities. A common way to address this difficulty is through the localization approach. The main idea is to derive estimates on events where the exact or numerical solutions remain bounded and then prove that the probability of the complementary events goes to zero. This approach was studied in the convergence-in-probability analysis of locally Lipschitz stochastic parabolic equations in \cite{Printems_M2AN2001}, and later has been developed for stochastic Navier--Stokes equations \cite{Carelli_Prohl_SIAMNu2012, BessaihBrzezniakMillet_SPDAC2014, BreitDodgson_NumMath2021} and for nonlinear stochastic wave equations \cite{Li_Wu_Xing_2022}.

In the literature, two forms of localization methods are adopted to derive convergence in probability or localized moment estimations. 
The first class is stopping-time localization, where one defines a stopping time and analyzes the stopped processes. 
Discrete stopping times were developed in \cite{BreitProhl_FoCM2024} to preserve the adaptedness required by stochastic energy and martingale arguments.
The other class is known as the sample-space localization, where one defines a high-probability event $\Omega_R$ on which the discrete trajectory satisfies a prescribed bound on the exact solution or the numerical solution.
The sequence of events is constructed so that their probability converges to one under an appropriate relation between the discretization parameters.
Discretization-dependent events are utilized in \cite{Carelli_Prohl_SIAMNu2012, BessaihBrzezniakMillet_SPDAC2014, Vo_Boussinesq_Multiplicative2025} to allow discrete Gr\"onwall type arguments, and \cite{BessaihMillet_IMAJNA2019} used events depending on the exact solution and combined localization with exponential moment estimates to obtain convergence in $L^2(\Omega)$. 
For stochastic Navier--Stokes equations with additive noise, high-moment strong and pathwise error estimates have been established in \cite{FengVo_CiCP2024}. In the presence of multiplicative noise, its interaction with nonlinear terms may produce random growth factors that generally prevent a direct application of deterministic Gr\"onwall arguments. To the best of our knowledge, optimal high-moment strong error estimates over the full sample space have not been established for fully discrete high-order DG approximations of multidimensional equations that simultaneously involve solution-dependent diffusion, nonlinear convection and source terms, and gradient-dependent multiplicative noise.

In our IMEX-LDG scheme, the nonlinear convection and source terms are treated explicitly. 
To control the potentially trajectory-dependent unbounded growth of $f(u)$ and $u^{\lambda}$, we introduce the following recursively defined, decreasing family of localized events, parameterized by the time step index $\ell$ and a fixed positive number $\delta$: 
\begin{align}
\label{def:stability-subset}
    &\tilde{\Omega}_{\delta,-1} = \Omega  \notag \\
    &\tilde{\Omega}_{\delta,\ell} = \left\{ \omega \in \Omega: \max_{0 \leq n \leq \ell} \mathbf{1}_{\tilde{\Omega}_{\delta,n-1}} \left( 1 + \| u_h^{n}\|_{\infty}^{2p_0} \right) \leq \delta/k \right\},  \qquad  0 \leq \ell \leq N_T,
\end{align}
where $p_0:=\max \{p,\lambda-1\}$ with $p$ and $\lambda$ given in hypotheses (iii) and (iv), respectively. The recursive definition will play an important role in our proofs, as explained in Remark \ref{rmk:recursive-def-subset}. 
Clearly, it holds that 
$$\tilde{\Omega}_{\delta,N_T} \subset \tilde{\Omega}_{\delta,N_T-1} \subset \cdot \cdot \cdot \subset \tilde{\Omega}_{\delta,0} \subset \tilde{\Omega}_{\delta,-1}=\Omega.$$ 



The following lemma shows that the initial discrete gradients $v_{1,h}^0,v_{2,h}^0$ are well-defined and bounded. The proof is the same as \cite[Lemma 3.1]{chen_semi-linear_2d} and is omitted.
\begin{lemma}
\label{lem:v_h^0-2d}
    Assume hypothesis {\rm(i)}. Let $u_h^0$ be a projection of $u_0$ satisfying the standard approximation property \eqref{ineq:proj-property-1-2d}. Then $v_{1,h}^0,v_{2,h}^0$ computed from \eqref{eq:fully-discrete-2-2d} are deterministic, and satisfy $\| v_{1,h}^0\| + \| v_{2,h}^0\| < C\| u_0\|_{1}$. 
\end{lemma}

The high-moment stability result for the fully-discrete scheme is as follows. For a fixed $q\in[1,\infty)$, suppose the parabolic coeffient satisfies that
$$\alpha>\alpha_0(q) := \left(2^{2q-1}C_b'+2^{5q-4}C_b\right)^{1/q} C_b^{1/q}D_4^2K.$$
\begin{theorem}
\label{thm:stability-estimate-2D}
Assume hypotheses {\rm(i)}-{\rm(v)}. Then there exists a constant $C$, independent of $h,k$ such that 
\begin{align}
\label{ineq:high-moment-stability-2D}
    \mathbb{E}\left[ \max_{0\leq n \leq N_T} \mathbf{1}_{\tilde{\Omega}_{\delta,n-1}}\| u_h^{n}\|^{2q} \right] &+ \mathbb{E}\left[\left(\sum_{n=0}^{N_T-1} \mathbf{1}_{\tilde{\Omega}_{\delta,n}} \| u_h^{n+1} - u_h^n\|^2 \right)^q \right] + \mathbb{E}\left[ \left(k\sum_{n=0}^{N_T}\mathbf{1}_{\tilde{\Omega}_{\delta,n-1}} (\| v_{1,h}^{n}\|^2 + \| v_{2,h}^{n}\|^2) \right)^q \right] \notag \\
    &\leq C\left(1+\| u_0\|^{2q} + k^q \|u_0\|_{1}^{2q} \right).
\end{align}
Moreover, if $\alpha > \max\{ \alpha_0(q), \alpha_0(p_0)\}$, there exists $\delta>0$ such that $\mathbb{P}(\tilde{\Omega}_{\delta,N_T}) \to 1$ as $k/h^{2p_0} \to 0$.
\end{theorem}

\begin{proof}
Since the proof is long, we divide the process into three steps.

\noindent {\it Step 1.} Let the test functions be $r_h = u_h^{n+1}$, $p_h = w_{1,h}^{n+1}$, $q_h = w_{2,h}^{n+1}$ and $z_h = v_{1,h}^{n+1}$, $\phi_h = v_{2,h}^{n+1}$ in the fully discrete scheme \eqref{eq:fully-discrete-1-2d}-\eqref{eq:fully-discrete-4-2d}. Let $A(\cdot,t_{n+1},u_h^{n+1}) = \{ a_{11}, a_{12}; a_{21}, a_{22}\}$ be the $ 2\times2$ coefficient matrix and $\mathbf{v}_h^{n+1} = (v_{1,h}^{n+1},v_{2,h}^{n+1})^{\top}$. After summing over all computational cells and applying Lemma \ref{lem:num-flux-2D}, we have
\begin{align}
\label{eq:fully-discrete-5-2d}
    \bigl(A(\cdot,t_{n+1},u_h^{n+1})\mathbf{v}_h^{n+1}, \mathbf{v}_h^{n+1} \bigr) = L^-(u_h^{n+1},w_{1,h}^{n+1},w_{2,h}^{n+1}) = -H^+(w_{1,h}^{n+1},w_{2,h}^{n+1},u_h^{n+1}).
\end{align}
Multiplying \eqref{eq:fully-discrete-5-2d} by $k$, adding it to \eqref{eq:fully-discrete-1-2d}, and multiplying the resulting identity by $\mathbf{1}_{\tilde{\Omega}_{\delta,n}}$, we obtain
\begin{align}
\label{eq:stability-main-2d}
    \mathbf{1}_{\tilde{\Omega}_{\delta,n}}(u_h^{n+1} - u_h^n,u_h^{n+1}) 
    &+k\mathbf{1}_{\tilde{\Omega}_{\delta,n}}\bigl(A(\cdot,t_{n+1},u_h^{n+1})\mathbf{v}_h^{n+1}, \mathbf{v}_h^{n+1} \bigr) \notag \\
    &= \mathbf{1}_{\tilde{\Omega}_{\delta,n}}k\tilde{H}(f_1,f_2,u_h^n,u_h^{n+1}) 
    + \mathbf{1}_{\tilde{\Omega}_{\delta,n}} k(\psi(\cdot,t,u_h^n,v_{1,h}^n,v_{2,h}^n),u_h^{n+1}) \notag \\
    &\quad + \mathbf{1}_{\tilde{\Omega}_{\delta,n}}(g(\cdot,t,u_h^n,v_{1,h}^n,v_{2,h}^n) \Delta W_n, u_h^{n+1}). 
\end{align}
By hypothesis (ii) and the nesting property of $\{ \tilde{\Omega}_{\delta,n}\}$, the left-hand side of \eqref{eq:stability-main-2d} satisfies
\begin{align*}
    LHS \geq \frac{1}{2}\mathbf{1}_{\tilde{\Omega}_{\delta,n}}\| u_h^{n+1}\|^2 - \frac{1}{2} \mathbf{1}_{\tilde{\Omega}_{\delta,n-1}} \|u_h^n\|^2 + \frac{1}{2} \mathbf{1}_{\tilde{\Omega}_{\delta,n}} \| u_h^{n+1} - u_h^n\|^2 + \alpha k \mathbf{1}_{\tilde{\Omega}_{\delta,n}} \left(\| v_{1,h}^{n+1}\|^2 +\| v_{2,h}^{n+1}\|^2 \right).
\end{align*}

Using Lemma \ref{lem:num-flux-2D} and the oddness of $\lambda$, we have
\[
    \tilde H(f_1,f_2,u_h^n,u_h^n)\leq0,
    \qquad
    (-(u_h^n)^\lambda,u_h^n)\leq0.
\]
Therefore, for every $0\leq l\leq m$, taking the sum $\sum_{n=0}^l$ on both sides of \eqref{eq:stability-main-2d} yields
\begin{align}
\label{ineq:stability-energy-l}
    &\frac12 \mathbf{1}_{\tilde{\Omega}_{\delta,l}} \|u_h^{l+1}\|^2
    + \frac12 \sum_{n=0}^l \mathbf{1}_{\tilde{\Omega}_{\delta,n}} \|u_h^{n+1}-u_h^n\|^2 
    + \alpha k \sum_{n=0}^l \mathbf{1}_{\tilde{\Omega}_{\delta,n}} \left( \|v_{1,h}^{n+1}\|^2+\|v_{2,h}^{n+1}\|^2 \right) \notag \\
    &\leq 
    \frac12\|u_h^0\|^2 + k \sum_{n=0}^{l} \mathbf{1}_{\tilde{\Omega}_{\delta,n}} \tilde{H}(f_1,f_2,u_h^{n},u_h^{n+1}-u_h^n) 
    + k \sum_{n=0}^{l} \mathbf{1}_{\tilde{\Omega}_{\delta,n}} \left(-(u_h^n)^{\lambda}, u_h^{n+1}-u_h^n \right) \notag\\
    &+ k \sum_{n=0}^{l} \mathbf{1}_{\tilde{\Omega}_{\delta,n}} (\psi_1(\cdot,t_n,u_h^n,v_{1,h}^n,v_{2,h}^n), u_h^{n+1}) 
     + \sum_{n=0}^{l} \mathbf{1}_{\tilde{\Omega}_{\delta,n}} (g_h^n \Delta W_n,u_h^{n+1}-u_h^n) 
     + \sum_{n=0}^{l} \mathbf{1}_{\tilde{\Omega}_{\delta,n}} (g_h^n \Delta W_n,u_h^{n}) \notag \\    
    &:= \frac12\|u_h^0\|^2+I_1(l)+I_2(l)+I_3(l)+I_4(l)+M_l,
\end{align}
with 
$$g_h^n:=g(\cdot,t_n,u_h^n,v_{1,h}^n,v_{2,h}^n).$$
By Corollary \ref{coro:nonlinear-convec-2D}, hypothesis (iv), Young's inequality and Cauchy-Schwarz inequality, we have
\begin{align*}
    \max_{0\leq l\leq m}|I_1(l)| & 
    \leq kC_{\mu}^{*} \sum_{n=0}^{m} \mathbf{1}_{\tilde{\Omega}_{\delta,n}} (1 + \| u_h^{n}\|_{\infty}^p) (\| v_{1,h}^{n}\| + \| v_{2,h}^{n}\|) \| u_h^{n+1}-u_h^n\| \\
    &\leq \frac{1}{2\epsilon_0} \sum_{n=0}^{m} \mathbf{1}_{\tilde{\Omega}_{\delta,n}} \| u_h^{n+1}-u_h^n\|^2 + Ck(\| v_{1,h}^{0}\|^2 + \| v_{2,h}^{0}\|^2)  \\
    & \qquad + 2(C_{\mu}^{*})^2 \epsilon_0 \delta k\sum_{n=0}^{m} \mathbf{1}_{\tilde{\Omega}_{\delta,n}} (\| v_{1,h}^{n+1}\|^2 + \| v_{2,h}^{n+1}\|^2), \\
    \max_{0\leq l\leq m}|I_2(l)| & 
    \leq \frac{\epsilon_1}{2} \sum_{n=0}^{m} \mathbf{1}_{\tilde{\Omega}_{\delta,n}} \| u_h^{n+1} - u_h^n\|^2 
    + \frac{k^2}{2\epsilon_1} \sum_{n=0}^{m} \mathbf{1}_{\tilde{\Omega}_{\delta,n}} \| (u_h^n)^{\lambda}\|^2 \\
    &\leq \frac{\epsilon_1}{2} \sum_{n=0}^{m} \mathbf{1}_{\tilde{\Omega}_{\delta,n}} \| u_h^{n+1} - u_h^n\|^2 
    + \frac{k \delta}{2\epsilon_1} \sum_{n=0}^{m} \mathbf{1}_{\tilde{\Omega}_{\delta,n}} \|u_h^n\|^2, \\
    \max_{0\leq l\leq m}|I_3(l)| &\leq \frac{k}{2\epsilon_2}\sum_{n=0}^{m} \mathbf{1}_{\tilde{\Omega}_{\delta,n}} \| u_h^{n+1}\|^2 
    + \frac{k\epsilon_2}{2} \sum_{n=0}^{m} \mathbf{1}_{\tilde{\Omega}_{\delta,n}} \| \psi_1(\cdot,t_n,u_h^n,v_{1,h}^n,v_{2,h}^n)\|^2 \\
    &\leq Ck \sum_{n=0}^{m} \mathbf{1}_{\tilde{\Omega}_{\delta,n}} \| u_h^{n+1}\|^2 + Ck(1+\| u_h^{0}\|^2 + \| v_{1,h}^{0}\|^2 + \| v_{2,h}^{0}\|^2)  \\
    & \qquad +  \frac{k\epsilon_2}{2} B_3^2\sum_{n=0}^{m} \mathbf{1}_{\tilde{\Omega}_{\delta,n}} (\| v_{1,h}^{n+1}\|^2 + \| v_{2,h}^{n+1}\|^2), \\
    \max_{0\leq l\leq m}|I_4(l)| &\leq \frac{1}{2\epsilon_3} \sum_{n=0}^{m} \mathbf{1}_{\tilde{\Omega}_{\delta,n}} \| u_h^{n+1} - u_h^n\|^2 
    + \frac{\epsilon_3}{2}\sum_{n=0}^{m} \mathbf{1}_{\tilde{\Omega}_{\delta,n}} \| g_h^n \Delta W_n\|^2,
\end{align*}
where the last inequality in the estimation of $I_1$ and $I_2$ follows from the definition \eqref{def:stability-subset} and the nesting property is used to shift the terms involving $v_{i,h}^n$ to terms involving $v_{i,h}^{n+1}$.

Taking $\max_{0 \leq l \leq m}$ on both sides of \eqref{ineq:stability-energy-l} leads to
\begin{align}
\label{ineq:stability-energy-l2}
    &\max \left\{ \frac12 \max_{0 \leq l \leq m} \mathbf{1}_{\tilde{\Omega}_{\delta,l}} \|u_h^{l+1}\|^2, \,
    \frac{1}{2} \sum_{n=0}^{m} \mathbf{1}_{\tilde{\Omega}_{\delta,n}} \| u_h^{n+1} - u_h^n\|^2 
    + \alpha k \sum_{n=0}^{m} \mathbf{1}_{\tilde{\Omega}_{\delta,n}} \left(\| v_{1,h}^{n+1}\|^2 +\| v_{2,h}^{n+1}\|^2 \right) \right\} \notag \\
    &\qquad \leq \frac12\|u_h^0\|^2+\max_{0\leq l\leq m}I_1(l)+\max_{0\leq l\leq m}I_2(l)+\max_{0\leq l\leq m}I_3(l)+\max_{0\leq l\leq m}I_4(l)+\max_{0\leq l\leq m}M_l,
\end{align}
Combining this with the estimates of $I_1$ to $I_4$ leads to 
\begin{align} 
    \frac{1}{2} &\max_{0 \leq n \leq m} \mathbf{1}_{\tilde{\Omega}_{\delta,n}}\| u_h^{n+1}\|^2 + \left(\frac{1}{2}-\frac{1}{\epsilon_0} - \frac{1}{\epsilon_3} -\epsilon_1 \right) \sum_{n=0}^{m} \mathbf{1}_{\tilde{\Omega}_{\delta,n}} \| u_h^{n+1} - u_h^n\|^2 \\
    &\,\, + (\alpha - 4(C_{\mu}^{*})^2 \epsilon_0 \delta - \epsilon_2 B_3^2)k \sum_{n=0}^{m} \mathbf{1}_{\tilde{\Omega}_{\delta,n}} \left(\| v_{1,h}^{n+1}\|^2 +\| v_{2,h}^{n+1}\|^2 \right) \notag \\
    &\leq C( 1+ \| u_h^{0}\|^2 + k\| v_{1,h}^{0}\|^2 + k\| v_{2,h}^{0}\|^2) \notag \\
    &\quad + Ck \sum_{n=0}^{m} \mathbf{1}_{\tilde{\Omega}_{\delta,n}} \| u_h^{n+1}\|^2 + \epsilon_3\sum_{n=0}^{m} \mathbf{1}_{\tilde{\Omega}_{\delta,n}} \| g_h^n \Delta W_n\|^2
    + 2\max_{0 \leq l \leq m} |M_l|.  \label{ineq:stability-main-2D}
\end{align}
Taking the $q$-th power of both sides, and then taking the expectation yield
\begin{align}
\label{ineq:stability-main-expect-2D}
    \frac{1}{2^q} & \mathbb{E}\left[\max_{0 \leq l \leq m} \mathbf{1}_{\tilde{\Omega}_{\delta,l}} \| u_h^{l+1}\|^{2q} \right] + \left(\frac{1}{2}-\frac{1}{\epsilon_0}-\frac{1}{\epsilon_3} - \epsilon_1\right)^q \mathbb{E}\left[ \left(\sum_{n=0}^{m} \mathbf{1}_{\tilde{\Omega}_{\delta,n}} \| u_h^{n+1} - u_h^n\|^2 \right)^q \right] \notag \\
    & \quad + \left(\alpha - 4(C_{\mu}^{*})^2\epsilon_0 \delta - \epsilon_2 B_3^2 \right)^q \mathbb{E}\left[ \left(k\sum_{n=0}^{m} \mathbf{1}_{\tilde{\Omega}_{\delta,n}} \left(\| v_{1,h}^{n+1}\|^2 +\| v_{2,h}^{n+1}\|^2 \right) \right)^q \right] \notag \\
    & \leq C(1 + \| u_h^{0}\|^{2q} + k^q \| v_{1,h}^{0}\|^{2q} + k^q \| v_{2,h}^{0}\|^{2q}) + Ck\sum_{n=0}^{m} \mathbb{E}\left[ \max_{0 \leq l \leq n}\mathbf{1}_{\tilde{\Omega}_{\delta,l}} \| u_h^{l+1}\|^{2q} \right]  \ \\
    &\quad + (2^{q-1}+\epsilon_4)\epsilon_3^q \mathbb{E}\left[ \left(\sum_{n=0}^{m} \mathbf{1}_{\tilde{\Omega}_{\delta,n}} \| g_h^n \Delta W_n\|^2 \right)^q \right] 
     + (2^{q-1}+\epsilon_4) 2^q\mathbb{E}\left[ \max_{0 \leq l \leq m} \left| \sum_{n=0}^{l} \mathbf{1}_{\tilde{\Omega}_{\delta,n}} (g_h^n \Delta W_n,u_h^{n}) \right|^q \right], \notag
\end{align}
where we use \eqref{ineq:convex-3} with $\epsilon = \epsilon_4$, and the discrete H\"older inequality.

\noindent {\it Step 2.} Next, we proceed to estimate the last two terms of the right-hand-side of \eqref{ineq:stability-main-expect-2D}. 
We first note that, by the recursive definition of $\tilde{\Omega}_{\delta,n}$, the indicator $\mathbf{1}_{\tilde{\Omega}_{\delta,n}}$ is $\mathcal{F}_{t_n}$-measurable. 
Since $u_h^n$, $v_{1,h}^n$, and $v_{2,h}^n$ are also $\mathcal{F}_{t_n}$-measurable, the stochastic sums
    $$\sum_{n=0}^{l}
    \mathbf{1}_{\tilde{\Omega}_{\delta,n}}g_h^n\Delta W_n,
    \qquad
    \sum_{n=0}^{l}
    \mathbf{1}_{\tilde{\Omega}_{\delta,n}}
    (g_h^n\Delta W_n,u_h^n),
    \qquad 0\leq l\leq m,$$
are discrete martingales.
By the discrete and continuous BDG inequality, and hypothesis (v), we have
\begin{align}
\label{ineq:stability-I1-2D}
    \mathbb{E}&\left[\left(\sum_{n=0}^{m}\mathbf{1}_{\tilde{\Omega}_{\delta,n}} \| g_h^n \Delta W_n\|^2 \right)^q \right] 
    \leq C_{b}' \mathbb{E}\left[ \max_{0 \leq l \leq m} \left\| \sum_{n=0}^{l}\mathbf{1}_{\tilde{\Omega}_{\delta,n}} g_h^n \Delta W_n \right\|^{2q} \right]  
    \leq {C_b'}{C_b} K^q \mathbb{E}\left[\left( k\sum_{n=0}^{m}\mathbf{1}_{\tilde{\Omega}_{\delta,n}} \| g_h^n\|^2 \right)^q \right] \notag \\
    &\hspace{1cm}
    \leq C_b'C_b K^q \mathbb{E}\left[ \left( C+ k\sum_{n=0}^{m}\mathbf{1}_{\tilde{\Omega}_{\delta,n}} D_3^2\| u_h^n\|^2 + kD_4^2 \sum_{n=0}^{m}\mathbf{1}_{\tilde{\Omega}_{\delta,n}} (\| v_{1,h}^n\|^2 + \| v_{2,h}^n\|^2) \right)^q \right] \notag \\
    &\hspace{1cm}
    \leq C(1 + \| u_h^{0}\|^{2q} + k^q \| v_{1,h}^{0}\|^{2q} + k^q \| v_{2,h}^{0}\|^{2q}) 
    + Ck\sum_{n=0}^{m}\mathbb{E}\left[\mathbf{1}_{\tilde{\Omega}_{\delta,n}} \| u_h^{n+1}\|^{2q} \right] \notag \\
    &\hspace{1.5cm}
    + C_b'C_b K^qD_4^{2q}(1+\epsilon_5) \mathbb{E}\left[\left(k\sum_{n=0}^{m} \mathbf{1}_{\tilde{\Omega}_{\delta,n}} \left(\| v_{1,h}^{n+1}\|^2 +\| v_{2,h}^{n+1}\|^2 \right) \right)^q \right], 
\end{align}
where the last inequality follows from the discrete H\"older inequality and \eqref{ineq:convex-1} in Lemma \ref{lem:generalized-discrete-Holder}. By the BDG inequality, the Cauchy-Schwarz inequality, discrete H\"older inequality, hypothesis (v) and inequality \eqref{ineq:convex-1},
\begin{align}
\label{ineq:stability-I2-2D}
    \mathbb{E}&\left[\max_{0 \leq l \leq m} \left| \sum_{n=0}^{l} \mathbf{1}_{\tilde{\Omega}_{\delta,n}} (g_h^n \Delta W_n,u_h^{n}) \right|^q \right] \leq C_b \mathbb{E}\left[ \left( kK\sum_{n=0}^{m} \mathbf{1}_{\tilde{\Omega}_{\delta,n}} \|g_h^n \|^2 \|u_h^{n}\|^2 \right)^{\frac{q}{2}} \right] \notag \\
    &\leq C_b \mathbb{E}\left[ \max_{0 \leq n \leq m} \mathbf{1}_{\tilde{\Omega}_{\delta,n}} \|u_h^{n}\|^{q} \left( kK\sum_{n=0}^{m} \mathbf{1}_{\tilde{\Omega}_{\delta,n}} \|g_h^n \|^2 \right)^{\frac{q}{2}} \right] \notag \\
    &\leq \frac{1}{4\epsilon_6} \mathbb{E}\left[ \max_{0 \leq n \leq m} \mathbf{1}_{\tilde{\Omega}_{\delta,n}} \|u_h^{n}\|^{2q} \right] + \epsilon_6C_b^2 K^q \mathbb{E}\left[ \left( k\sum_{n=0}^{m} \mathbf{1}_{\tilde{\Omega}_{\delta,n}} \|g_h^n \|^2 \right)^q \right] \notag \\
    & \leq \frac{1}{4\epsilon_6} \mathbb{E}\left[ \max_{0 \leq n \leq m} \mathbf{1}_{\tilde{\Omega}_{\delta,n}} \|u_h^{n}\|^{2q} \right] + \epsilon_6C_b^2 K^q \mathbb{E}\left[ \left( C+ kD_3^2\sum_{n=0}^{m}\mathbf{1}_{\tilde{\Omega}_{\delta,n}} \| u_h^n\|^2 + kD_4^2\sum_{n=0}^{m}\mathbf{1}_{\tilde{\Omega}_{\delta,n}} (\| v_{1,h}^n\|^2 +\| v_{2,h}^n\|^2) \right)^q \right] \notag \\
    &\leq C(1+\|u_h^{0}\|^{2q}+k^q \| v_{1,h}^{0}\|^{2q} + k^q \| v_{2,h}^{0}\|^{2q}) + Ck\sum_{n=0}^{m} \mathbb{E}\left[\mathbf{1}_{\tilde{\Omega}_{\delta,n}} \|u_h^{n+1}\|^{2q} \right]  \notag \\
    &+ \frac{1}{4\epsilon_6} \mathbb{E}\left[ \max_{0 \leq n \leq m} \mathbf{1}_{\tilde{\Omega}_{\delta,n}} \|u_h^{n+1}\|^{2q} \right]
    + \epsilon_6C_b^2 K^q D_4^{2q} (1+\epsilon_7) \mathbb{E}\left[\left(k\sum_{n=0}^{m} \mathbf{1}_{\tilde{\Omega}_{\delta,n}} \left(\| v_{1,h}^{n+1}\|^2 +\| v_{2,h}^{n+1}\|^2 \right) \right)^q \right].
\end{align}

Plugging \eqref{ineq:stability-I1-2D} and \eqref{ineq:stability-I2-2D} into \eqref{ineq:stability-main-expect-2D}, we reach
\begin{align*}
    \mathcal{C}_1 &\mathbb{E}\left[\max_{0 \leq l \leq m} \mathbf{1}_{\tilde{\Omega}_{\delta,l}} \| u_h^{l+1}\|^{2q} \right] + \mathcal{C}_2 \mathbb{E}\left[\left(\sum_{n=0}^{m} \mathbf{1}_{\tilde{\Omega}_{\delta,n}} \| u_h^{n+1} - u_h^n\|^2 \right)^q \right] + \mathcal{C}_3 \mathbb{E}\left[\left(k\sum_{n=0}^{m} \mathbf{1}_{\tilde{\Omega}_{\delta,n}} \left(\| v_{1,h}^{n+1}\|^2 +\| v_{2,h}^{n+1}\|^2 \right) \right)^q \right] \\
    &\leq C(1 + \| u_h^{0}\|^{2q} + k^q \| v_{1,h}^{0}\|^{2q} + k^q \| v_{2,h}^{0}\|^{2q}) + Ck\sum_{n=0}^{m}\mathbb{E}\left[ \max_{0 \leq l \leq n} \mathbf{1}_{\tilde{\Omega}_{\delta,l}} \| u_h^{l+1}\|^{2q} \right],
\end{align*}
where the three coefficients are given by
\begin{align*}
    \mathcal{C}_1 &= \frac{1}{2^q}-\frac{2^q(2^{q-1}+\epsilon_4)}{4\epsilon_6}, \quad \mathcal{C}_2 = \left(\frac{1}{2}-\frac{1}{\epsilon_0}-\frac{1}{\epsilon_3} - \epsilon_1 \right)^q, \\
    \mathcal{C}_3 &= \left(\alpha - 4(C_{\mu}^{*})^2\epsilon_0 \delta - \epsilon_2 B_3^2 \right)^q -  (2^{q-1}+\epsilon_4)\epsilon_3^q C_b'C_b K^qD_4^{2q}(1+\epsilon_5) - (2^{q-1}+\epsilon_4) 2^q \epsilon_6C_b^2 K^q D_4^{2q} (1+\epsilon_7).
\end{align*}
In Appendix \ref{appendix-eps-choice}, we show that we can find fixed $\delta >0$ and $\{\epsilon_i >0: i=0,...,7\}$ such that $\mathcal{C}_1, \mathcal{C}_2, \mathcal{C}_3$ are all positive, if and only if $\alpha > \alpha_0(q) := \left(2^{2q-1}C_b' + 2^{5q-4}C_b \right)^{1/q} C_b^{1/q} D_4^2K$. Applying the discrete Gr\"onwall inequality, Lemma \ref{lem:v_h^0-2d}, and rearranging the coefficients, we obtain \eqref{ineq:high-moment-stability-2D}. 

\noindent {\it Step 3.} Lastly, we show that the subsets $\{ \tilde{\Omega}_{\delta,n}\}_{n=0}^{N_T}$ defined in \eqref{def:stability-subset} satisfy $\mathbb{P}(\tilde{\Omega}_{\delta,N_T}) \to 1$ as $k/h^{2p_0} \to 0$. Actually, we will prove this result for any $\delta \in(0, \min\{\delta(q), \delta(p_0)\}]$, where $\delta(q)$ is defined in Appendix \ref{appendix-eps-choice}.
By Markov's inequality and the inverse inequality, we have
\begin{align}
\label{ineq:stability-markov-2D}
    \mathbb{P}(\tilde{\Omega}_{\delta,N_T}) &\geq 1 - \frac{k}{\delta} \mathbb{E}\left[ \max_{0 \leq n \leq N_T} \mathbf{1}_{\tilde{\Omega}_{\delta,n-1}}\left( 1 + \| u_h^{n}\|_{\infty}^{2p_0} \right) \right] \notag \\
    &\geq 1 - Ck \mathbb{E}\left[ \max_{0 \leq n \leq N_T} \mathbf{1}_{\tilde{\Omega}_{\delta,n-1}} \left( 1 + h^{-2p_0}\| u_h^{n}\|^{2p_0} \right) \right] \geq 1 - Ck\left(1 + h^{-2p_0} \right),
\end{align}
where the last inequality follows from the stability estimate \eqref{ineq:high-moment-stability-2D} applied with exponent $q=p_0$, namely,
\begin{align*}
    \mathbb{E}\left[\max_{0 \leq n \leq N_T-1} \mathbf{1}_{\tilde{\Omega}_{\delta,n}} \| u_h^{n+1}\|^{2p_0} \right] 
    & \leq C.
\end{align*}
Hence, we require $\delta \in (0, \min\{\delta(q), \delta(p_0)\}]$, and $\alpha > \max\{ \alpha(q), \alpha(p_0)\}$ so that \eqref{ineq:high-moment-stability-2D} holds with $q=p_0$. Therefore, $\mathbb{P}(\tilde{\Omega}_{\delta,N_T}) \to 1$ as $k\to 0$ and $k/h^{2p_0} \to 0$. 
\end{proof}


As a special case, the lower bound on $\alpha$ can be further relaxed for the second moment estimate. 
\begin{remark}[Range of $\alpha$ for Second Moment Stability Estimate]
\label{rmk:stability-second-moment-2d}
Consider the case $q = 1$ in Theorem \ref{thm:stability-estimate-2D}. The stochastic parabolic condition can be relaxed to $\alpha > KD_4^2/2$, which is sharp. To achieve this, we can first show the following weaker stability estimate: 
\begin{align}
\label{ineq:stability-estimate-weaker-2d}
    \max_{0\leq n \leq N_T} \mathbb{E}\left[ \mathbf{1}_{\tilde{\Omega}_{\delta,n-1}}\| u_h^{n}\|^2 \right] &+ \sum_{n=0}^{N_T-1} \mathbb{E}\left[ \mathbf{1}_{\tilde{\Omega}_{\delta,n}} \| u_h^{n+1} - u_h^n\|^2 \right] + k\sum_{n=0}^{N_T} \mathbb{E}\left[ \mathbf{1}_{\tilde{\Omega}_{\delta,n-1}} (\| v_{1,h}^{n}\|^2 + \| v_{2,h}^{n}\|^2) \right] \notag \\
    &\leq C\left(1+\| u_0\|^{2} + k \|u_0\|_{1}^{2} \right). 
\end{align}
The key point is that, by taking expectation of the discrete energy identity before applying $\max_{0\leq n \leq N_T}$ over time, the martingale term has zero expectation. This avoids estimating the last term in \eqref{ineq:stability-main-2D} which leads to the stronger bound on $\alpha$ in the high-moment argument. 
The maximal-in-time estimate \eqref{ineq:high-moment-stability-2D} for $q=1$ can be recovered under the same relaxed condition by following the proof of Theorem \ref{thm:stability-estimate-2D}. For a more detailed proof, we refer to \cite{chen_semi-linear_2d}. 

Moreover, as mentioned in \cite{LiShuTang2021_ESAIM}, when \eqref{spde:conv-diff-2d} is driven by the standard Brownian motion ($K=1$), the $L^2$ stability estimate still holds for the degenerate case $\alpha =D_4^2/2$, provided $f_1=f_2=0$ and $B_3=0$. 
\end{remark}

The condition $k/h^{2p_0}\to0$ in Theorem \ref{thm:stability-estimate-2D}, which ensures that the subsets $\{ \tilde{\Omega}_{\delta,n}\}_{n=0}^{N_T}$ converge to the full sample space, can be restrictive.
However, this stability condition can be relaxed by using the high-moment error estimate result in Section \ref{sec:error-estimate-2d}. 
\begin{corollary}[Improved Stability Condition]
\label{coro:relaxed-stability-2d}
For any fixed $q \geq 1$, under the hypothesis of Theorems \ref{thm:stability-estimate-2D} and \ref{thm:high-moment-error-estimate-2D}, the high moment stability estimate holds with $\mathbb{P}(\tilde{\Omega}_{\delta,N_T}) \to 1$ as $h \to 0$ provided $k \leq Ch^{2+2C_q^*\beta}$, where $\beta >0$ can be arbitrarily small and $C_q^*$ is the constant in Theorem \ref{thm:high-moment-error-estimate-2D}.  
\end{corollary}
\begin{proof}
Theorem \ref{thm:high-moment-error-estimate-2D} leads to the conclusion that $\mathbb{P}(\Omega_{\kappa,N_T}) \to 1$ as $h \to 0$ provided $k \leq Ch^{2+2C_q^*\beta}$, with the set $\Omega_{\kappa,N_T}$ defined in \eqref{def:error-subset-2d} and $\kappa=\ln\bigl(\ln(h^{-\beta})\bigr)$. 
Next, we show that $\Omega_{\kappa,n}\subset \tilde{\Omega}_{\delta,n}$ for each $n$ and small $h$.
This follows from $\kappa k\leq \delta$ which is satisfied by our choice of $\kappa$, and the requirement that $k \leq Ch^{2+2C_q^*\beta}$ when $h$ is small enough. Therefore, $\mathbb{P}(\tilde{\Omega}_{\delta,N_T}) \to 1$ as $h \to 0$ under the relaxed condition $k \leq Ch^{2+2C_q^*\beta}$.
\end{proof}

\begin{remark}[Stability Condition]
\label{rmk:stability-condition}
The condition $k \leq Ch^{2+2C_q^*\beta}$ in Corollary \ref{coro:relaxed-stability-2d} is not a CFL stability restriction in the deterministic setting.  
Note that the localized stability estimate \eqref{ineq:high-moment-stability-2D} and error estimates \eqref{ineq:high-moment-error-estimate-2D} hold for arbitrary $h$ and $k$. This coupling condition is only used to ensure that the localized subsets converge to the full sample space in probability. 
Localization is needed because, in the nonlinear stochastic setting, one cannot assume pathwise boundedness of the exact or numerical solution.  
In particular, the localization treatment is unnecessary for the globally Lipschitz semilinear problem \cite{chen_semi-linear_2d}. 

In addition, we would like to point out that the power of $h$ in the condition is closely related to the inverse inequality. 
In general, the same argument in $d$ spatial dimensions leads to a condition of the form $k\leq C h^{d+2C_{q}^*\beta}$.
\end{remark}

\begin{remark}[Recursive Definition of Subsets]
\label{rmk:recursive-def-subset}
Unlike the subsets used in \cite{Carelli_Prohl_SIAMNu2012, Li_Wu_Xing_2022}, the subsets in \eqref{def:stability-subset} and \eqref{def:error-subset-2d} are defined recursively.
This recursive structure allows the estimates to be carried out by a discrete induction argument: once a path exits the localized set, the corresponding indicator switches off the later contributions. 
This is equivalent to a mathematical induction type argument, which is needed if the subsets are not defined recursively or only the weaker estimate \eqref{ineq:stability-estimate-weaker-2d} is proved.
For the semi-discrete LDG scheme, the analogous argument would require a continuous-in-time induction argument.
\end{remark}

The localization required by the explicit nonlinear convection and source terms can be avoided if these deterministic nonlinearities are evaluated implicitly,
i.e., approximated by the implicit backward-Euler method. 
Furthermore, the conclusions in Remark \ref{rmk:stability-fully-implicit-2d} also hold when there is no nonlinear convection or source term, that is, when $f=0$ and $\psi = \psi_1$ \cite{chen_semi-linear_2d}. 
\begin{remark}[Stability Estimate for an Implicit Treatment of the Nonlinear Terms]
\label{rmk:stability-fully-implicit-2d}
If we consider an implicit Euler scheme in which $f_1$, $f_2$, and $-u^\lambda$ are evaluated at $u_h^{n+1}$, we have $\left(-(u_h^{n+1})^{\lambda}, u_h^{n+1} \right) \leq 0$ and $\tilde{H}(f_1,f_2,u_h^{n+1},u_h^{n+1}) \leq 0$ by Lemma \ref{lem:num-flux-2D}. Therefore, no localization is needed in the high moment stability estimate. By the same argument, for any $q \geq 1$, 
\begin{align*}
    \mathbb{E}\left[ \max_{0\leq n \leq N_T} \| u_h^{n}\|^{2q} \right] &+ \mathbb{E}\left[\left(\sum_{n=0}^{N_T-1} \| u_h^{n+1} - u_h^n\|^2 \right)^q \right] + \mathbb{E}\left[ \left(k\sum_{n=0}^{N_T} (\| v_{1,h}^{n}\|^2 + \| v_{2,h}^{n}\|^2) \right)^q \right] \\
    &\leq C\left(1+\| u_0\|^{2q} + k^q \|u_0\|_{1}^{2q} \right),
\end{align*}
where the convection fluxes $f_1$ and $f_2$ can be any differentiable function that is locally Lipschitz, and the lower bound on $\alpha$ stays the same. 
In the case $q=1$, the bound is simplified to $\alpha > KD_4^2/2$. 
Moreover, when \eqref{spde:conv-diff-2d} is driven by the standard Brownian motion ($K=1$), the stability estimate still holds for the degenerate case $\alpha = D_4^2/2$, provided $B_3=0$, where the implicit convection term remains nonpositive and no assumption $f_1=f_2=0$ is required. 
\end{remark}

\end{section}

\begin{section}{High Moment Error Estimate}
\label{sec:error-estimate-2d}
In this section, we show the high-moment convergence of the numerical scheme \eqref{eq:fully-discrete-1-2d}-\eqref{eq:fully-discrete-4-2d} for the strong solution with sufficient regularity. 
High moment strong order of convergence in space and time is proved on recursively defined subsets of the sample space whose probabilities converge to one. 
The resulting convergence rates are arbitrarily close to the optimal rates $\mathcal{O}(h^{r+1})$ in space and $\mathcal{O}(k^{1/2})$ in time.

Suppose $u$ represents the strong solution to \eqref{spde:conv-diff-2d}, with the auxiliary variables $v_1, v_2, w_1,$ and $w_2$ specified according to \eqref{eq:auxiliary}, and $u^n:=u(t_n)$, $v_i^n:=v_i(t_n)$, $w_i^n:=w_i(t_n)$, $i=1,2$. 
Let $(u_h^n,v_{1,h}^n,v_{2,h}^n,w_{1,h}^n,w_{2,h}^n)$ be the numerical solution of the LDG-IMEX-Euler method \eqref{eq:fully-discrete-1-2d}--\eqref{eq:fully-discrete-4-2d}. 
We define
\[
e_u^n=u^n-u_h^n,\qquad
e_{v_i}^n=v_i^n-v_{i,h}^n,\qquad
e_{w_i}^n=w_i^n-w_{i,h}^n,\qquad i=1,2.
\]
Using the projections introduced in Section \ref{subsec:prelim-2d}, we split the errors as
\[
e_u^n=\xi_u^n-\eta_u^n, \qquad
e_{v_i}^n=\xi_{v_i}^n-\eta_{v_i}^n, \qquad
e_{w_i}^n=\xi_{w_i}^n-\eta_{w_i}^n, \qquad i=1,2.
\]
where $\xi_u^n = \mathcal{P}^-u^n - u_h^n$, $\eta_u^n = \mathcal{P}^-u^n - u^n$, and, for $i=1,2$, $\xi_{v_i}^n = \mathcal{P} v_i^n - v_{i,h}^n$, $\eta_{v_i}^n = \mathcal{P} v_i^n - v_i^n$ and $\xi_{w_i}^n = \mathcal{P}^+ w_i^n - w_{i,h}^n$, $\eta_{w_i}^n = \mathcal{P}^+ w_i^n - w_i^n$. 

In addition to hypotheses (i)--(v), we impose an extra condition concerning the leading matrix $A=\{a_{ij}\}$ and the stochastic diffusion term $g$. 
\begin{enumerate}[(vi)]
    \item For any $(\omega,t) \in \Omega \times [0,T]$,
    \begin{align*}
       \| g(\omega,\cdot,t,u,v_1,v_2)\|_{\mathcal{L}_2^2} \leq C\left(1 + \| u\|_{2} + \| v_1\|_{2} + \| v_2\|_{2} \right).
    \end{align*}

    \item For any $(\omega,x,y,t,s) \in \Omega \times \mathcal{D} \times [0,T]^2$, and $j=1,2$, 
    \begin{align*}
        &|(a_{1,j})_x(x,y,t,u)| \leq C(1+ |u| + |v_1|), \quad |(a_{2,j})_y(x,y,t,u)| \leq C(1+ |u| + |v_2|), \\
        &\left|(a_{1,j})_x(x,y,t,u(\cdot,t)) - (a_{1,j})_x(x,y,s,u(\cdot,s)) \right| \\
        &\hspace{1cm} 
        \leq C|t-s|^{\frac{1}{2}}(1+ |u(t)| + |v_1(t)|) 
        + C(|u(t)-u(s)| + |v_1(t)-v_1(s)|), \\
        &\left|(a_{2,j})_y(x,y,t,u(\cdot,t)) - (a_{2,j})_y(x,y,s,u(\cdot,s)) \right| \\
        &\hspace{1cm} 
        \leq C|t-s|^{\frac{1}{2}}(1 + |u(t)| + |v_2(t)|) 
        + C(|u(t)-u(s)| + |v_2(t)-v_2(s)|).
    \end{align*}
    
\end{enumerate}
Assumption (vi) follows from the growth condition in Lemma \ref{lem:Holder-High-Moments}, while assumption (vii) specifies the linear growth and Lipschitz condition of the spatial derivatives of the diffusion matrix coefficients. We note that unlike the semilinear setting considered in \cite{chen_semi-linear_2d}, the matrix $A = \{ a_{ij}\}$ is allowed to depend on the solution $u$ for error estimate. In the one-dimensional case, $\{ a_{ij}\}$ can even grow linearly in $u$ instead of being uniformly bounded as in assumption (ii). 

For the error estimate, we introduce a recursively defined family of subsets $\{ \Omega_{\kappa,m}\}_{m=0}^{N_T}$ parameterized by $\kappa$: 
\begin{align}
\label{def:error-subset-2d}
    &\Omega_{\kappa,-1} =\Omega,  \notag \\
    &\Omega_{\kappa,\ell} = \{ \omega \in \Omega: \sup_{0 \leq t \leq t_{\ell+1}} \left(1+\| u(t)\|_{W^{2,\infty}}^2 + \| u(t)\|_{\infty}^{2p_0} \right) + \max_{0 \leq n \leq \ell} \mathbf{1}_{\Omega_{\kappa,n-1}}\left( 1+\| u_h^{n}\|_{\infty}^{2p_0} \right) \leq \kappa \}, \,0 \leq \ell \leq N_T-1, \notag \\
    &\Omega_{\kappa,N_T} = \{ \omega \in \Omega: \sup_{0 \leq t \leq T} \left(1+\| u(t)\|_{W^{2,\infty}}^2 + \| u(t)\|_{\infty}^{2p_0} \right) + \max_{0 \leq n \leq N_T} \mathbf{1}_{\Omega_{\kappa,n-1}}\left( 1+\| u_h^{n}\|_{\infty}^{2p_0} \right) \leq \kappa \}.
\end{align}
Clearly, these events are decreasing: $\Omega_{\kappa,N_T} \subset \Omega_{\kappa,N_T-1} \subset \cdot \cdot \cdot \subset \Omega_{\kappa,0} \subset \Omega_{\kappa,-1} = \Omega$. 

Note that in our definition, $\Omega_{\kappa,\ell}$ is not $\mathcal{F}_{t_{\ell}}$-measurable. To convert the discrete process $\sum_{n=0}^m \mathbf{1}_{\Omega_{\kappa,n}} X_n$ into a martingale, we make use of the following lemma to enable the application of BDG inequality to the terms $\mathcal{T}_1$, $\mathcal{T}_3$ in Theorem \ref{thm:high-moment-error-estimate-2D}. 

\begin{lemma}
\label{lem:subset-property}
    Let $\{ X_n\}_{n=0}^{\infty}$ be a discrete stochastic process. For any fixed positive integer $m$, we have
    \[ \max_{0 \leq l \leq m} \left|\sum_{n=0}^l \mathbf{1}_{\Omega_{\kappa,n}} X_n  \right| \leq \max_{0 \leq l \leq m} \left|\sum_{n=0}^l \mathbf{1}_{\Omega_{\kappa,n-1}} X_n  \right|. \]
\end{lemma}
\begin{proof}
    Fix $\omega \in \Omega$. If $\omega \in \Omega_{\kappa,m}$, then by the nesting property of $\{\Omega_{\kappa,n}\}$, all indicators on both sides are equal to one for $0\leq n\leq m$, and the two sides are identical.  
    If $\omega \notin \Omega_{\kappa,m}$, then there exists an integer $0 \leq s \leq m$ such that $\omega \in \Omega_{\kappa,s-1}\backslash \Omega_{\kappa,s}$. In this case, 
    \begin{align*}
        \text{LHS} = \max_{0 \leq l \leq s-1} \left|\sum_{n=0}^l \mathbf{1}_{\Omega_{\kappa,n}} X_n  \right| = \max_{0 \leq l \leq s-1} \left|\sum_{n=0}^l X_n  \right| \leq \max_{0 \leq l \leq s} \left|\sum_{n=0}^l X_n  \right| = \text{RHS}, 
    \end{align*}
    which proves the claim.
\end{proof}

The following lemmas estimate the individual terms $\mathcal{I}_1$--$\mathcal{I}_6$ appearing in the proof of Theorem \ref{thm:high-moment-error-estimate-2D}. Their proofs are provided in Appendix \ref{appendixA}.

\begin{lemma}
\label{lem:error2D-term1-2}
Assume hypotheses {\rm(iv)} and {\rm(v)} hold. For any $0\le m\le N_T-1$, any $\epsilon_0, \epsilon_1 > 0$, we have
\begin{align}
    &\sum_{n=0}^{m} \mathbf{1}_{\Omega_{\kappa,n}} \int_{t_n}^{t_{n+1}} \| \psi_1(\cdot,t,u,v_1,v_2) - \psi_1(\cdot,t_{n},u_h^{n},v_{1,h}^n,v_{2,h}^n) \|^2 \, \mathrm{d}t  \notag \\
    &\leq C\sum_{n=0}^{m} \int_{t_n}^{t_{n+1}} \mathbf{1}_{\Omega_{\kappa,n}}\|u-u^n\|_{1}^2 \, \mathrm{d}t + Ck^2\sum_{n=0}^{m}(1+\| u^n\|_{1}^2) + (2+\epsilon_0) B_1^2 k\sum_{n=0}^{m} \mathbf{1}_{\Omega_{\kappa,n}} (\| \xi_{v_1}^{n}\|^2 +  \| \xi_{v_2}^{n}\|^2)  \notag \\
    &\quad + Ck \sum_{n=0}^{m} \mathbf{1}_{\Omega_{\kappa,n}} \| \xi_u^{n}\|^2 + Ch^{2r+2} k \sum_{n=0}^{m} (\| u^{n}\|_{{r+1}}^2 + \| v_1^{n}\|_{{r+1}}^2 + \| v_2^{n}\|_{{r+1}}^2), \label{ineq:error2D-psi} \\
    &\sum_{n=0}^{m} \mathbf{1}_{\Omega_{\kappa,n}} \int_{t_n}^{t_{n+1}} \|g(\cdot,t,u,v_1,v_2) - g(\cdot,t_{n},u_h^{n},v_{1,h}^n,v_{2,h}^n) \|^2 \, \mathrm{d}t  \notag \\
    &\leq C\sum_{n=0}^{m} \int_{t_n}^{t_{n+1}} \mathbf{1}_{\Omega_{\kappa,n}}\|u-u^n\|_{1}^2 \, \mathrm{d}t + Ck^2\sum_{n=0}^{m}(1+\| u^n\|_{1}^2) + (2+\epsilon_1) D_2^2 k\sum_{n=0}^{m} \mathbf{1}_{\Omega_{\kappa,n}} (\| \xi_{v_1}^{n}\|^2 +  \| \xi_{v_2}^{n}\|^2) \notag \\
    &\quad + Ck \sum_{n=0}^{m} \mathbf{1}_{\Omega_{\kappa,n}} \| \xi_u^{n}\|^2 + Ch^{2r+2} k \sum_{n=0}^{m} (\| u^{n}\|_{{r+1}}^2 + \| v_1^{n}\|_{{r+1}}^2 + \| v_2^{n}\|_{{r+1}}^2).  \label{ineq:error2D-g}
\end{align}
\end{lemma}

\begin{lemma}
\label{lem:error2D-term3}
Assume hypothesis {\rm(ii)} holds. Let $i,j = 1,2$, for any $0\le m\le N_T-1$ and $\epsilon_2 > 0$,
\begin{align}
\label{ineq:lem-error-term3-2D}
    &\max_{0 \leq l \leq m} k\sum_{n=0}^l \mathbf{1}_{\Omega_{\kappa,n}} \big(a_{ij}^{n+1} - a_{h,ij}^{n+1},\, -v_j^{n+1} \xi_{v_i}^{n+1} \big) + \max_{0 \leq l \leq m} k\sum_{n=0}^l \mathbf{1}_{\Omega_{\kappa,n}} \big(a_{h,ij}^{n+1}\,\eta_{v_j}^{n+1},\, \xi_{v_i}^{n+1} \big) \notag \\
    &\leq \epsilon_2 k\sum_{n=0}^{m} \mathbf{1}_{\Omega_{\kappa,n}} \| \xi_{v_i}^{n+1}\|^2 
    + Ch^{2r+2} k \sum_{n=0}^{m} \left( \| v_j^{n+1}\|_{{r+1}}^2 + \kappa \mathbf{1}_{\Omega_{\kappa,n}} \| u^{n+1}\|_{{r+1}}^2\right)
    + C\kappa k \sum_{n=0}^{m} \mathbf{1}_{\Omega_{\kappa,n}} \| \xi_u^{n+1}\|^2 . 
\end{align}

\end{lemma}

\begin{lemma}
\label{lem:error2D-term4}
Assume hypotheses {\rm(ii)}, {\rm(vii)} hold. For any $0\le m\le N_T-1$,
\begin{align}
\label{ineq:error-estimate-term4-2D}
    \mathcal{I}_4 
    &:= \max_{0 \leq l \leq m} \Bigg| \sum_{n=0}^l \int_{t_n}^{t_{n+1}} \mathbf{1}_{\Omega_{\kappa,n}} H^+(w_1(t)-w_{1}^{n+1},w_2(t)-w_{2}^{n+1},\xi_u^{n+1}) \, \mathrm{d}t \Bigg| \notag \\
    &\leq k\sum_{n=0}^{m} \mathbf{1}_{\Omega_{\kappa,n}} \|\xi_u^{n+1}\|^2 + C\kappa \sum_{n=0}^{m} \int_{t_n}^{t_{n+1}} \mathbf{1}_{\Omega_{\kappa,n}}\|u(t)-u^{n+1}\|_{2}^2 \, \mathrm{d}t + C\kappa k^2\sum_{n=0}^{m} (\|v_1^{n+1}\|_{1}^2 + \|v_2^{n+1}\|_{1}^2).
\end{align}
\end{lemma}

\begin{lemma}
\label{lem:error2D-term5}
Assume hypothesis {\rm(iii)} holds. For any $0\le m\le N_T-1$ and $\epsilon_3 > 0$,
\begin{align}
\label{ineq:error-estimate-term5-2D}
    \mathcal{I}_5 
    &:= \max_{0 \leq l \leq m} \Bigg| \sum_{n=0}^l \int_{t_n}^{t_{n+1}} \mathbf{1}_{\Omega_{\kappa,n}} D(f_1,f_2;u(t),u_h^n,\xi_u^{n+1}) \, \mathrm{d}t \Bigg|  \notag \\
    &\leq Ch^{2r+2} k\sum_{n=0}^{m}(\| v_1^{n+1}\|_{r+1}^2 + \| v_2^{n+1}\|_{r+1}^2) + \epsilon_3 k \sum_{n=0}^{m}\mathbf{1}_{\Omega_{\kappa,n}} (\| \xi_{v_1}^{n+1}\|^2 + \| \xi_{v_2}^{n+1}\|^2) \notag \\
    &\quad + C\kappa \sum_{n=0}^{m} \int_{t_n}^{t_{n+1}} \|u(t)-u^{n}\|^2 \, \mathrm{d}t 
    + \kappa k\sum_{n=0}^{m} \mathbf{1}_{\Omega_{\kappa,n}} \|\xi_u^{n}\|^2 + C\kappa h^{2r+2} k\sum_{n=0}^{m}\| u^{n+1}\|_{r+1}^2.
\end{align}
\end{lemma}

\begin{lemma}
Assume hypothesis {\rm(iv)} holds. For any $\epsilon_4 >0$, any $0\le m\le N_T-1$, we have 
\label{lem:error2D-term6}
\begin{align}
\label{ineq:error-estimate-term6-2d}
    \mathcal{I}_6 &:= \max_{0 \leq l \leq m} \sum_{n=0}^{l} \int_{t_n}^{t_{n+1}} \mathbf{1}_{\Omega_{\kappa,n}} \left((u_h^n)^{\lambda} - u(t)^{\lambda}, \xi_u^{n+1} \right) \, \mathrm{d}t
    \leq \frac{\lambda^2}{2} \kappa \sum_{n=0}^{m} \int_{t_n}^{t_{n+1}}\mathbf{1}_{\Omega_{\kappa,n}} \left\| u^n - u(t) \right\|^2 \, \mathrm{d}t  \notag \\
    &\quad + \epsilon_4 \sum_{n=0}^{m} \mathbf{1}_{\Omega_{\kappa,n}} \| \xi_u^{n+1}-\xi_u^{n}\|^2 + C\kappa k \sum_{n=0}^{m+1} \mathbf{1}_{\Omega_{\kappa,n}} \left\| \xi_u^n \right\|^2 + Ch^{2r+2}\kappa k\sum_{n=0}^{m} \mathbf{1}_{\Omega_{\kappa,n}} \| u^{n}\|_{r+1}^2.
\end{align}

\end{lemma}

Next, we propose a Lemma which establishes a relationship between $\| \xi_{v}^n\|$ and $\| \xi_{w}^n\|$. 
This result will be used to estimate the term $\mathcal{I}_7$ in Theorem \ref{thm:high-moment-error-estimate-2D}.  
The proof is provided in Appendix \ref{appendix-w_h-v_h-2d}. 
\begin{lemma}
\label{lem:w_h-v_h-2D}
Assuming hypothesis {\rm(ii)} holds. For any $0\le m\le N_T-1$, we have 
\begin{align}
\label{ineq:w_h-v_h-2D}
    k\sum_{n=0}^m &\mathbf{1}_{\Omega_{\kappa,n}} (\| \xi_{w_1}^{n+1}\|^2 + \| \xi_{w_2}^{n+1}\|^2) \leq 14\Lambda k\sum_{n=0}^m \mathbf{1}_{\Omega_{\kappa,n}} (\| \xi_{v_1}^{n+1}\|^2 + \| \xi_{v_2}^{n+1}\|^2) + C \kappa k\sum_{n=0}^m \mathbf{1}_{\Omega_{\kappa,n}} \| \xi_{u}^{n+1}\|^2  \\
    &+ Ch^{2r+2} k\sum_{n=0}^m (\| w_1^{n+1}\|_{r+1}^2 + \| w_2^{n+1}\|_{r+1}^2 + \| v_1^{n+1}\|_{r+1}^2 + \| v_2^{n+1}\|_{r+1}^2) + C\kappa h^{2r+2} k\sum_{n=0}^m \| u^{n+1}\|_{r+1}^2. \notag 
\end{align}
\end{lemma}

Lastly, we show the high-order approximation to the initial condition $u_0$ and its gradient, provided $u_0$ is sufficiently smooth. The proof is the same as in \cite{chen_semi-linear_2d}.  
\begin{lemma}
\label{lem:xi_v^0-2d}
Assume $u_0 \in H^{r+2}(\mathcal{D})$ and set $u_h^0=\mathcal P^-u_0$.
Let $\xi_{u}^0 := \mathcal{P}^- u_0 - u_{h}^0$, $\xi_{v_i}^0 := \mathcal{P} v_i^0 - v_{i,h}^0$ for $i=1,2$, where $v_i^0=\partial_i u_0$, and where $v_{i,h}^0$ is computed from
\eqref{eq:fully-discrete-2-2d}. Then
\begin{align*}
    \| \xi_{u}^0\|^2 + \| \xi_{v_1}^0\|^2 + \| \xi_{v_2}^0\|^2 \leq Ch^{2r+2}. 
\end{align*}
\end{lemma}

Next, we impose the following regularity assumption $(\mathcal{A}_q)$ on the exact solution. 

\begin{enumerate}
    \item[($\mathcal{A}_q$)] For a fixed $q\in[1,\infty)$, suppose $u_0 \in H^{r+2}(\mathcal{D})$ and 
    \begin{align*}
        &u \in L^2(\Omega,L^{\infty}[0,T; W^{2,\infty}(\mathcal{D})]) \cap L^{2p_0}\left(\Omega,L^{\infty}([0,T] \times \mathcal D) \right) \cap L^{2q}(\Omega, L^{\infty}[0,T; H^{r+3}(\mathcal{D})]),  \\
        &w_1,w_2,f_1(u),f_2(u) \in L^{2q}(\Omega \times[0,T], H^{r+2}(\mathcal{D})), \quad \{\psi,g\}(\cdot,u,\nabla u) \in L^{2q}(\Omega \times[0,T]; H^{r+1}(\mathcal{D})).
    \end{align*}
\end{enumerate}

In the following, we give the result of error estimate. For arbitrarily fixed $q\in[1,\infty)$, suppose the parabolic coefficient satisfies that
$$
\alpha > \widetilde{\alpha_0}:= \left(2^{3q-1}C_b' + 2^{6q-4}C_b \right)^{1/q}C_b^{1/q}D_2^2K.
$$
\begin{theorem}[Fully Discrete Error Estimate]
\label{thm:high-moment-error-estimate-2D}
    Assume hypotheses {\rm(i)}-{\rm(vii)} and ($\mathcal{A}_q$). Then there exist positive constants $C$, $C_q^{*}$, independent of $h$,$k$, such that for every arbitrarily small $\beta>0$ and all sufficiently small $h$,
\begin{align}
\label{ineq:high-moment-error-estimate-2D}
    \mathbb{E}&\left[\max_{0 \leq n \leq N_T} \mathbf{1}_{\Omega_{\kappa,n-1}}\| e_u^{n}\|^{2q} \right]^{\frac{1}{2q}} + \mathbb{E}\left[ \left(\sum_{n=0}^{N_T-1} \mathbf{1}_{\Omega_{\kappa,n}}\| e_u^{n+1} - e_u^{n}\|^2 \right)^q \right]^{\frac{1}{2q}} 
    \notag \\ 
    &\qquad + \mathbb{E}\left[ \left(k\sum_{n=0}^{N_T}\mathbf{1}_{\Omega_{\kappa,n-1}} (\|e_{v_1}^{n}\|^2 + \|e_{v_2}^{n}\|^2) \right)^q \right]^{\frac{1}{2q}} 
    \leq C h^{-C_q^*\beta}(k^{\frac{1}{2}} + h^{r+1}).
\end{align}
where $\kappa = \ln \bigl(\ln(h^{-\beta}) \bigr)$. Moreover, it holds that $\mathbb{P}(\Omega_{\kappa,N_T}) \to 1$ as $h \to 0$ provided $k \leq Ch^{2+2C_q^*\beta}$. 
\end{theorem}

\begin{proof}
By Lemma \ref{lem:xi_v^0-2d}, we have $\| \xi_u^0\| + \| \xi_{v_1}^0\| + \| \xi_{v_2}^0\| \leq Ch^{r+1}$. Thus all initial-error terms appearing below are bounded by $Ch^{2r+2}$, and can be absorbed into the final error bound. For simplicity, we denote
\begin{alignat*}{2}
    a_{ij}^{n+1} &:= a_{ij}(x,y,t_{n+1},u^{n+1}),
    &\qquad
    a_{h,ij}^{n+1} &:= a_{ij}(x,y,t_{n+1},u_h^{n+1}), \\
    g(t)&:=g(\cdot,t,u(t),v_1(t),v_2(t)),    
    &\qquad
    g_h^n &:=g(\cdot,t_n,u_h^n,v_{1,h}^n,v_{2,h}^n), \\
    \psi(t)&:=\psi(\cdot,t,u(t),v_1(t),v_2(t)),
    &\qquad
    \psi_h^n &:=\psi(\cdot,t_n,u_h^n,v_{1,h}^n,v_{2,h}^n),\\
    \psi_1(t)&:=\psi_1(\cdot,t,u(t),v_1(t),v_2(t)),
    &\qquad
    \psi_{1,h}^n &:=\psi_1(\cdot,t_n,u_h^n,v_{1,h}^n,v_{2,h}^n),
\end{alignat*}
in this proof. We also define, for vector-valued functions $\boldsymbol{\zeta}=(\zeta_1,\zeta_2)$ and $\boldsymbol{\chi}=(\chi_1,\chi_2)$,
\[
    \mathcal{A}_h^{n+1}(\boldsymbol{\zeta},\boldsymbol{\chi})
    :=
    \sum_{i,j=1}^2
    (a_{h,ij}^{n+1}\zeta_j,\chi_i).
\]
We divide the long proof into three steps. 

\noindent {\it Step 1.}     
Subtracting the fully discrete scheme \eqref{eq:fully-discrete-1-2d}-\eqref{eq:fully-discrete-4-2d} from the time-integrated equations satisfied by the exact solutions $(u,v_1,v_2,w_1,w_2)$, and summing over all cells $I_i \times J_j$, we obtain
\begin{align}
    &(e_u^{n+1},r_h) = (e_u^{n},r_h) + \int_{t_n}^{t_{n+1}} H^+(w_1-w_{1,h}^{n+1},w_2-w_{2,h}^{n+1},r_h) \, \mathrm{d}t \notag \\
    &\qquad 
    + \int_{t_n}^{t_{n+1}} D(f_1,f_2;u(t),u_h^n, r_h) \, \mathrm{d}t 
    + \left( \int_{t_n}^{t_{n+1}} \psi(t) - \psi_h^n \, \mathrm{d}t + \int_{t_n}^{t_{n+1}} g(t) - g_h^n \, \mathrm{d}W_t, r_h \right), \label{eq:error-1-2d} \\
    &(e_{v_1}^{n+1},p_h) + (e_{v_2}^{n+1},q_h) = L^-(e_u^{n+1},p_h,q_h), \label{eq:error-2-2d} \\
    &(e_{w_i}^{n+1},z_{i,h}) = \sum_{j=1}^2 \left( a_{ij}^{n+1}v_j^{n+1} -a_{h,ij}^{n+1}v_{j,h}^{n+1},z_{i,h} \right), \quad i=1,2. \label{eq:error-3-4-2d}
\end{align}
Taking $p_h = \xi_{w_1}^{n+1}$, $q_h = \xi_{w_2}^{n+1}$, and $z_{i,h}=\xi_{v_i}^{n+1}$ for $i=1,2$, by the definition of $L^2$ projection, we get
\begin{align*}
    (\xi_{v_1}^{n+1},\xi_{w_1}^{n+1}) +(\xi_{v_2}^{n+1},\xi_{w_2}^{n+1})
    &=  L^-(e_u^{n+1},\xi_{w_1}^{n+1},\xi_{w_2}^{n+1})\\
    &= \sum_{i=1}^2 (\eta_{w_i}^{n+1},\xi_{v_i}^{n+1})
    +\sum_{i,j=1}^2 \left(a_{ij}^{n+1}v_j^{n+1} -a_{h,ij}^{n+1}v_{j,h}^{n+1}, \xi_{v_i}^{n+1}\right).
\end{align*}
For $i,j = 1,2$, by splitting each term into three terms
$$(a_{ij}^{n+1}v_j^{n+1} - a_{h,ij}^{n+1}v_{j,h}^{n+1}) = (a_{ij}^{n+1} - a_{h,ij}^{n+1})v_j^{n+1} + a_{h,ij}^{n+1}\xi_{v_j}^{n+1} - a_{h,ij}^{n+1}\eta_{v_j}^{n+1}, $$ 
we can rewrite the above equation as
\begin{align*}
    \mathcal{A}_h^{n+1}(\boldsymbol{\xi}_v^{n+1},\boldsymbol{\xi}_v^{n+1})
    &=
    L^-(\xi_u^{n+1},\xi_{w_1}^{n+1},\xi_{w_2}^{n+1}) -L^-(\eta_u^{n+1},\xi_{w_1}^{n+1},\xi_{w_2}^{n+1})\\
    &\quad
    -\sum_{i=1}^2 (\eta_{w_i}^{n+1},\xi_{v_i}^{n+1})
    -\sum_{i,j=1}^2 (a_{ij}^{n+1}-a_{h,ij}^{n+1}, v_j^{n+1}\xi_{v_i}^{n+1})
    +\sum_{i,j=1}^2 (a_{h,ij}^{n+1}\eta_{v_j}^{n+1}, \xi_{v_i}^{n+1}),
\end{align*}
where $\boldsymbol{\xi}_v^{n+1} =(\xi_{v_1}^{n+1},\xi_{v_2}^{n+1})^{\top}$.
Multiplying the above equation by $k$, adding it to \eqref{eq:error-1-2d} with $r_h = \xi_u^{n+1}$, multiplying by $\mathbf{1}_{\Omega_{\kappa,n}}$, and summing from $n=0$ to $l$, we get
\begin{align*}
    &\sum_{n=0}^l \mathbf{1}_{\Omega_{\kappa,n}} (\xi_u^{n+1}-\xi_u^n,\xi_u^{n+1}) + k\sum_{n=0}^l \mathbf{1}_{\Omega_{\kappa,n}}  \mathcal{A}_h^{n+1} (\boldsymbol{\xi}_v^{n+1},\boldsymbol{\xi}_v^{n+1}) \\
    &=
    \sum_{n=0}^l \mathbf{1}_{\Omega_{\kappa,n}} (\eta_u^{n+1}-\eta_u^n,\xi_u^{n+1}-\xi_u^n) 
    + \sum_{n=0}^l \mathbf{1}_{\Omega_{\kappa,n}} (\eta_u^{n+1}-\eta_u^n,\xi_u^n) 
    - k\sum_{n=0}^l \mathbf{1}_{\Omega_{\kappa,n}} L^-(\eta_u^{n+1},\xi_{w_1}^{n+1},\xi_{w_2}^{n+1})     \\
    &\quad
    - k\sum_{n=0}^l \mathbf{1}_{\Omega_{\kappa,n}} H^+(\eta_{w_1}^{n+1},\eta_{w_2}^{n+1},\xi_u^{n+1}) 
    - k\sum_{n=0}^l \mathbf{1}_{\Omega_{\kappa,n}} \sum_{i=1}^2 (\eta_{w_i}^{n+1},\xi_{v_i}^{n+1}) \\
    &\quad    
    - k\sum_{n=0}^l \mathbf{1}_{\Omega_{\kappa,n}} \sum_{i,j=1}^2 (a_{ij}^{n+1}-a_{h,ij}^{n+1}, v_j^{n+1}\xi_{v_i}^{n+1}) 
    + k\sum_{n=0}^l \mathbf{1}_{\Omega_{\kappa,n}} \sum_{i,j=1}^2 (a_{h,ij}^{n+1}\eta_{v_j}^{n+1}, \xi_{v_i}^{n+1}) \\    
    &\quad
    + \sum_{n=0}^l \int_{t_n}^{t_{n+1}} \mathbf{1}_{\Omega_{\kappa,n}} H^+(w_1(t)-w_1^{n+1},w_2(t)-w_2^{n+1},\xi_u^{n+1}) \,\mathrm{d}t
    + \sum_{n=0}^l \int_{t_n}^{t_{n+1}} \mathbf{1}_{\Omega_{\kappa,n}} D(f_1,f_2;u(t),u_h^n,\xi_u^{n+1}) \,\mathrm{d}t      \\
    &\quad
    + \sum_{n=0}^l \int_{t_n}^{t_{n+1}} \mathbf{1}_{\Omega_{\kappa,n}} \left((u_h^n)^\lambda-u(t)^\lambda,\xi_u^{n+1}\right) \,\mathrm{d}t 
    + \sum_{n=0}^l \mathbf{1}_{\Omega_{\kappa,n}} \left( \int_{t_n}^{t_{n+1}} (\psi_1(t)-\psi_{1,h}^n)\,\mathrm{d}t, \xi_u^{n+1} \right)     \\
    &\quad
    + \sum_{n=0}^l \mathbf{1}_{\Omega_{\kappa,n}} \left( \int_{t_n}^{t_{n+1}} (g(t)-g_h^n)\,\mathrm{d}W_t, \xi_u^{n+1}-\xi_u^n \right) 
    + \sum_{n=0}^l \mathbf{1}_{\Omega_{\kappa,n}} \left( \int_{t_n}^{t_{n+1}} (g(t)-g_h^n)\,\mathrm{d}W_t, \xi_u^n \right),
\end{align*}
where we use the result $L^-(\xi_u^{n+1},\xi_{w_1}^{n+1},\xi_{w_2}^{n+1}) + H^+(\xi_{w_1}^{n+1},\xi_{w_2}^{n+1},\xi_u^{n+1}) = 0$ by Lemma \ref{lem:num-flux-2D}. 

Using the nesting property of $\{\Omega_{\kappa,n} \}$ and hypothesis (ii), the left-hand side of the above equation satisfies
\begin{align*}
    \text{LHS} 
    &= \frac{1}{2}\sum_{n=0}^l \mathbf{1}_{\Omega_{\kappa,n}}(\| \xi_u^{n+1}\|^2 - \| \xi_u^{n}\|^2) 
    + \frac{1}{2} \sum_{n=0}^l \mathbf{1}_{\Omega_{\kappa,n}}\| \xi_u^{n+1} - \xi_u^{n}\|^2 
    + k\sum_{n=0}^l \mathbf{1}_{\Omega_{\kappa,n}} \mathcal{A}_h^{n+1} (\boldsymbol{\xi}_v^{n+1},\boldsymbol{\xi}_v^{n+1})  \\
    &\geq 
    \frac{1}{2}\mathbf{1}_{\Omega_{\kappa,l}}\| \xi_u^{l+1}\|^2 
    - \frac{1}{2}\| \xi_u^{0}\|^2 + \frac{1}{2} \sum_{n=0}^l \mathbf{1}_{\Omega_{\kappa,n}}\| \xi_u^{n+1} - \xi_u^{n}\|^2 
    + \alpha k\sum_{n=0}^l \mathbf{1}_{\Omega_{\kappa,n}} (\| \xi_{v_1}^{n+1}\|^2 + \| \xi_{v_2}^{n+1}\|^2).
\end{align*}
For an arbitrary $0\leq m \leq N_T-1$, we estimate the maximum energy and the full dissipation separately, and then add the two estimates. 
Using Lemma \ref{lem:xi_v^0-2d}, we obtain
\begin{align}
\label{ineq:error-main-2D}
    &\frac{1}{2} \max_{0 \leq n \leq m} \mathbf{1}_{\Omega_{\kappa,n}}\| \xi_u^{n+1}\|^2 + \frac{1}{2} \sum_{n=0}^{m} \mathbf{1}_{\Omega_{\kappa,n}}\| \xi_u^{n+1} - \xi_u^{n}\|^2 + \alpha k \sum_{n=0}^{m}\mathbf{1}_{\Omega_{\kappa,n}} (\| \xi_{v_1}^{n+1}\|^2 + \| \xi_{v_2}^{n+1}\|^2) \notag \\
    &\qquad
    \leq \|\xi_u^0\|^2 + 2(\mathcal{I}_1 + \mathcal{I}_2 + \mathcal{I}_3 + \mathcal{I}_4 + \mathcal{I}_5 + \mathcal{I}_6 + \mathcal{I}_7),
\end{align}
where
\begin{align*}
    \mathcal{I}_1 
    &:= \max_{0 \leq l \leq m} \Bigg| \sum_{n=0}^l\mathbf{1}_{\Omega_{\kappa,n}}(\eta_u^{n+1}-\eta_u^{n},\xi_u^{n+1}-\xi_u^n) 
     - k\sum_{n=0}^l \mathbf{1}_{\Omega_{\kappa,n}} \sum_{i=1}^2(\eta_{w_i}^{n+1},\xi_{v_i}^{n+1}) \\
    &\qquad  \qquad
     + \sum_{n=0}^l \mathbf{1}_{\Omega_{\kappa,n}} \left( \int_{t_n}^{t_{n+1}} \psi_1(t) - \psi_{1,h}^n \, \mathrm{d}t, \xi_u^{n+1} \right) 
     + \sum_{n=0}^l \mathbf{1}_{\Omega_{\kappa,n}} \left( \int_{t_n}^{t_{n+1}} g(t) - g_h^n \, \mathrm{d}W_t, \xi_u^{n+1}-\xi_u^n \right) \Bigg|, \\
    \mathcal{I}_2 
    &:= \max_{0 \leq l \leq m} \Bigg| \sum_{n=0}^l\mathbf{1}_{\Omega_{\kappa,n}}(\eta_u^{n+1}-\eta_u^{n},\xi_u^n) 
    +  \sum_{n=0}^l \mathbf{1}_{\Omega_{\kappa,n}} \left( \int_{t_n}^{t_{n+1}} g(t) - g_h^n \, \mathrm{d}W_t, \xi_u^n \right) \Bigg|, \\
    \mathcal{I}_3 
    &:= \max_{0\leq l\leq m} \Bigg|
    -k\sum_{n=0}^l \mathbf{1}_{\Omega_{\kappa,n}} \sum_{i,j=1}^2 (a_{ij}^{n+1}-a_{h,ij}^{n+1}, v_j^{n+1}\xi_{v_i}^{n+1})     
    + k\sum_{n=0}^l \mathbf{1}_{\Omega_{\kappa,n}} \sum_{i,j=1}^2 (a_{h,ij}^{n+1}\eta_{v_j}^{n+1}, \xi_{v_i}^{n+1}) \Bigg|,\\    
    \mathcal{I}_4 
    &:= \max_{0 \leq l \leq m} \Bigg| \sum_{n=0}^l \int_{t_n}^{t_{n+1}} \mathbf{1}_{\Omega_{\kappa,n}} H^+(w_1(t)-w_{1}^{n+1},w_2(t)-w_{2}^{n+1},\xi_u^{n+1}) \, \mathrm{d}t \Bigg|, \\
    \mathcal{I}_5 
    &:= \max_{0 \leq l \leq m} \Bigg| \sum_{n=0}^l \int_{t_n}^{t_{n+1}} \mathbf{1}_{\Omega_{\kappa,n}} D(f_1,f_2;u(t),u_h^n,\xi_u^{n+1}) \, \mathrm{d}t \Bigg|, \\
    \mathcal{I}_6 
    &:= \max_{0 \leq l \leq m} \sum_{n=0}^{l} \int_{t_n}^{t_{n+1}} \mathbf{1}_{\Omega_{\kappa,n}} \left((u_h^n)^{\lambda} - u(t)^{\lambda}, \xi_u^{n+1} \right) \, \mathrm{d}t , \\
    \mathcal{I}_7 
    &:= \max_{0 \leq l \leq m} \Bigg| 
    - k\sum_{n=0}^l \mathbf{1}_{\Omega_{\kappa,n}} L^-(\eta_u^{n+1},\xi_{w_1}^{n+1},\xi_{w_2}^{n+1}) 
    - k\sum_{n=0}^l \mathbf{1}_{\Omega_{\kappa,n}} H^+(\eta_{w_1}^{n+1},\eta_{w_2}^{n+1},\xi_u^{n+1}) \Bigg|. 
\end{align*}
Let $ \{ \epsilon_i >0: i=0,1,\ldots,12 \}$ be arbitrary constants. By the Cauchy-Schwarz inequality, H\"older inequality and Lemma \ref{lem:proj-property-2d}, we have
\begin{align}
\label{ineq:error-estimate-term1-2D}
    2\mathcal{I}_1 
    &\leq Ch^{2r+2}\sum_{n=0}^{m} \mathbf{1}_{\Omega_{\kappa,n}} \|u^{n+1}-u^n\|_{{r+1}}^2 
    + \epsilon_5 \sum_{n=0}^{m} \mathbf{1}_{\Omega_{\kappa,n}} \|\xi_u^{n+1}-\xi_u^n\|^2 
    + Ch^{2r+2} k\sum_{n=0}^{m} (\| w_1^{n+1} \|_{{r+1}}^2 + \| w_2^{n+1} \|_{{r+1}}^2) \notag \\
    &+ \epsilon_6 k \sum_{n=0}^{m} \mathbf{1}_{\Omega_{\kappa,n}} (\| \xi_{v_1}^{n+1}\|^2 + \| \xi_{v_2}^{n+1}\|^2) 
    + \frac{k}{\epsilon_7}\sum_{n=0}^{m} \mathbf{1}_{\Omega_{\kappa,n}} \| \xi_u^{n+1}\|^2 
    + \frac{1}{\epsilon_8} \sum_{n=0}^{m} \mathbf{1}_{\Omega_{\kappa,n}} \|\xi_u^{n+1}-\xi_u^n\|^2 \notag \\
    &+ \epsilon_7 \sum_{n=0}^{m} \mathbf{1}_{\Omega_{\kappa,n}} \int_{t_n}^{t_{n+1}} \| \psi_1(t) - \psi_{1,h}^n \|^2 \, \mathrm{d}t 
     + \epsilon_8 \sum_{n=0}^{m} \mathbf{1}_{\Omega_{\kappa,n}} \left \| \int_{t_n}^{t_{n+1}} g(t) - g_h^n \, \mathrm{d}W_t \right \|^2.
\end{align}
For $\mathcal{I}_2$, note that 
\[
\eta_u^{n+1}-\eta_u^n =
(\mathcal P^- - I)\left[
\int_{t_n}^{t_{n+1}} \bigl((w_1)_x+(w_2)_y - f_1(u)_x - f_2(u)_y + \psi \bigr)\,\mathrm{d}t
+ \int_{t_n}^{t_{n+1}} g\,\mathrm{d}W_t  \right].
\]

By H\"older inequality and Lemma \ref{lem:proj-property-2d}, we obtain
{\begin{align}
\label{ineq:error-estimate-term2-2D}
    &2\mathcal{I}_{2} 
    \leq Ch^{2r+2}\int_0^T \| (w_1)_x\|_{r+1}^2 + \| (w_2)_y\|_{r+1}^2 + \| f_1(u)_x\|_{r+1}^2 + \| f_2(u)_y\|_{r+1}^2 + \| \psi\|_{r+1}^2\, \mathrm{d}t + k\sum_{n=0}^{m} \mathbf{1}_{\Omega_{\kappa,n}} \| \xi_u^{n}\|^2 \notag \\
    &
    + \!2\!\max_{0 \leq l \leq m} \left| \sum_{n=0}^{l} \mathbf{1}_{\Omega_{\kappa,n}}\!\int_{t_n}^{t_{n+1}} \left(\xi_u^{n},(\mathcal{P}^- -I)(g \, \mathrm{d}W_t) \right) \right| 
    + \!2 \! \max_{0 \leq l \leq m} \left| \sum_{n=0}^{l} \mathbf{1}_{\Omega_{\kappa,n}}\!\left( \int_{t_n}^{t_{n+1}} g(t) - g_h^n \, \mathrm{d}W_t, \xi_u^{n} \right)\right|.
\end{align}
}
By \eqref{ineq:lem-error-term3-2D} of Lemma \ref{lem:error2D-term3},
\begin{align}
\label{ineq:error-estimate-term3-2D}
    2\mathcal{I}_{3} 
    &\leq \epsilon_2 k\sum_{n=0}^{m} \mathbf{1}_{\Omega_{\kappa,n}} (\| \xi_{v_1}^{n+1}\|^2 + \| \xi_{v_2}^{n+1}\|^2) 
    + Ch^{2r+2} k \sum_{n=0}^{m} (\| v_1^{n+1}\|_{{r+1}}^2 + \| v_2^{n+1}\|_{{r+1}}^2)  \notag \\
    &+ C\kappa k \sum_{n=0}^{m} \mathbf{1}_{\Omega_{\kappa,n}} \| \xi_u^{n+1}\|^2 
    + Ch^{2r+2} \kappa k \sum_{n=0}^{m} \mathbf{1}_{\Omega_{\kappa,n}} \| u^{n+1}\|_{{r+1}}^2. 
\end{align}
The estimates for $ \mathcal{I}_4, \mathcal{I}_5, \mathcal{I}_6$ are given in \eqref{ineq:error-estimate-term4-2D}, \eqref{ineq:error-estimate-term5-2D} and \eqref{ineq:error-estimate-term6-2d}. 
By Lemma \ref{lem:superconvergence}, for any $\epsilon_9 >0$,
\begin{align*}
    \mathcal{I}_7 
    &\leq k\sum_{n=0}^m \mathbf{1}_{\Omega_{\kappa,n}} |H^+(\eta_{w_1}^{n+1},\eta_{w_2}^{n+1},\xi_u^{n+1})| + k\sum_{n=0}^m \mathbf{1}_{\Omega_{\kappa,n}} |L^-(\eta_u^{n+1},\xi_{w_1}^{n+1},\xi_{w_2}^{n+1})| \\
    &\leq Ckh^{r+1}\sum_{n=0}^m \mathbf{1}_{\Omega_{\kappa,n}} (\| w_1^{n+1}\|_{r+2} + \| w_2^{n+1}\|_{r+2})\| \xi_u^{n+1}\| + Ckh^{r+1} \sum_{n=0}^m \mathbf{1}_{\Omega_{\kappa,n}} \| u^{n+1}\|_{r+2}(\| \xi_{w_1}^{n+1}\| + \| \xi_{w_2}^{n+1}\|) \\
    &\leq Ch^{2r+2}k \sum_{n=0}^m (\| w_1^{n+1}\|_{r+2}^2 + \| w_2^{n+1}\|_{r+2}^2+) 
    + k\sum_{n=0}^m \mathbf{1}_{\Omega_{\kappa,n}} \| \xi_u^{n+1}\|^2 \\ 
    & \qquad + Ch^{2r+2} k\sum_{n=0}^m \| u^{n+1}\|_{r+2}^2 
    + \epsilon_9 k\sum_{n=0}^m \mathbf{1}_{\Omega_{\kappa,n}} (\| \xi_{w_1}^{n+1}\|^2 + \| \xi_{w_2}^{n+1}\|^2).
\end{align*}
By Lemma \ref{lem:w_h-v_h-2D}, we obtain
\begin{align}
\label{ineq:error-estimate-term7-2D}
    \mathcal{I}_7 
    &\leq 14\Lambda \epsilon_9 k\sum_{n=0}^m \mathbf{1}_{\Omega_{\kappa,n}} (\| \xi_{v_1}^{n+1}\|^2 + \| \xi_{v_2}^{n+1}\|^2) 
    + C\kappa k\sum_{n=0}^m \mathbf{1}_{\Omega_{\kappa,n}} \| \xi_{u}^{n+1}\|^2 
    + C\kappa h^{2r+2} k\sum_{n=0}^m \| u^{n+1}\|_{r+2}^2 \notag \\
    &+ Ch^{2r+2} k\sum_{n=0}^m (\| w_1^{n+1}\|_{r+2}^2 + \| w_2^{n+1}\|_{r+2}^2 + \| v_1^{n+1}\|_{r+1}^2 + \| v_2^{n+1}\|_{r+1}^2).
\end{align}
Substituting \eqref{ineq:error-estimate-term1-2D}, \eqref{ineq:error-estimate-term2-2D}, \eqref{ineq:error-estimate-term3-2D}, \eqref{ineq:error-estimate-term4-2D}, \eqref{ineq:error-estimate-term5-2D}, \eqref{ineq:error-estimate-term6-2d}, \eqref{ineq:error-estimate-term7-2D} into \eqref{ineq:error-main-2D}, we get
\begin{align*}
    \frac{1}{2} & \max_{0 \leq n \leq m} \mathbf{1}_{\Omega_{\kappa,n}}\| \xi_u^{n+1}\|^2 
    + \left(\frac{1}{2} - \epsilon_4 - \epsilon_5 - \frac{1}{\epsilon_8} \right) \sum_{n=0}^{m} \mathbf{1}_{\Omega_{\kappa,n}}\| \xi_u^{n+1} - \xi_u^{n}\|^2 \\
    & \qquad
    + \left(\alpha - \epsilon_2 - \epsilon_3  - \epsilon_6 - 28\Lambda \epsilon_9 \right) k \sum_{n=0}^{m}\mathbf{1}_{\Omega_{\kappa,n}} \left(\| \xi_{v_1}^{n+1}\|^2 + \| \xi_{v_2}^{n+1}\|^2 \right) \\
    & \quad  
    \leq 
    C\kappa k \sum_{n=0}^{m} \mathbf{1}_{\Omega_{\kappa,n}} \| \xi_u^{n+1}\|^2 + Ch^{2r+2}\sum_{n=0}^{m} \mathbf{1}_{\Omega_{\kappa,n}} \|u^{n+1}-u^n\|_{{r+1}}^2 \\
    & \quad 
    + \epsilon_7 \sum_{n=0}^{m} \mathbf{1}_{\Omega_{\kappa,n}} \int_{t_n}^{t_{n+1}} \| \psi_1(t) - \psi_{1,h}^n \|^2 \, \mathrm{d}t 
    + \epsilon_8 \sum_{n=0}^{m} \mathbf{1}_{\Omega_{\kappa,n}} \left \| \int_{t_n}^{t_{n+1}} g(t) - g_h^n \, \mathrm{d}W_t \right \|^2 \\
    & \quad 
    + C\kappa \sum_{n=0}^{m} \int_{t_n}^{t_{n+1}} \mathbf{1}_{\Omega_{\kappa,n}} \bigl(\|u(t)-u^{n}\|^2 + \|u(t)-u^{n+1}\|_{2}^2 \bigr) \, dt 
    + C\kappa k^2\sum_{n=0}^{m} \bigl(\|v_1^{n+1}\|_{1}^2 + \|v_2^{n+1}\|_{1}^2 \bigr)\\
    & \quad 
    + 2\max_{0 \leq l \leq m} \sum_{n=0}^{l} \mathbf{1}_{\Omega_{\kappa,n}}\left( \int_{t_n}^{t_{n+1}} g(t) - g_h^n \, \mathrm{d}W_t, \xi_u^{n} \right) 
    + 2\max_{0 \leq l \leq m} \sum_{n=0}^{l} \mathbf{1}_{\Omega_{\kappa,n}}\int_{t_n}^{t_{n+1}} \left(\xi_u^{n},(\mathcal{P}^- -I)(g \, \mathrm{d}W_t) \right)  \\
    & \quad 
    + Ch^{2r+2}\int_0^T \| (w_1)_x\|_{r+1}^2 + \| (w_2)_y\|_{r+1}^2 + \| f_1(u)_x\|_{r+1}^2 + \| f_2(u)_y\|_{r+1}^2 + \| \psi\|_{r+1}^2\, \mathrm{d}t \\
    & \quad 
    + Ch^{2r+2} k\sum_{n=0}^m \bigl(\| w_1^{n+1}\|_{r+2}^2 + \| w_2^{n+1}\|_{r+2}^2 + \| v_1^{n+1}\|_{r+1}^2 + \| v_2^{n+1}\|_{r+1}^2 + \kappa \| u^{n+1}\|_{r+2}^2 \bigr).
\end{align*}
Substituting \eqref{ineq:error2D-psi} in Lemma \ref{lem:error2D-term1-2}, and taking the $q$-th power of both sides, then taking expectation, and applying the discrete H\"older inequality, Corollary \ref{coro:Holder-High-Moments}, Lemma \ref{lem:Holder-High-Moments} and \eqref{ineq:convex-3}, we get 
\begin{align}
\label{ineq:error-main-expect-2d}
    \frac{1}{2^q}&\mathbb{E}\left[\max_{0 \leq n \leq m} \mathbf{1}_{\Omega_{\kappa,n}}\| \xi_u^{n+1}\|^{2q} \right] 
    + \left(\frac{1}{2}-\epsilon_4-\epsilon_5-\frac{1}{\epsilon_8} \right)^q \mathbb{E}\left[ \left(\sum_{n=0}^{m} \mathbf{1}_{\Omega_{\kappa,n}}\| \xi_u^{n+1} - \xi_u^{n}\|^2 \right)^q \right] \notag \\
    & \quad 
    + \left(\alpha-\epsilon_2-\epsilon_3-\epsilon_6-28\Lambda \epsilon_9-(2+\epsilon_0) \epsilon_7 B_1^2 \right)^q \mathbb{E}\left[ \left(k\sum_{n=0}^{m}\mathbf{1}_{\Omega_{\kappa,n}} (\| \xi_{v_1}^{n+1}\|^2 + \| \xi_{v_2}^{n+1}\|^2) \right)^q \right] \notag \\
    &\leq C( \kappa h^{2r+2})^q + C\kappa^q \mathbb{E}\left[ \left(\sum_{n=0}^{m} \int_{t_n}^{t_{n+1}} \|u(t)-u^{n+1}\|_{2}^2 + \|u-u^n\|_{1}^2 \, \mathrm{d}t \right)^q \right] \notag \\
    & \quad 
    + C\kappa^q k\sum_{n=0}^{m} \mathbb{E}\left[\max_{0 \leq l \leq n}\mathbf{1}_{\Omega_{\kappa,l}} \|\xi_u^{l+1} \|^{2q} \right] 
    + C(\kappa k^{2})^q \mathbb{E}\left[ \left(\sum_{n=0}^{m}(1+\| u^n\|_{1}^2 + \| v_1^{n+1}\|_{1}^2 + \| v_2^{n+1}\|_{1}^2) \right)^q \right] \notag \\
    & \quad 
    + C_{\#}(\epsilon_{10})2^q\mathcal{T}_1 + \epsilon_8^q (2^{q-1}+\epsilon_{10})\mathcal{T}_2 + (2^{q-1}+\epsilon_{10}) 2^q\mathcal{T}_3 \notag \\
    &\leq C(\kappa k)^q + C(\kappa h^{2r+2})^q 
    + C\kappa^q k\sum_{n=0}^{m} \mathbb{E}\left[\max_{0 \leq l \leq n}\mathbf{1}_{\Omega_{\kappa,l}} \|\xi_u^{l+1} \|^{2q} \right] \notag \\
    & \quad 
    + C_{\#}(\epsilon_{10})2^q\mathcal{T}_1 + \epsilon_8^q (2^{q-1}+\epsilon_{10})\mathcal{T}_2 + (2^{q-1}+\epsilon_{10}) 2^q\mathcal{T}_3, 
\end{align}
where $C_{\#}(\epsilon_{10}) = (1+\epsilon_{10}) \left(1-2 \left(2^{q-1}+\epsilon_{10} \right)^{\frac{1}{1-q}} \right)^{1-q}$ is derived by \eqref{ineq:convex-1}, \eqref{ineq:convex-3}, and 
\begin{align*}
    \mathcal{T}_1 &:= \mathbb{E}\left[\max_{0 \leq l \leq m} \left|\sum_{n=0}^{l} \mathbf{1}_{\Omega_{\kappa,n}}\int_{t_n}^{t_{n+1}} \left(\xi_u^{n},(\mathcal{P}^- -I)(g \, \mathrm{d}W_t) \right) \right|^q \right], \\
    \mathcal{T}_2 &:= \mathbb{E}\left[ \left(\sum_{n=0}^{m} \mathbf{1}_{\Omega_{\kappa,n}} \left \| \int_{t_n}^{t_{n+1}} g(t) - g_h^n \, \mathrm{d}W_t \right \|^2 \right)^q \right], \\
    \mathcal{T}_3 &:= \mathbb{E}\left[\max_{0 \leq l \leq m} \left|\sum_{n=0}^{l} \mathbf{1}_{\Omega_{\kappa,n}}\left( \int_{t_n}^{t_{n+1}} g(t) - g_h^n \, \mathrm{d}W_t, \xi_u^{n} \right) \right|^q \right].
\end{align*}

\noindent {\it Step 2.} Next, we estimate the three stochastic terms $ \mathcal{T}_1,  \mathcal{T}_2,  \mathcal{T}_3$. By the BDG inequality and Lemma \ref{lem:subset-property}, we have 
\begin{align}
\label{ineq:error-estimate-t1-2D}
    \mathcal{T}_1 
    &\leq \mathbb{E}\left[\max_{0 \leq l \leq m} \left|\sum_{n=0}^{l} \mathbf{1}_{\Omega_{\kappa,n-1}}\int_{t_n}^{t_{n+1}} \left(\xi_u^{n},(\mathcal{P}^- -I)(g \, \mathrm{d}W_t) \right) \right|^q \right] \notag \\
    &\leq C_b \mathbb{E}\left[ \left( \sum_{n=0}^{m} \int_{t_n}^{t_{n+1}} \mathbf{1}_{\Omega_{\kappa,n-1}} \|\xi_u^{n} \|^2 \sum_{j=1}^{\infty} \gamma_j \|(\mathcal P^- - I)(g e_j)\|^2\, \mathrm{d}t \right)^\frac{q}{2} \right] \notag \\
    &\leq \epsilon_{11} \mathbb{E}\left[ \max_{0 \leq n \leq m} \mathbf{1}_{\Omega_{\kappa,n-1}} \|\xi_u^{n} \|^{2q} \right] 
    + \frac{C_b^2}{4\epsilon_{11}} \mathbb{E}\left[ \left( \sum_{n=0}^{m} \int_{t_n}^{t_{n+1}} \mathbf{1}_{\Omega_{\kappa,n-1}} \sum_{j=1}^{\infty} \gamma_j \|(\mathcal P^- - I)(g e_j)\|^2 \, \mathrm{d}t \right)^q \right] \notag \\
    &\leq \epsilon_{11} \mathbb{E}\left[ \max_{0 \leq n \leq m} \mathbf{1}_{\Omega_{\kappa,n}} \|\xi_u^{n+1} \|^{2q} \right] + C(h^{2r+2})^q,
\end{align}
where, in the last inequality, the initial term generated by the index shift is absorbed into $C(h^{2r+2})^q$.
By the same argument as in \eqref{ineq:stability-I1-2D}, utilizing the nesting property $\Omega_{\kappa,n} \subset \Omega_{\kappa,n-1}$, and applying the continuous and discrete BDG inequalities yield
\begin{align*}
    \mathcal{T}_{2} 
    &\leq C_b'C_b K^q \mathbb{E}\left[ \left( \sum_{n=0}^{m} \int_{t_n}^{t_{n+1}} \mathbf{1}_{\Omega_{\kappa,n-1}} \|g(t) - g_h^n \|^2 \, \mathrm{d}t \right)^{q} \right]. 
\end{align*}
By \eqref{ineq:error2D-g} of lemma \ref{lem:error2D-term1-2}, and applying \eqref{ineq:convex-1} with $\epsilon = \epsilon_1$, we get
\begin{align}
\label{ineq:error-estimate-t2-2D}
    \mathcal{T}_{2} &\leq C(k^q + (h^{2r+2})^q) + C k\sum_{n=0}^{m} \mathbb{E}\left[ \max_{0 \leq l \leq n}\mathbf{1}_{\Omega_{\kappa,l}}\| \xi_u^{l+1}\|^{2q} \right] \notag \\
    &\quad 
    + C_b'C_b K^q D_2^{2q}(2+\epsilon_{1})^q (1+\epsilon_{1}) \mathbb{E}\left[ \left(k\sum_{n=0}^{m}\mathbf{1}_{\Omega_{\kappa,n}} (\| \xi_{v_1}^{n+1}\|^2 + \| \xi_{v_2}^{n+1}\|^2) \right)^q \right].
\end{align}
Lastly, by the BDG inequality, Lemma \ref{lem:subset-property} and \eqref{ineq:convex-1}, we have
\begin{align}
\label{ineq:error-estimate-t3-2D}
    \mathcal{T}_{3} 
    &\leq \mathbb{E}\left[\max_{0 \leq l \leq m} \left|\sum_{n=0}^{l} \mathbf{1}_{\Omega_{\kappa,n-1}}\left( \int_{t_n}^{t_{n+1}} g(t) - g_h^n \, \mathrm{d}W_t, \xi_u^{n} \right) \right|^q \right]  \notag \\ 
    &\leq C_b \mathbb{E}\left[ \left( K\sum_{n=0}^{m} \int_{t_n}^{t_{n+1}} \mathbf{1}_{\Omega_{\kappa,n-1}}\|g(t) - g_h^n \|^2 \| \xi_u^{n}\|^2 \, \mathrm{d}t \right)^\frac{q}{2} \right] \notag \\
    &\leq C_b^2 K^q \epsilon_{12}\mathbb{E}\left[ \left(\sum_{n=0}^{m} \int_{t_n}^{t_{n+1}} \mathbf{1}_{\Omega_{\kappa,n-1}}\|g(t) - g_h^n\|^2 \, \mathrm{d}t  \right)^q \right] 
    + \frac{1}{4\epsilon_{12}}\mathbb{E}\left[ \max_{0 \leq n \leq m} \mathbf{1}_{\Omega_{\kappa,n-1}}\| \xi_u^{n}\|^{2q} \right] \notag \\
    & \leq \frac{1}{4\epsilon_{12}}\mathbb{E}\left[ \max_{0 \leq n \leq m} \mathbf{1}_{\Omega_{\kappa,n}}\| \xi_u^{n+1}\|^{2q} \right] 
    + Ck\sum_{n=0}^{m} \mathbb{E}\left[ \max_{0 \leq l \leq n}\mathbf{1}_{\Omega_{\kappa,l}}\| \xi_u^{l+1}\|^{2q} \right] 
    + C(k^q + (h^{2r+2})^q) \notag \\
    &\quad + C_b^2 K^q \epsilon_{12} D_2^{2q}(2+\epsilon_{1})^q (1+\epsilon_{1}) \mathbb{E}\left[ \left(k\sum_{n=0}^{m}\mathbf{1}_{\Omega_{\kappa,n}} (\| \xi_{v_1}^{n+1}\|^2 + \| \xi_{v_2}^{n+1}\|^2) \right)^q \right] .
\end{align}
Plugging the estimates \eqref{ineq:error-estimate-t1-2D}, \eqref{ineq:error-estimate-t2-2D}, \eqref{ineq:error-estimate-t3-2D} into \eqref{ineq:error-main-expect-2d}, we get
\begin{align}
\label{ineq:high-moment-error-strong-final-2d}
    \mathcal{C}_1 &\mathbb{E}\left[\max_{0 \leq n \leq m} \mathbf{1}_{\Omega_{\kappa,n}}\| \xi_u^{n+1}\|^{2q} \right]  
    + \mathcal{C}_2 \mathbb{E}\left[ \left(\sum_{n=0}^{m} \mathbf{1}_{\Omega_{\kappa,n}}\| \xi_u^{n+1} - \xi_u^{n}\|^2 \right)^q \right] 
    + \mathcal{C}_3 \mathbb{E}\left[ \left(k\sum_{n=0}^{m}\mathbf{1}_{\Omega_{\kappa,n}} (\| \xi_{v_1}^{n+1}\|^2 + \| \xi_{v_2}^{n+1}\|^2) \right)^q \right] \notag \\
    &\leq C(\kappa h^{2r+2})^q + C(\kappa k)^q + C\kappa^q k\sum_{n=0}^{m} \mathbb{E}\left[ \max_{0 \leq l \leq n}\mathbf{1}_{\Omega_{\kappa,l}}\| \xi_u^{l+1}\|^{2q} \right],
\end{align}
where the three coefficients of the left-hand side terms are given by 
\begin{align*}
    \mathcal{C}_1 &= \frac{1}{2^q}-C_{\#}(\epsilon_{10})2^q \epsilon_{11}-\frac{2^q(2^{q-1}+\epsilon_{10})}{4\epsilon_{12}}, \quad  \mathcal{C}_2 = \left(\frac{1}{2}-\epsilon_4- \epsilon_5-\frac{1}{\epsilon_8} \right)^q,  \\
    \mathcal{C}_3 &= \left(\alpha-\epsilon_2-\epsilon_3-\epsilon_6-28\Lambda \epsilon_9-(2+\epsilon_0) \epsilon_7 B_1^2 \right)^q - \epsilon_8^q (2^{q-1}+\epsilon_{10}) C_b'C_b K^q D_2^{2q}(1+\epsilon_1)(2+\epsilon_{1})^{q} \\
    &- (2^{q-1}+\epsilon_{10}) 2^qC_b^2 K^q \epsilon_{12} D_2^{2q}(1+\epsilon_1)(2+\epsilon_{1})^{q}. 
\end{align*}
By the same argument as in the proof of Theorem \ref{thm:stability-estimate-2D} (see Appendix \ref{appendix-eps-choice}) for choosing the $\epsilon_i$'s, we can find fixed constants $\{ \epsilon_i >0: i=0,1,...,12\}$, such that $\mathcal{C}_1, \mathcal{C}_2, \mathcal{C}_3>0$ if and only if $\alpha > \left(2^{3q-1}C_b' + 2^{6q-4}C_b\right)^{1/q}C_b^{1/q}D_2^2K$. Applying the discrete Gr\"onwall's inequality, and rearranging the coefficients, we get
\begin{align*}
    \mathbb{E}\left[\max_{0 \leq n \leq m} \mathbf{1}_{\Omega_{\kappa,n}}\| \xi_u^{n+1}\|^{2q} \right] 
    &+ \mathbb{E}\left[ \left(\sum_{n=0}^{m} \mathbf{1}_{\Omega_{\kappa,n}}\| \xi_u^{n+1} - \xi_u^{n}\|^2 \right)^q \right] + \mathbb{E}\left[ \left(k\sum_{n=0}^{m}\mathbf{1}_{\Omega_{\kappa,n}} (\| \xi_{v_1}^{n+1}\|^2 + \| \xi_{v_2}^{n+1}\|^2) \right)^q \right] \\
    &\leq Ce^{C\kappa^q}\left(k^q + h^{q(2r+2)}  \right) 
    \leq Ch^{-C_q\beta} \left(k^{\frac{1}{2}} + h^{r+1} \right)^{2q},
\end{align*}
where $C_q$ is a fixed constant depending on $q$. 
In the last inequality, we used the choice $\kappa=\ln\bigl(\ln(h^{-\beta})\bigr)$ with arbitrary $\beta>0$, and the fact that logarithmic growth is dominated by any negative power of $h$ as $h\to0$.
Taking the $2q$-th root of the above inequality with $m = N_T-1$ gives
\begin{align*}
    \mathbb{E}&\left[\max_{0 \leq n \leq N_T-1} \mathbf{1}_{\Omega_{\kappa,n}}\| \xi_u^{n+1}\|^{2q} \right]^{\frac{1}{2q}} 
    + \mathbb{E}\left[ \left(\sum_{n=0}^{N_T-1} \mathbf{1}_{\Omega_{\kappa,n}}\| \xi_u^{n+1} - \xi_u^{n}\|^2 \right)^q \right]^{\frac{1}{2q}}  \\ 
    &+ \mathbb{E}\left[ \left(k\sum_{n=0}^{N_T-1}\mathbf{1}_{\Omega_{\kappa,n}} (\| \xi_{v_1}^{n+1}\|^2 + \| \xi_{v_2}^{n+1}\|^2) \right)^q \right]^{\frac{1}{2q}} \leq C h^{-C_q^*\beta} \left(k^{\frac{1}{2}} + h^{r+1} \right),
\end{align*}
where $C_q^* = C_q/(2q)$. 
Finally, \eqref{ineq:high-moment-error-estimate-2D} follows from the triangle inequality and the projection estimate \eqref{ineq:proj-property-1-2d}.

\noindent {\it Step 3.} Lastly, we prove that $\mathbb{P}(\Omega_{\kappa,N_T})\to1$ as $h\to0$. Take $q\geq p_0:=\max\{p,\lambda-1\}$ and assume $k\leq Ch^{2+2C_q^*\beta}$. By the definition of $\Omega_{\kappa,N_T}$,
\begin{align*}
    \mathbb{P}(\Omega_{\kappa,N_T}^c)
    &\leq
    \mathbb{P}\left( \sup_{0\leq t\leq T} \left(1+\|u(t)\|_{W^{2,\infty}}^2 +\|u(t)\|_\infty^{2p_0}\right) > \frac{\kappa}{2} \right) 
    + \mathbb{P}\left( \max_{0\leq n\leq N_T} \mathbf{1}_{\Omega_{\kappa,n-1}} \left(1+\|u_h^n\|_\infty^{2p_0}\right) > \frac{\kappa}{2} \right).
\end{align*}
By Markov's inequality and assumption $(\mathcal{A}_q)$, the first probability is bounded by $C/\kappa$.
For the second probability, by $u_h^n=u^n-\xi_u^n+\eta_u^n$, the triangle inequality, and the projection estimate, 
\begin{align*}
    &\mathbb{P}\left( \max_{0\leq n\leq N_T} \mathbf{1}_{\Omega_{\kappa,n-1}} \left(1+\|u_h^n\|_\infty^{2p_0}\right) > \frac{\kappa}{2} \right) \\
    &\qquad
    \leq \mathbb{P}\left( \sup_{0\leq t\leq T}\|u(t)\|_\infty^{2p_0} > c\kappa \right) 
    + \mathbb{P}\left( \max_{0\leq n\leq N_T} \mathbf{1}_{\Omega_{\kappa,n-1}} \|\xi_u^n\|_\infty^{2p_0} > c\kappa \right) 
    + \mathbb{P}\left( \max_{0\leq n\leq N_T} \|\eta_u^n\|_\infty^{2p_0} > c\kappa \right).
\end{align*}
The first term is bounded by $C/\kappa$. The $L^\infty$ projection estimate and $(\mathcal A_q)$ bound the third term by $Ch^{2rp_0}/\kappa$. For the second term, using $q\geq p_0$, the inverse inequality, and the result from Step 2, we have 
\begin{align*}
    &\mathbb{P}\left( \max_{0\leq n\leq N_T} \mathbf{1}_{\Omega_{\kappa,n-1}} \|\xi_u^n\|_\infty^{2p_0} > c\kappa \right)\leq
    \frac{C}{\kappa^{q/p_0}} \mathbb{E}\left[ \max_{0\leq n\leq N_T} \mathbf{1}_{\Omega_{\kappa,n-1}} \|\xi_u^n\|_\infty^{2q} \right] \\
    &\quad\leq
    \frac{C}{\kappa^{q/p_0}} h^{-2q} \mathbb{E}\left[ \|\xi_u^0\|^{2q} 
    + \max_{0\leq n\leq N_T-1} \mathbf{1}_{\Omega_{\kappa,n}} \|\xi_u^{n+1}\|^{2q} \right] 
    \leq
    \frac{C}{\kappa^{q/p_0}} h^{-2q-2qC_q^*\beta} \left(k^{1/2}+h^{r+1}\right)^{2q} + \frac{C}{\kappa^{q/p_0}}h^{2qr}.
\end{align*}
Under $k\leq Ch^{2+2C_q^*\beta}$, the right-hand side is bounded by $C/\kappa^{q/p_0}$. Therefore, 
\begin{align*}
     \mathbb{P}(\Omega_{\kappa,N_T}^c) \leq \frac{C}{\kappa} + \frac{C}{\kappa^{q/p_0}} \leq \frac{C}{\ln\bigl(\ln(h^{-\beta})\bigr)}, 
\end{align*}
which tends to zero as $h\to0$. Hence, provided $k\leq Ch^{2+2C_q^*\beta}$, we have
\[
    \mathbb{P}(\Omega_{\kappa,N_T}) \geq 1 - \frac{C}{\ln\bigl(\ln(h^{-\beta})\bigr)} \to1     \qquad\text{as } h\to0,
\]
which means that $\Omega_{\kappa,N_T}$ converges to $\Omega$ in probability with rate at least $\mathcal{O}\bigl(1/\ln\bigl(\ln(h^{-\beta})\bigr) \bigl)$.
As a consequence, the localized estimate implies convergence in probability on the full sample space. Let
\[ E_{h,k} = \max_{0 \leq n \leq N_T} \| e_u^{n}\|^{2q} + \left(\sum_{n=0}^{N_T-1} \| e_u^{n+1} - e_u^{n}\|^2 \right)^q + \left(k\sum_{n=0}^{N_T} (\|e_{v_1}^{n}\|^2 + \|e_{v_2}^{n}\|^2) \right)^q. \]
For any fixed $\gamma > 0$, by Markov inequality
\begin{align*}
     \mathbb{P}(E_{h,k} > k^{q-\gamma} + h^{2q(r+1)-\gamma}) &\leq \mathbb{P}(\Omega_{\kappa,N_T}^c) + 
        \mathbb{P}(\mathbf{1}_{\Omega_{\kappa,N_T}} \cdot ( E_{h,k} > k^{q-\gamma} + h^{2q(r+1)-\gamma})) \\
     &\leq \mathbb{P}(\Omega_{\kappa,N_T}^c) + C(k^{\gamma} + h^{\gamma})/h^{C_q \beta} \to 0. 
\end{align*}
Thus the numerical error converges in probability with temporal and spatial orders arbitrarily close to $1/2$ and $r+1$.
\end{proof}

As a result of Lemma \ref{lem:Kolmogorov} and Theorem \ref{thm:high-moment-error-estimate-2D}, we establish the pathwise error estimate on subsets. 
\begin{corollary}[Pathwise Error Estimate]
\label{coro:pathwise-error-2d}
Let $\{(h_\ell,k_\ell)\}_{\ell\geq N_0}$ be a refinement sequence satisfying $k_\ell+h_\ell^{2r+2}\leq C/\ell$.
Let $q > 1$ and $0 < \zeta < {1}/{2} - {1}/(2q)$, there exists a random variable $Z(\omega,\zeta)$ satisfying $\mathbb{E}[|Z|^{2q}] < \infty$ such that the following holds almost surely:
\begin{align*}
    \max_{0 \leq n \leq N_T} \mathbf{1}_{\Omega_{\kappa,n-1}}\| e_u^{n}\| &+ \left( \sum_{n=0}^{N_T-1} \mathbf{1}_{\Omega_{\kappa,n}}\| e_u^{n+1} - e_u^{n}\|^2\right)^{\frac{1}{2}} + \left(k\sum_{n=0}^{N_T}\mathbf{1}_{\Omega_{\kappa,n-1}} (\|e_{v_1}^{n}\|^2 + \|e_{v_2}^{n}\|^2) \right)^\frac{1}{2} \\
    &\leq Z h_\ell^{-C_q^*\beta}(k_\ell^{\zeta} + h_\ell^{(2r+2)\zeta}),
\end{align*}
for every $\ell\geq N_0$, where $\beta>0$ is an arbitrarily small constant.     
\end{corollary}   

Two special cases of Theorem \ref{thm:high-moment-error-estimate-2D} are summarized in the following remarks. 
\begin{remark}[Simplified $g$]
\label{rem:simplified-g}    
    Suppose $g(\omega,x,t,u,v_1,v_2)$ does not depend on the last two variables, that is, $D_2 = D_4 = 0$ in hypothesis {\rm(v)}, then the stability and error estimates hold for any $\alpha > 0$. 
\end{remark}


\begin{remark}[No nonlinear convection term or polynomial source term]
\label{rem:error-no-nonlinear-term-2d}
    Suppose \eqref{spde:conv-diff-2d} has no nonlinear convection term and no polynomial source term, that is, $f=0$ and $\psi = \psi_1$. Then we can show the high-moment error estimate \eqref{ineq:high-moment-error-estimate-2D} without introducing the subsets \eqref{def:error-subset-2d}. More precisely, for $q \geq 1$,
    \begin{align*}
        \mathbb{E}\left[\max_{0 \leq n \leq N_T} \| e_u^{n}\|^{2q} \right]^{\frac{1}{2q}} +       \mathbb{E}\left[ \left(\sum_{n=0}^{N_T-1} \| e_u^{n+1} - e_u^{n}\|^2 \right)^q \right]^{\frac{1}{2q}} + \mathbb{E}\left[ \left(k\sum_{n=0}^{N_T} (\|e_{v_1}^{n}\|^2 + \|e_{v_2}^{n}\|^2) \right)^q \right]^{\frac{1}{2q}} 
        \leq C(k^{\frac{1}{2}} + h^{r+1}),
    \end{align*}
    with the same lower bound on $\alpha$ as in Theorem \ref{thm:high-moment-error-estimate-2D}. 
    For the second-moment case $q=1$, the bound is relaxed to $\alpha >KD_2^2$. We refer to \cite{chen_semi-linear_2d} for the detailed
two-dimensional semilinear analysis.
\end{remark}

\begin{remark}[{Comparison with one-dimensional error estimate}]
\label{rmk:compare-to-1d-error-estimate}
    The error estimate for the one-dimensional version of \eqref{spde:conv-diff} can be proved with relaxed hypotheses. We may assume that the leading coefficients $\{a_{ij}(\cdot,x,y,t,u)\}$ grow at most linearly in $u$, instead of being uniformly bounded in {\rm(ii)}. In two-dimensional case, the uniform boundedness condition is used to establish the relation between $\| \xi_{v}^n\|$ and $\| \xi_{w}^n\|$ in Lemma \ref{lem:w_h-v_h-2D}, which is proposed to estimate the non-vanishing term $\mathcal{I}_7$ in Theorem \ref{thm:high-moment-error-estimate-2D} due to the properties of the two-dimensional projection errors. 
\end{remark}

We close with a comment on the additional technical difficulties in proving Theorem \ref{thm:high-moment-error-estimate-2D} for a fully implicit LDG-Euler scheme.  
\begin{remark}[Error Estimate for an Implicit Treatment of the Nonlinear Terms]
    The current proof does not directly extend to the fully implicit treatment of the nonlinear convection and polynomial source terms.
    The subsets $\Omega_{\kappa,n}$ in \eqref{def:error-subset-2d} control $\|u_h^n\|_\infty^{2p_0}$, whereas a fully implicit discretization would produce terms involving $\|u_h^{n+1}\|_\infty^{2p_0}$. 
    Since the convergence in probability of these subsets is proved using the conclusion of Theorem \ref{thm:high-moment-error-estimate-2D}, this creates an issue if the same subsets are used. 
    One possible way to avoid this difficulty is to introduce a different family of subsets based on a stability estimate on the full sample space. We leave this issue for future work.
\end{remark}  

\end{section}

\begin{section}{Numerical Tests}
\label{sec:numerical-test}
In this section, we apply the proposed numerical method to some classic stochastic models. We use Monte Carlo method to approximate moments of the numerical error. Denoting by $M$ the number of samples, we have 
\begin{gather*}
     e_2^2 := \frac{1}{M} \sum_{i=1}^M \max_{0 \leq n \leq N_T} \| u_h^n(\omega_i, \cdot) - u(\omega_i, \cdot, t_n) \|^2 
     \approx \mathbb{E}\left[\max_{0 \leq n \leq N_T} \| u(\omega,\cdot,t_n) - u_h^n(\omega,\cdot)\|^2 \right].  
\end{gather*}
We thus use $e_2$ to approximate the root mean-square error in $L^2(\Omega;L^\infty(0,T;L^2(\mathcal{D})))$. To investigate the convergence of higher moments, we compute 
\[
     e_4 := \left( \frac{1}{M}\sum_{i=1}^M \max_{0\leq n\leq N_T} \|u_h^n(\omega_i,\cdot)-u(\omega_i,\cdot,t_n)\|^4 \right)^{1/4},
     \quad
     e_\infty := \max_{1\leq i\leq M} \max_{0\leq n\leq N_T} \|u_h^n(\omega_i,\cdot)-u(\omega_i,\cdot,t_n)\|,
\]
which approximates the error in $L^4(\Omega;L^\infty(0,T;L^2(\mathcal{D})))$ and $L(\Omega,L^{\infty}([0,T] \times L^2(\mathcal{D})))$, respectively. 

In all numerical simulations, $r$ denotes the degree of the piecewise polynomial space $\mathbb{V}_h$ and $T$ is the terminal time. We apply periodic boundary conditions, with $N$ being the number of spatial cells on each side of the rectangular domain, and $N_T$ being the number of time steps. We take $N_T \sim N$ to verify the temporal convergence rate, and $N_T \sim N^4$, $N_T \sim N^6$ to compute the spatial rate for $r = 1,2$ respectively. Unless otherwise specified, $W_t$ is the standard Brownian motion.

\begin{subsection}{One-dimensional Stochastic Viscous Burgers' Equation}
The first example that we consider is the stochastic viscous Burgers' equation 
\begin{gather}
    \begin{cases}
    \label{spde:sburgers-1d}
        \mathrm{d}u = \left( \dfrac{a^2}{2}u_{xx} - \dfrac{1}{2}(u^2)_x \right) \,\mathrm{d}t + (au_x +b) \mathrm{d}W_t, \quad & (\omega,x,t)\in \Omega \times [0,2\pi]\times (0,T], \\
        u(\omega,x,0) = \sin(x),  & (\omega,x) \in \Omega \times [0,2\pi]. 
    \end{cases}
\end{gather}
The exact solution of \eqref{spde:sburgers-1d} is given by
\[ u(\omega,x,t) = v\left(x + aW_t - b\int_0^t W_s\, \mathrm{d}s,t \right) + bW_t, \]
where $v$ is the solution of the deterministic inviscid Burgers' equation
\begin{gather}
    \begin{cases}
    \label{pde:burgers-1d}
        v_t + \frac{1}{2} (v^2)_{x} =0,  \quad & (x,t)\in [0,2\pi]\times (0,T], \\
        v(x,0) = \sin(x),  &  x \in   [0,2\pi].
    \end{cases}
\end{gather}
Note that the exact solution of \eqref{pde:burgers-1d} forms a shock at $T_0 = 1$, so the exact solution of \eqref{spde:sburgers-1d} also develops a shock at $T_0$. For $t < T_0$, $v$ is smooth and is implicitly given by $v(x,t) = \sin(x-tv)$. 

For the accuracy test, we fix $a = 1$, $b = 0$ and run the simulation up to $T = 0.9$ with $M=2000$ paths. From Tables \ref{table:sburgers-1d-timerate}-\ref{table:sburgers-1d-spacerate}, we can observe that, at $T=0.5$, both the errors $e_2$ and $e_4$  achieve the optimal convergence rate of $(r+1)$-th in space and $0.5$-th in time, which is consistent with the high-moment error estimate in Theorem \ref{thm:high-moment-error-estimate-2D} for the fully-discrete scheme. Moreover, Table \ref{table:sburgers-1d-timerate} compares the error and temporal convergence rate as time approaches $T_0$. We see that the error increases and the order of accuracy deteriorates, as $T$ increases. Finer choices of $N_T/N$ and a larger number of sample paths may be needed to clearly observe the theoretical half-order rate near $T_0$.

Next, to investigate the behavior of the numerical solution before and after shock formation, we plot the snapshots of the numerical solution along one fixed path at $T=0.6,0.9,1.2,1.5$. In Fig. \ref{fig:sburgers-snapshot-1d}, the three rows of solutions are plotted at different noise level $a=0,0.5,1$, with $a=0$ corresponding to the deterministic equation \eqref{pde:burgers-1d}. We observe the steepening of the wave from $T=0.9$ to $1.2$, suggesting shock formation. The plots after $T = 0.9$ show that the random shift effect driven by Brownian motion displaces the local extrema of the sine wave. The displacement is clearly more pronounced for higher noise intensity. Moreover, at $T=1.2$ and $1.5$, oscillations near the shock are increasingly smeared by numerical diffusion for larger $a$.
Nonlinear limiters could be incorporated to reduce these post-shock oscillations.

Lastly, even with the boundedness of exact solution, we cannot prove the numerical solution is almost surely bounded uniformly on the full sample space. Therefore, the analytic error estimate is still formulated on subsets of the sample space. Numerically, we have not observed any pathwise blow-up and we have taken $k\sim h$ in some simulations.

\begin{table}[htb]
\centering
\renewcommand\arraystretch{1.25}
\caption{Numerical error and temporal convergence rate for stochastic Burgers' Equation. \eqref{spde:sburgers-1d} with $r=1$.}
\begin{tabular}{c c c c c c c c}
\hline\hline
$T$ & $N$ & $e_2$ & rate & $e_4$ & rate & $e_{\infty}$ & rate \\
\hline\hline
\multirow{5}{*}{0.1}
& 10  & 2.18E-02 & --   & 2.58E-02 & --   & 3.78E-02 & -- \\
& 20  & 6.09E-03 & 1.84 & 6.25E-03 & 2.05 & 1.14E-02 & 1.73 \\
& 40  & 2.31E-03 & 1.40 & 2.68E-03 & 1.22 & 4.21E-03 & 1.44 \\
& 80  & 1.33E-03 & 0.80 & 1.71E-03 & 0.65 & 1.99E-03 & 1.08 \\
& 160 & 9.00E-04 & 0.56 & 1.18E-03 & 0.54 & 1.16E-03 & 0.78 \\
\hline
\multirow{5}{*}{0.5}
& 10  & 4.05E-02 & --   & 4.33E-02 & --   & 1.02E-01 & -- \\
& 20  & 1.87E-02 & 1.12 & 2.23E-02 & 0.95 & 4.68E-02 & 1.13 \\
& 40  & 1.17E-02 & 0.68 & 1.51E-02 & 0.57 & 2.47E-02 & 0.92 \\
& 80  & 8.01E-03 & 0.55 & 1.04E-02 & 0.53 & 1.47E-02 & 0.75 \\
& 160 & 5.62E-03 & 0.51 & 7.33E-03 & 0.51 & 9.63E-03 & 0.61 \\
\hline
\multirow{5}{*}{0.9}
& 10  & 1.26E-01 & --   & 1.33E-01 & --   & 3.69E-01 & --  \\
& 20  & 7.79E-02 & 0.70 & 8.52E-02 & 0.64 & 2.75E-01 & 0.42 \\
& 40  & 5.39E-02 & 0.53 & 6.14E-02 & 0.47 & 2.31E-01 & 0.25 \\
& 80  & 3.95E-02 & 0.45 & 4.74E-02 & 0.37 & 1.90E-01 & 0.28 \\
& 160 & 2.98E-02 & 0.41 & 3.65E-02 & 0.38 & 1.55E-01 & 0.30 \\
\hline\hline
\end{tabular}
 \label{table:sburgers-1d-timerate}
\end{table}

\begin{table}[htb]
\centering
\renewcommand\arraystretch{1.25}
\caption{Numerical error and spatial convergence rate for stochastic Burgers' Equation. \eqref{spde:sburgers-1d} with $T=0.5$.}
\begin{tabular}{c c c c c c c c}
\hline\hline
$r$ & $N$ & $e_2$ & rate & $e_4$ & rate & $e_{\infty}$ & rate \\
\hline\hline
\multirow{4}{*}{1}
& 10 & 1.08E-01 & --   & 1.37E-01 & --   & 2.25E-01 & -- \\
& 20 & 2.75E-02 & 1.97 & 3.35E-02 & 2.03 & 6.36E-02 & 1.82 \\
& 40 & 6.64E-03 & 2.05 & 8.32E-03 & 2.01 & 1.57E-02 & 2.02 \\
& 80 & 1.71E-03 & 1.96 & 2.14E-03 & 1.96 & 3.99E-03 & 1.97 \\
\hline
\multirow{3}{*}{2}
& 10 & 1.04E-01 & --   & 1.37E-01 & --   & 2.08E-01 & -- \\
& 20 & 1.29E-02 & 3.01 & 1.67E-02 & 3.03 & 2.36E-02 & 3.14 \\
& 40 & 1.60E-03 & 3.01 & 2.08E-03 & 3.01 & 2.83E-03 & 3.06 \\
\hline\hline
\end{tabular}
 \label{table:sburgers-1d-spacerate}
\end{table}


\begin{figure}[htb]
	\begin{center}
        \includegraphics[scale=0.6]{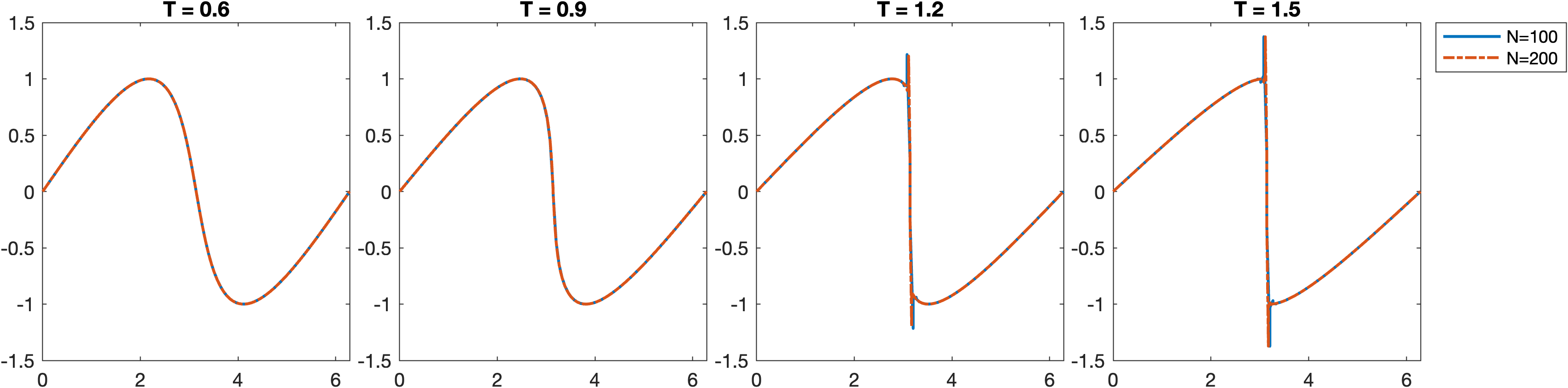}
        \includegraphics[scale=0.6]{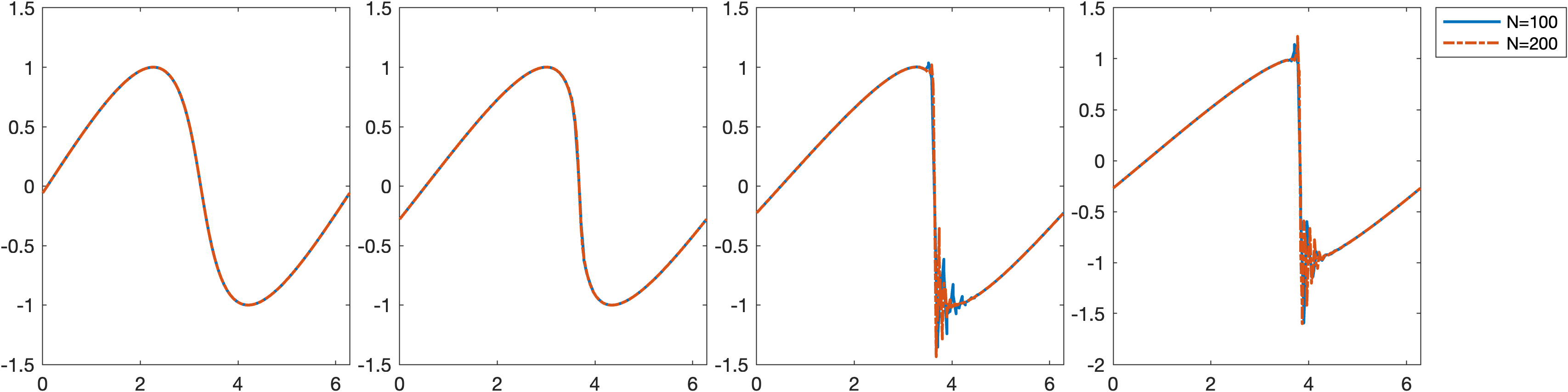}
        \includegraphics[scale=0.6]{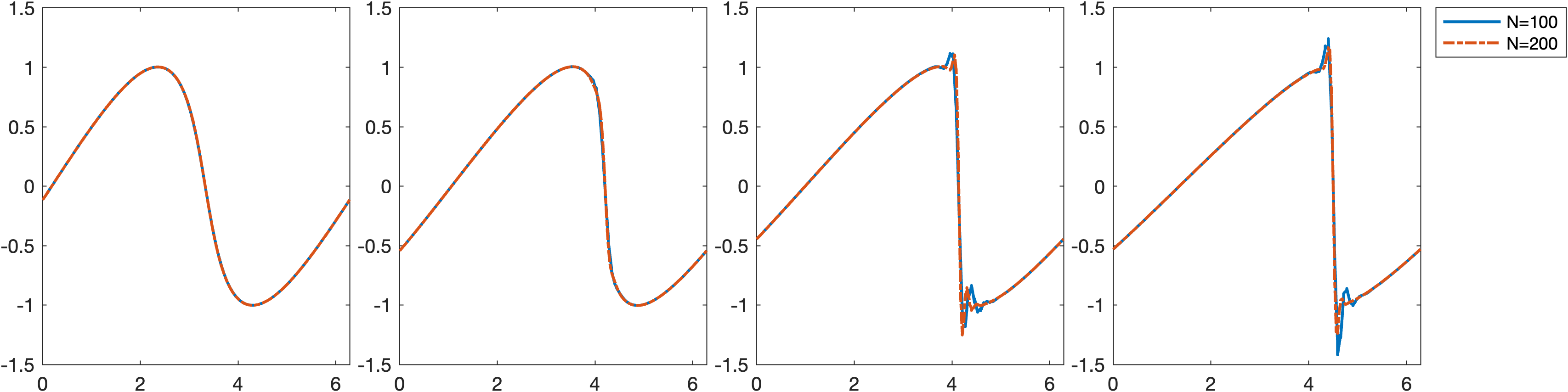}
		\caption{Solution snapshot of 1D stochastic viscous Burgers equation with $a=0$ (top), $0.5$ (middle) and $1$ (bottom) under one fixed path.}
        \label{fig:sburgers-snapshot-1d}
	\end{center}
\end{figure}

\end{subsection}

\begin{subsection}{One-dimensional Stochastic Nonlinear Convection-Diffusion Equation with cubic flux and Space-Time Colored Gradient Noise}
Next, we consider another nonlinear SPDE with cubic convection term and source term, driven by gradient-type colored noise:
\begin{gather}
\label{spde:cubicflux-1d}
    \begin{cases}
        \mathrm{d}u = \left(au_{xx} - \left(\dfrac{u^3}{3}\right)_x - u^3\right) \,\mathrm{d}t + bu_x \mathrm{d}W_t, \quad & (\omega,x,t)\in \Omega \times [0,2\pi]\times (0,T],  \\
        u(\omega,x,0) = \sin(x),  & (\omega,x) \in \Omega \times [0,2\pi],
    \end{cases}
\end{gather}
where the noise $W_t$ is a $\mathcal Q$-Wiener process given by 
\begin{align}
\label{def:Q-wiener-1d-1}
    W_t = \sum_{m=1}^{\infty} \sqrt{\gamma_m} \frac{\sin(mx)}{\sqrt{\pi}}\mathcal{B}_m, \qquad \gamma_m = \frac{1}{m^4},
\end{align}
where $\{\mathcal{B}_m(t)\}_{m\geq 1}$ are independent standard Brownian motions.

To study the stability and accuracy of the numerical scheme, we fix $a=1$, and truncate the noise term in \eqref{def:Q-wiener-1d-1} by retaining
the first $25$ modes. We use $M=1000$ Monte Carlo sample paths. Since an analytic solution is unavailable, we approximate the convergence rates by comparing the numerical solutions on two successive meshes, that is, by measuring the corresponding norms of $u_h-u_{h/2}$ on the coarse mesh and the refined mesh.

For temporal convergence rate, we fix $N_T/N = 10$ and run the simulation up to $T=1$. To study the spatial convergence rate, we fix $N_T = 10^4$ and take the final time $T=0.1$. For $b = 1$, the observed temporal and spatial convergence rates in Tables \ref{table:cubicflux-1d-timerate}-\ref{table:cubicflux-1d-spacerate} are close to the predicted order $1/2$ in time and order $(r+1)$ in space, as in Theorem \ref{thm:high-moment-error-estimate-2D}. When the noise level $b$ is increased, the error increases and the numerical convergence slows down. 

We also use this example to illustrate the influence of the stability condition.
The leading mode of the noise contributes the term $(b/\sqrt{\pi})\sin(x)u_x\,\mathrm{d}\mathcal{B}_1(t)$, which gives the benchmark condition $a>b^2/(2\pi)$. 
Thus, with $a=1$, $b=2.5$ satisfies this first-mode benchmark, whereas $b=3$ violates it.
Consistently, we observe temporal convergence for $b=2.5$ and numerical instability for $b=3$.

Lastly, we note that $W_t \notin L^2(\Omega, H^2(0,2\pi))$ and hence the regularity assumptions imposed on the exact solution in Theorem \ref{thm:high-moment-error-estimate-2D} may not be fully satisfied for the infinite-dimensional noise. 
In the computations, however, we use a finite-dimensional truncation of the noise, and the expected convergence rates are still observed. 
However, the $e_{\infty}$ error order may deteriorate due to the highly oscillatory modes. 
A rigorous analytic error estimate in this setting would require a more detailed study of how the spatial regularity of the noise affects the regularity of the exact solution.

\begin{table}[htb]
\centering
\renewcommand\arraystretch{1.25}
\caption{Numerical error and temporal convergence rate for \eqref{spde:cubicflux-1d} with $a=1, \,r=2$ at $T=1$.} 
\begin{tabular}{c c c c c c c c}
\hline\hline
$b$ & $N$ & $e_2$ & rate & $e_4$ & rate & $e_{\infty}$ & rate \\
\hline\hline
\multirow{5}{*}{1}
& 10  & 1.73E-03 & --   & 2.10E-03 & --   & 3.45E-03 & -- \\
& 20  & 1.11E-03 & 0.64 & 1.43E-03 & 0.55 & 1.81E-03 & 0.93 \\
& 40  & 7.46E-04 & 0.58 & 9.31E-04 & 0.62 & 1.19E-03 & 0.61 \\
& 80  & 5.30E-04 & 0.49 & 6.68E-04 & 0.48 & 8.40E-04 & 0.50 \\
& 160 & 3.84E-04 & 0.47 & 4.87E-04 & 0.45 & 6.02E-04 & 0.48 \\
\hline
\multirow{5}{*}{2.5}
& 10  & 2.03E-02 & --   & 2.81E-02 & --   & 4.06E-02 & -- \\
& 20  & 1.54E-02 & 0.40 & 2.17E-02 & 0.37 & 3.07E-02 & 0.41 \\
& 40  & 1.14E-02 & 0.43 & 1.67E-02 & 0.38 & 2.39E-02 & 0.36 \\
& 80  & 8.51E-03 & 0.42 & 1.25E-02 & 0.41 & 1.84E-02 & 0.38 \\
& 160 & 6.10E-03 & 0.48 & 8.39E-03 & 0.58 & 1.40E-02 & 0.39 \\
\hline
\multirow{2}{*}{3}
& 10 & 1.40E-01  & -- & 3.31E-01  & -- & 2.46E-01  & -- \\
& 20 &     NaN   & -- &   NaN     & -- &    NaN    & --  \\
\hline\hline
\end{tabular}
\label{table:cubicflux-1d-timerate}
\end{table}

\begin{table}[htb]
\centering
\renewcommand\arraystretch{1.25}
\caption{Numerical error and spatial convergence rate for \eqref{spde:cubicflux-1d} with $a=b=1$ at $T=0.1$.} 
\begin{tabular}{c c c c c c c c}
\hline\hline
$r$ & $N$ & $e_2$ & rate & $e_4$ & rate & $e_{\infty}$ & rate \\
\hline\hline
\multirow{4}{*}{1}
& 10 & 2.48E-02 & --   & 2.49E-02 & --   & 5.70E-02 & --   \\
& 20 & 6.36E-03 & 1.96 & 6.39E-03 & 1.96 & 1.56E-02 & 1.87 \\
& 40 & 1.60E-03 & 1.99 & 1.61E-03 & 1.99 & 4.04E-03 & 1.95 \\
& 80 & 4.01E-04 & 2.00 & 4.02E-04 & 2.00 & 1.02E-03 & 1.98 \\
\hline 
\multirow{4}{*}{2}
& 10 & 1.58E-03 & --   & 1.64E-03 & --   & 5.41E-03 & --   \\
& 20 & 2.47E-04 & 2.67 & 2.57E-04 & 2.68 & 9.91E-04 & 2.45 \\
& 40 & 2.68E-05 & 3.21 & 2.75E-05 & 3.22 & 1.60E-04 & 2.63 \\
& 80 & 3.32E-06 & 3.01 & 3.41E-06 & 3.01 & 2.54E-05 & 2.65 \\
\hline\hline
\end{tabular}
\label{table:cubicflux-1d-spacerate}
\end{table}

\end{subsection}

\begin{subsection}{Two-dimensional Stochastic Viscous Burgers' Equation}
As the first two-dimensional test, we consider the nonlinear stochastic viscous Burgers' equation
\begin{gather}
\label{spde:sburgers-2d}
    \begin{cases}
        \mathrm{d}u = \left[ a^2(u_{xx} + u_{yy}) - (\dfrac{u^2}{2})_x - (\dfrac{u^2}{2})_y \right] \,\mathrm{d}t + a(u_x +u_y) \mathrm{d}W_t, \quad & (\omega,x,y,t)\in \Omega \times \mathcal{D} \times (0,T], \\
        u(\omega,x,y,0) = \sin(x+y),  & (\omega,x,y) \in \Omega \times \mathcal{D}.
    \end{cases}
\end{gather}
This model admits the exact solution $u(\omega,x,y,t) = v\left(x + y + 2aW_t, t \right)$, where $v$ is the solution of the one-dimensional deterministic inviscid Burgers' equation:
\begin{gather}
\label{pde:burgers-2d}
    \begin{cases}
        v_t + (v^2)_x = 0,  \quad & (x,t)\in [0,2\pi] \times (0,T], \\
        v(x,0) = \sin(x),  &  x \in   [0,2\pi].
    \end{cases}
\end{gather}
The exact solution of \eqref{pde:burgers-2d} forms a shock at $T_0=0.5$.
Before shock formation, it is implicitly given by $ v(x,t) = \sin(x-2tv)$. 

We first perform the accuracy test by setting $a = 1$ and evolving the solution until $T = 0.4$ with $M=1000$ samples. From Tables \ref{table:sburgers-2d-timerate}-\ref{table:sburgers-2d-spacerate}, we observe the expected $(r+1)$-th and $0.5$-th order convergence rate in space and time for both the $L^2$- and $L^{\infty}$-error at $T=0.2$. As final time $T$ approaches $T_0$, the numerical errors increase and the observed temporal convergence rate drops below $0.5$. Finer choices of $N_T/N$ and a larger number of sample paths may be needed to clearly recover the theoretical half-order temporal rate near the shock time.

Next, to investigate the behavior of the numerical solution after shock formation, we plot snapshots of the numerical solution along one fixed path at $T=0.4,0.6,0.8,1$. Fig. \ref{fig:sburgers-snapshot-3dview} shows the steepening of the wave from $T=0.4$ to $0.6$, which indicates shock formation. The plots after $T = 0.6$ show that the random shift effect driven by Brownian motion displaces the local extrema of the sine wave. At $T=0.8$ and $T=1$ spurious oscillations appear near the discontinuity as the mesh is refined. 
Nonlinear limiters can be applied to control the oscillations near discontinuities.    

\begin{table}[htb]
\centering
\renewcommand\arraystretch{1.25}
\caption{Numerical error and temporal convergence rate for \eqref{spde:sburgers-2d} with $a=r=1$. } 
\begin{tabular}{c c c c c c c c}
\hline\hline
$T$ & $N$ & $e_2$ & rate & $e_4$ & rate & $e_{\infty}$ & rate \\
\hline\hline
\multirow{5}{*}{0.1}
& 10  & 4.40E-02 & --   & 4.81E-02 & --   & 1.02E-01 & --    \\
& 20  & 2.23E-02 & 0.98 & 2.76E-02 & 0.80 & 4.38E-02 & 1.22 \\
& 40  & 1.49E-02 & 0.58 & 1.91E-02 & 0.53 & 2.30E-02 & 0.93 \\
& 80  & 1.05E-02 & 0.50 & 1.37E-02 & 0.48 & 1.42E-02 & 0.70 \\
& 160 & 7.31E-03 & 0.52 & 9.32E-03 & 0.56 & 9.32E-03 & 0.61 \\
\hline
\multirow{5}{*}{0.2}
& 10  & 7.58E-02 & --   & 8.81E-02 & --   & 1.79E-01 & --    \\
& 20  & 4.67E-02 & 0.70 & 5.85E-02 & 0.59 & 9.69E-02 & 0.88 \\
& 40  & 3.26E-02 & 0.52 & 4.29E-02 & 0.45 & 5.82E-02 & 0.74 \\
& 80  & 2.31E-02 & 0.50 & 2.99E-02 & 0.52 & 3.77E-02 & 0.63 \\
& 160 & 1.66E-02 & 0.48 & 2.11E-02 & 0.50 & 2.58E-02 & 0.55 \\
\hline
\multirow{5}{*}{0.4}
& 10  & 1.89E-01 & --   & 2.34E-01 & --   & 4.75E-01 & --  \\
& 20  & 1.38E-01 & 0.46 & 1.62E-01 & 0.53 & 3.89E-01 & 0.29 \\
& 40  & 1.03E-01 & 0.42 & 1.21E-01 & 0.43 & 3.12E-01 & 0.32 \\
& 80  & 7.96E-02 & 0.38 & 9.93E-02 & 0.28 & 2.52E-01 & 0.31 \\
& 160 & 5.86E-02 & 0.44 & 6.98E-02 & 0.51 & 1.97E-01 & 0.36 \\
\hline\hline
\end{tabular}
\label{table:sburgers-2d-timerate}
\end{table}

\begin{table}[htb]
\centering
\renewcommand\arraystretch{1.25}
\caption{Numerical error and spatial convergence rate for \eqref{spde:sburgers-2d} with $a=1$ at $T=0.2$.}
\begin{tabular}{c c c c c c c c}
\hline\hline
$r$ & $N$ & $e_2$ & rate & $e_4$ & rate & $e_{\infty}$ & rate \\
\hline\hline
\multirow{4}{*}{1}
& 10 & 1.46E-01 & --   & 1.77E-01 & --   & 2.90E-01 & --    \\
& 20 & 3.86E-02 & 1.92 & 4.80E-02 & 1.88 & 8.31E-02 & 1.81 \\
& 40 & 9.75E-03 & 1.99 & 1.19E-02 & 2.02 & 2.13E-02 & 1.97 \\
& 80 & 2.47E-03 & 1.98 & 3.06E-03 & 1.95 & 5.30E-03 & 2.01 \\
\hline
\multirow{3}{*}{2}
& 10 & 1.42E-01 & --   & 1.76E-01 & --   & 2.48E-01 & --    \\
& 20 & 1.82E-02 & 2.96 & 2.28E-02 & 2.95 & 2.89E-02 & 3.10 \\
& 40 & 2.37E-03 & 2.94 & 3.01E-03 & 2.92 & 3.58E-03 & 3.01 \\
\hline\hline
\end{tabular}
 \label{table:sburgers-2d-spacerate}
\end{table}


\begin{figure}[htb]
	\begin{center}
		\includegraphics[scale=0.8]{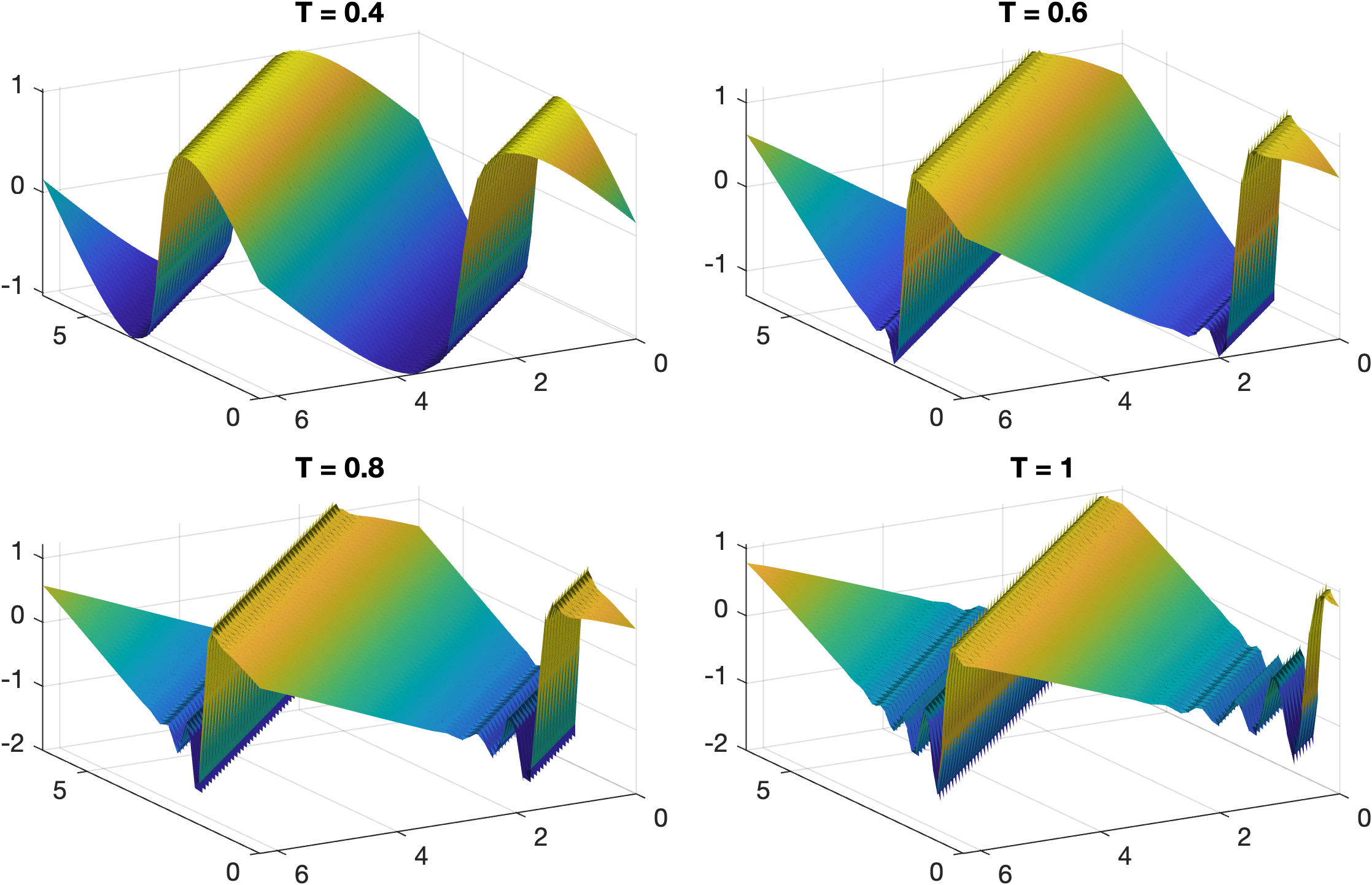}
		\caption{Solution snapshot of 2D stochastic viscous Burgers equation under one fixed path with $a=1$ and $N=50$.} 
        \label{fig:sburgers-snapshot-3dview}
	\end{center}
\end{figure}



\end{subsection}

\begin{subsection}{Two-dimensional Stochastic Allen-Cahn Equation with Multiplicative noise}
In the following, we apply our scheme to the two-dimensional stochastic Allen-Cahn equation with function-type multiplicative noise 
\begin{gather}
\label{spde:allen-cahn-2d}
    \begin{cases}
        \mathrm{d}u = (u_{xx} +u_{yy})\, \mathrm{d}t + \frac{1}{\epsilon^2}(u - u^3)\mathrm{d}t + \frac{b}{\epsilon^2}u \mathrm{d}W_t, \quad & (\omega,x,y,t)\in \Omega \times \mathcal{D}\times (0,T], \\
        u(\omega,x,y,0) = u_0(x,y),  & (\omega,x,y) \in \Omega \times\mathcal{D}.
    \end{cases}
\end{gather}
We conduct two numerical tests and investigate the effect of the multiplicative noise by varying the noise intensity $b$.

\begin{subsubsection}{Merging of Four Bubbles}
We simulate the merging of four bubbles with the following initial condition
\begin{align*}
u_0(x,y) &= -\tanh\left( ((x - 0.3)^2 + y^2 - 0.2^2) / \epsilon \right)
\tanh\left( ((x + 0.3)^2 + y^2 - 0.2^2) / \epsilon \right) \\
&\quad \times \tanh\left( (x^2 + (y - 0.3)^2 - 0.2^2) / \epsilon \right)
\tanh\left( (x^2 + (y + 0.3)^2 - 0.2^2) / \epsilon \right).
\end{align*}
We consider the interfacial width $\epsilon = 0.02$ and the computational domain $\mathcal{D} = [-1,1]^2$, which is uniformly divided into $10000$ square elements. The time step is $\Delta t = 5\times 10^{-5}$ and the sample mean is computed using $M=100$ sample paths. 
We plot the sample mean of the numerical solution at time $t = 0.0001, 0.007,0.009,0.025,0.04,0.05,0.08,0.2$, with different noise magnitude $b = 0,0.01,0.015,0.02$.
The dynamics before and after approximately $T=0.05$ correspond to the merging of four bubbles and the shrinking of the resulting single bubble,respectively.

The solution snapshots are shown in Figs. \ref{Fig:mergence-bubble-b=0}-\ref{Fig:mergence-bubble-b=0.02}. 
In all cases, the four initial bubbles merge into a single large bubble, which then gradually shrinks because the Allen-Cahn equation is not mass-conserving. 
In the deterministic case, the single bubble finally disappears and the solution converges to the stable steady state equilibrium $u=-1$. 
Under noisy perturbation, the sample mean of solution is observed to decay, while the overall shape of the bubble is largely preserved. Note that $u=0$ is the only constant equilibrium of the stochastic model \eqref{spde:allen-cahn-2d}, whereas the deterministic equilibria $u=\pm1$ are no longer pathwise stationary because of the multiplicative noise term.

The results show that stronger noise accelerates the decay toward the absorbing state, while the stochastic solution approaches the deterministic behavior for small values of $b$. 
Additional experiments with a larger noise intensity, for example $b=0.03$, show that the first and second moments of the numerical solution are of order $10^{-9}$ and $10^{-15}$, respectively, at $T=0.1$. 
This suggests that, for sufficiently large noise intensity, the solution is rapidly driven toward $u=0$ with very small sample variance.

\begin{figure}[htb]
	\begin{center}
		\includegraphics[scale=0.85]{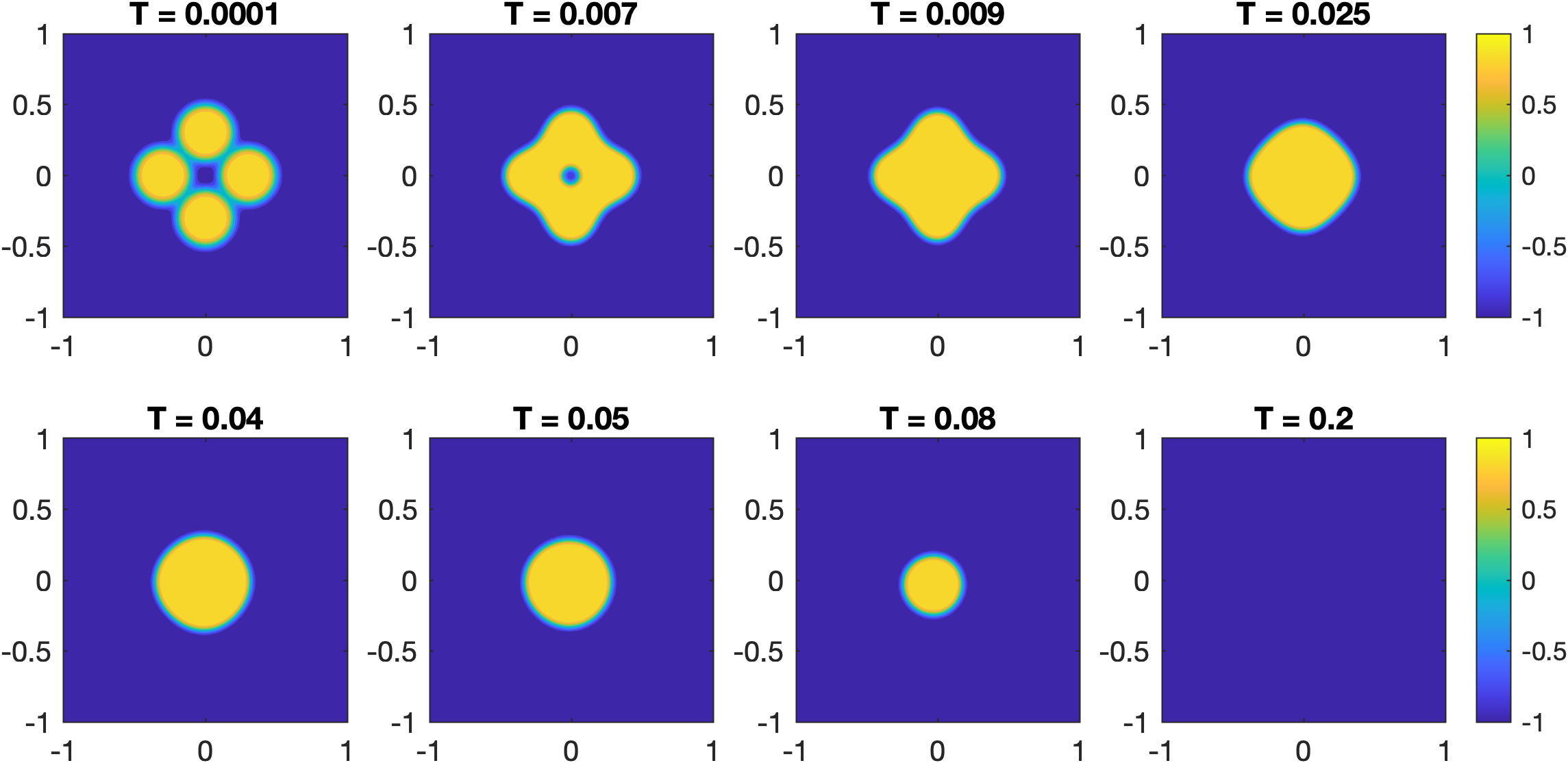}
		\caption{Merging of four bubbles: Solution of deterministic Allen-Cahn equation. 
        }
        \label{Fig:mergence-bubble-b=0}
	\end{center}
\end{figure}

\begin{figure}[htb]
	\begin{center}
		\includegraphics[scale=0.85]{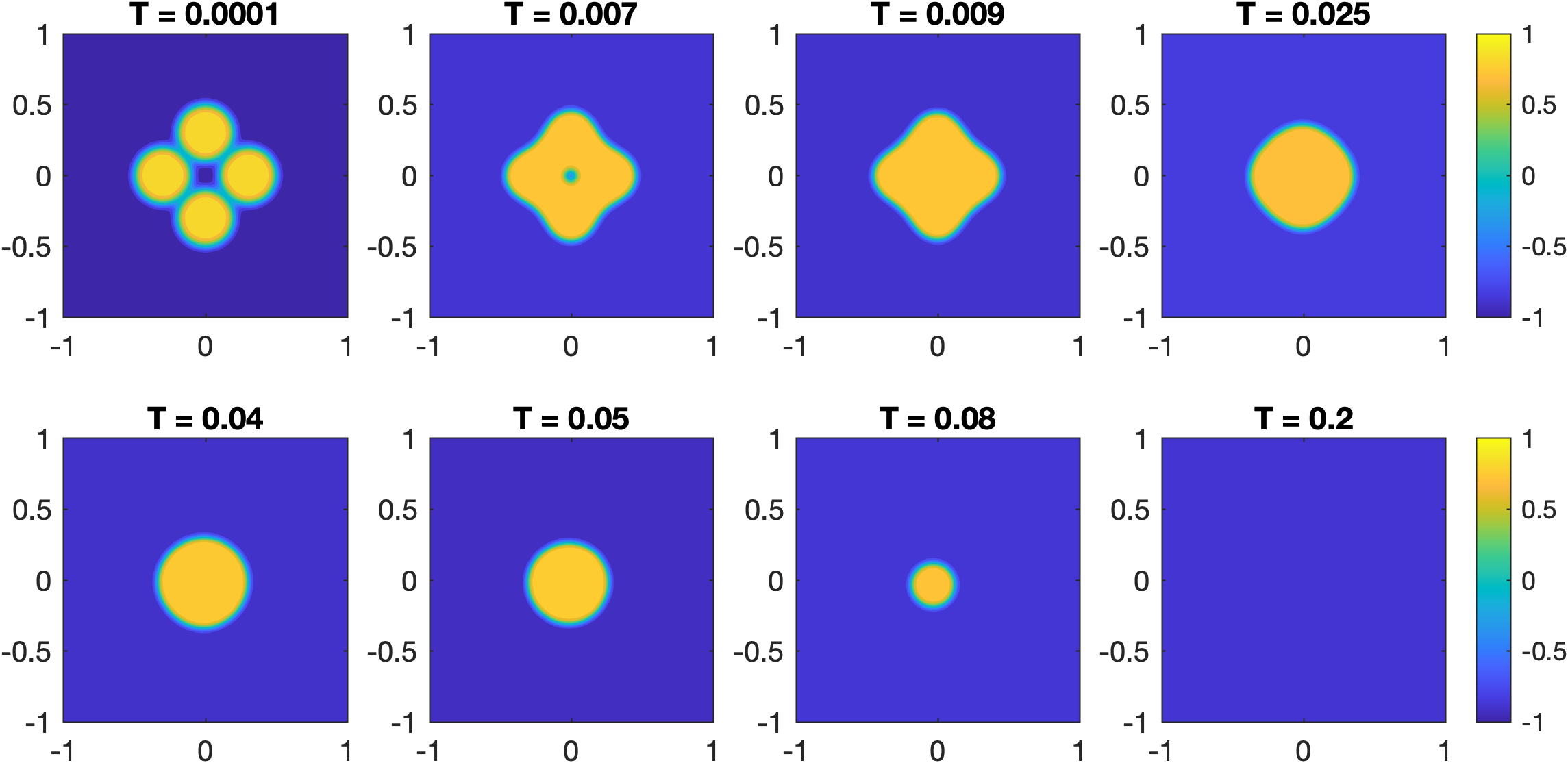}
		\caption{Merging of four bubbles: Averaged solution of stochastic Allen-Cahn equation with $b=0.01$.}
        \label{Fig:mergence-bubble-b=0.01}
	\end{center}
\end{figure}

\begin{figure}[htb]
	\begin{center}
		\includegraphics[scale=0.85]{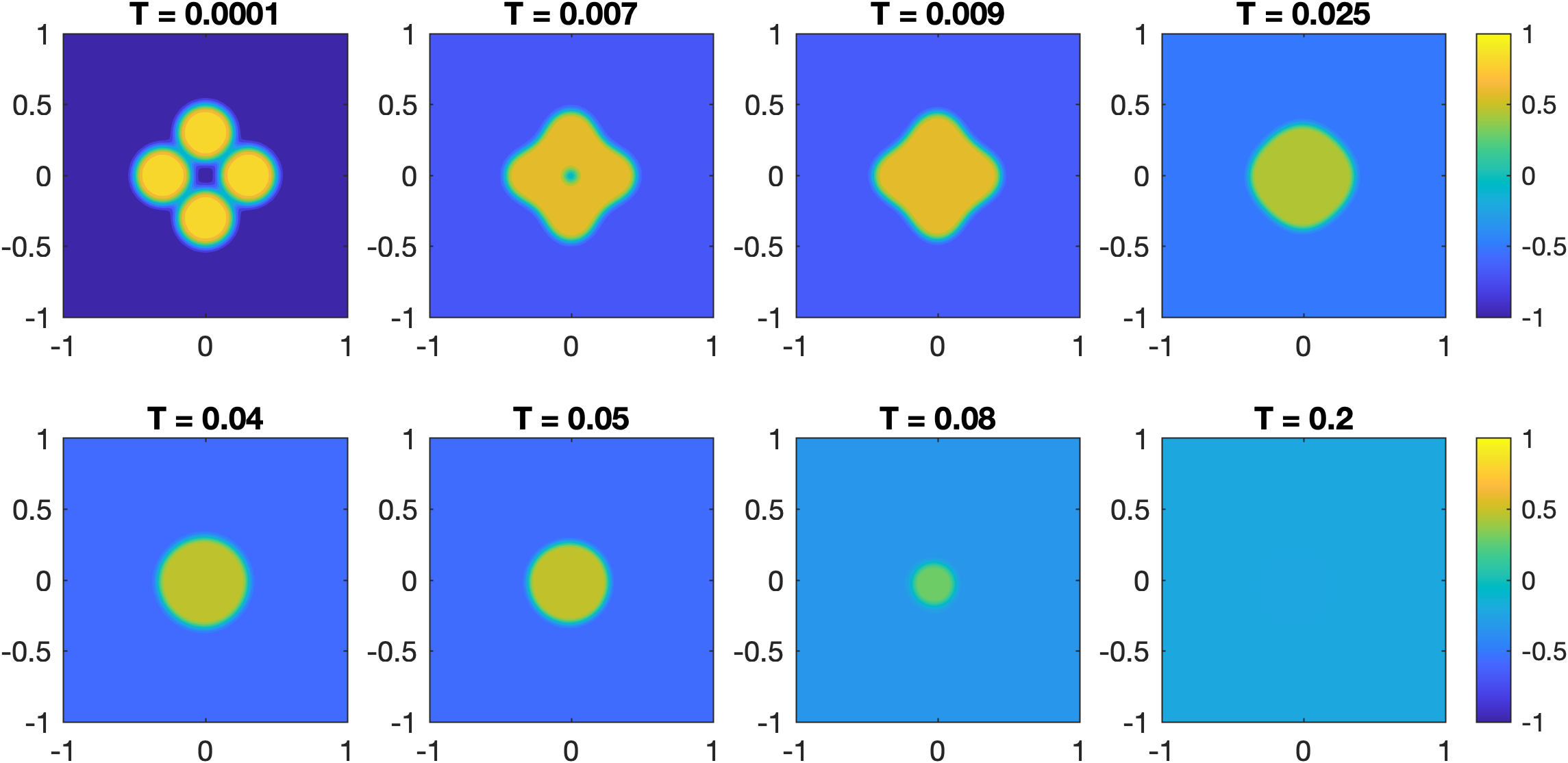}
		\caption{Merging of four bubbles: Averaged solution of stochastic Allen-Cahn equation with $b=0.015$.}
        \label{Fig:mergence-bubble-b=0.015}
	\end{center}
\end{figure}

\begin{figure}[htb]
	\begin{center}
		\includegraphics[scale=0.85]{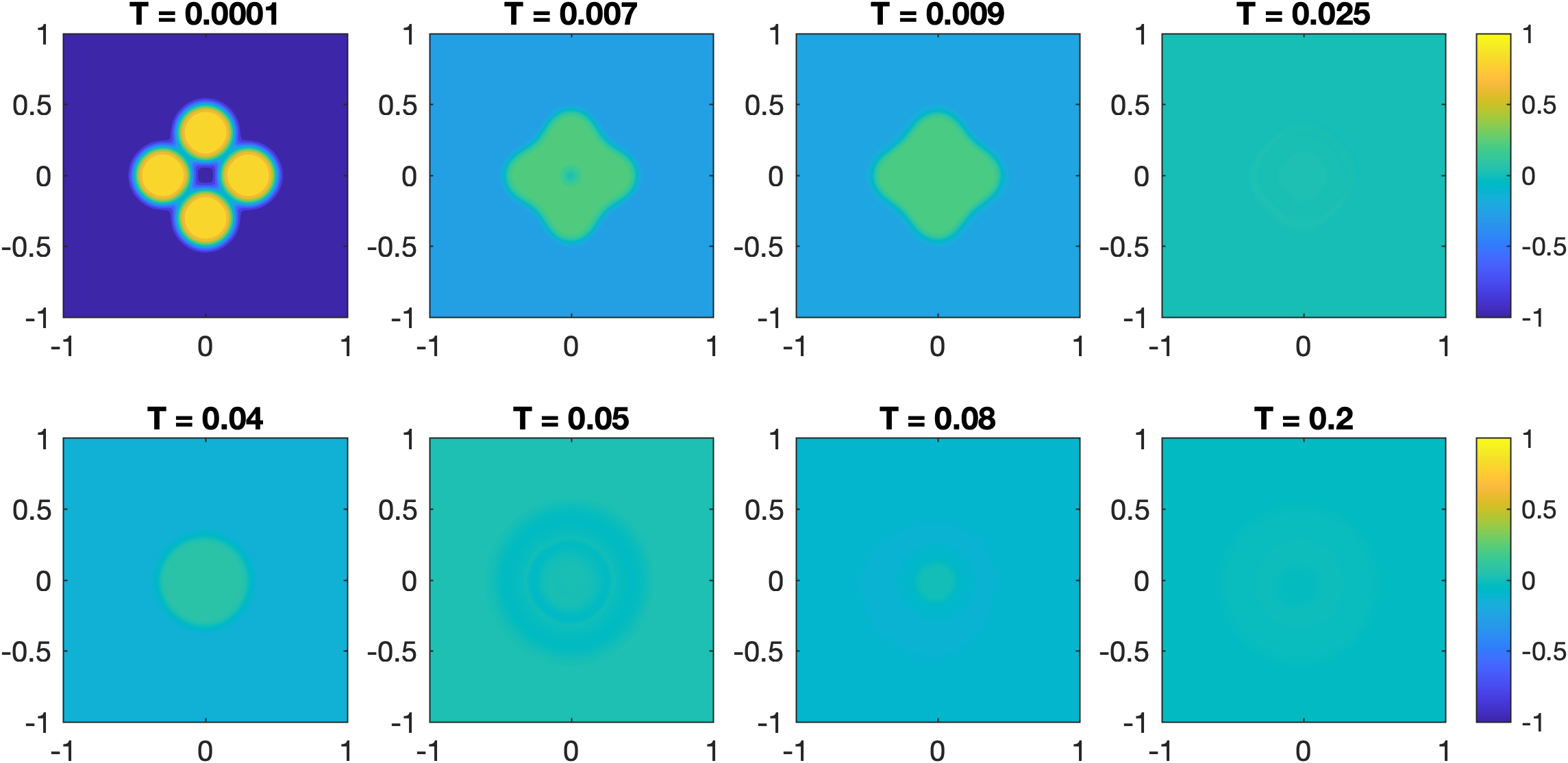}
		\caption{Merging of four bubbles: Averaged solution of stochastic Allen-Cahn equation with $b=0.02$.}
        \label{Fig:mergence-bubble-b=0.02}
	\end{center}
\end{figure}

\end{subsubsection}

\begin{subsubsection}{Phase Separation}
For the second test on the model \eqref{spde:allen-cahn-2d}, we investigate the coarsening dynamics of the Allen-Cahn equation with interfacial width $\epsilon = 0.01$. 
We choose the random initial condition $u(x,y,0) = 0.1\, \mathrm{rand}(x,y)$, where $\mathrm{rand}(x,y)$ is sampled uniformly from $[-1,1]$ at each grid point. 
The computational domain $\mathcal{D} = [0,1]^2$ is uniformly divided into $10000$ square elements. 
The initial condition is fixed for all sample paths and is shown in Fig \ref{Fig:Phase-sep-IC}. 
We run the simulation with $\Delta t = 10^{-5}$ until $T=0.01$ along $M=100$ paths, while varying the noise intensity $b$. 

The snapshots of the averaged evolution of the coarsening dynamics are plotted at $T=0.001$, $0.005$, $0.008$, $0.01$. From Figs. \ref{Fig:Phase-sep-b=0}-\ref{Fig:Phase-sep-b=0.001-0.01}, we observe that, in all cases, the initial fine-scale grains gradually coalesce into distinct phases, driven by energy dissipation mechanism
of the Allen-Cahn dynamics. 
When the noise level is small, the sample mean remains close to the deterministic phase separation process.
As the noise intensity increases, the numerical solution decays toward zero more rapidly.

Further tests show that, when $b=0.02$, the first and second moment of the numerical solution reach the order of $10^{-25}$ and $10^{-47}$, respectively, as the simulation time approaches $T=0.01$. 
This indicates that, for sufficiently large noise intensity, the multiplicative noise rapidly drives the solution toward the absorbing state $u=0$ with negligible sample variance.

\begin{figure}[htb]
	\begin{center}
		\includegraphics[scale=0.3]{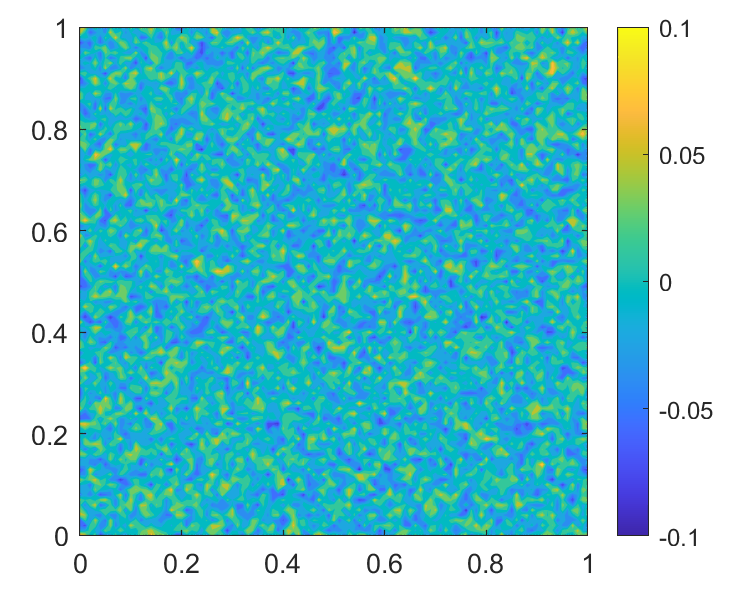}
		\caption{Random initial condition for phase separation.}
        \label{Fig:Phase-sep-IC}
	\end{center}
\end{figure}

\begin{figure}[htb]
	\begin{center}
		\includegraphics[scale=0.85]{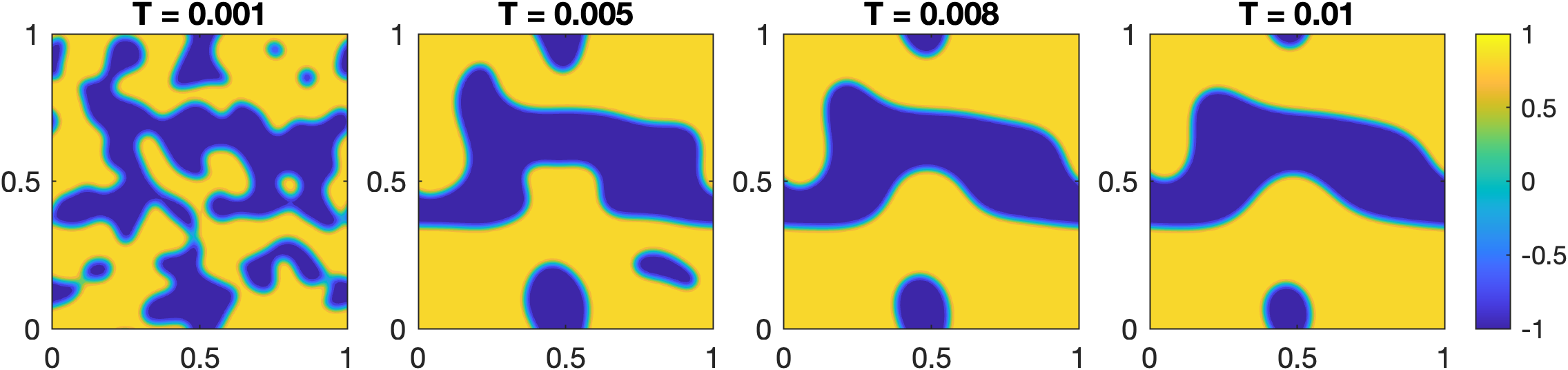}
		\caption{Evolution of 2D deterministic Allen-Cahn equation for phase separation. 
        }
         \label{Fig:Phase-sep-b=0}
	\end{center}
\end{figure}

\begin{figure}[htb]
	\begin{center}
		\includegraphics[scale=0.85]{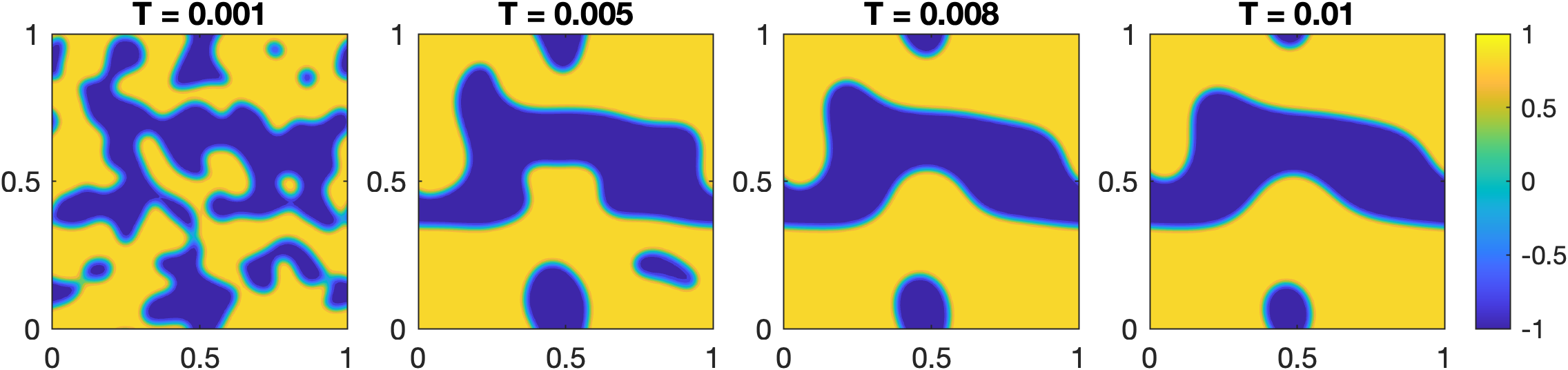}
        \includegraphics[scale=0.85]{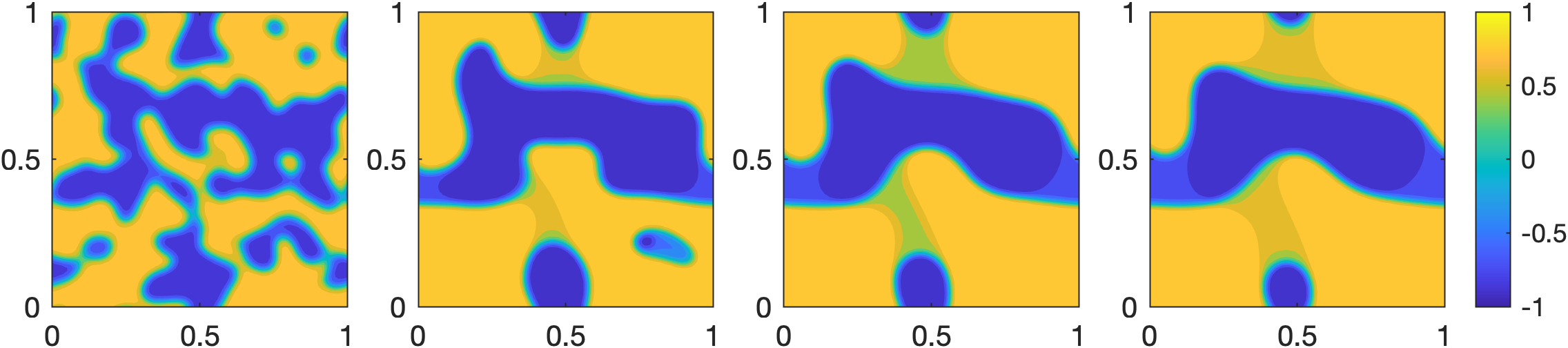}
        \includegraphics[scale=0.85]{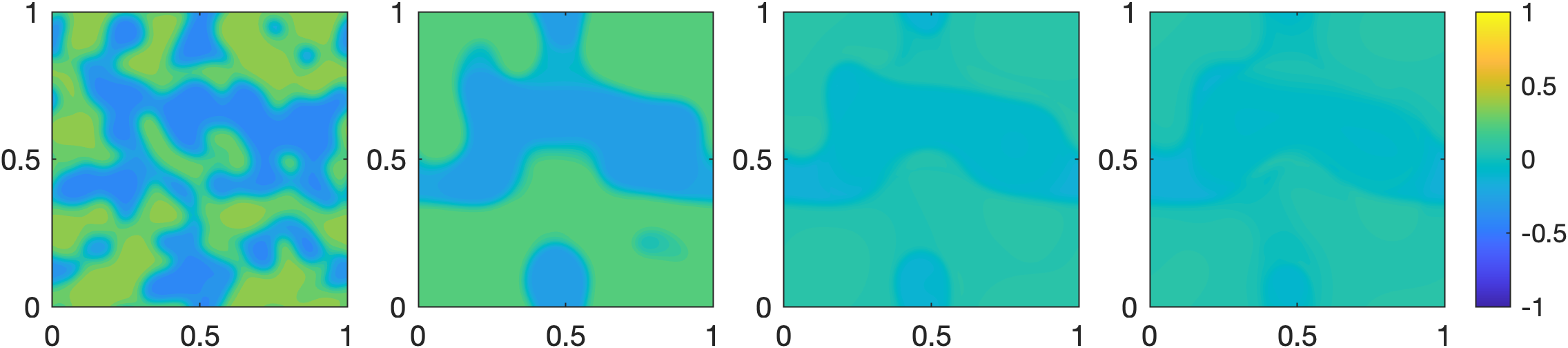}
		\caption{Averaged evolution of stochastic Allen-Cahn equation with $b=0.001$ (top), $0.005$ (middle) ,$0.01$ (bottom row) for phase separation.}
        \label{Fig:Phase-sep-b=0.001-0.01}
	\end{center}
\end{figure}

\end{subsubsection}

\end{subsection}

\end{section}

\begin{section}{Concluding Remarks}
\label{sec:conclusion}
In this paper, we proposed and analyzed a fully discrete LDG-IMEX-Euler scheme for general nonlinear stochastic convection-diffusion equations in two dimensions on Cartesian meshes. 
The diffusion term is treated implicitly by the LDG method, while the nonlinear convection, lower-order drift, and multiplicative noise terms are treated explicitly. 
Under suitable regularity and structural assumptions, we established high-moment stability and error estimates on subsets of the sample space whose probabilities converge to one as the discretization parameters vanish. 
The resulting convergence rate is arbitrarily close to the optimal order $r+1$ in space and order $1/2$ in time. 
A pathwise error estimate was further obtained from the high-moment error bound together with a discrete Kolmogorov lemma.

Several numerical experiments were presented to validate the theoretical rates and to illustrate the performance of the method for nonlinear stochastic Burgers-type equations and stochastic Allen-Cahn equations.
The results confirm the expected accuracy before shock formation and show the qualitative behavior of the proposed scheme for strongly nonlinear dynamics. 
The analysis developed here can be extended to higher-dimensional Cartesian meshes, and future work will consider more general meshes, fully implicit treatments of the nonlinear terms, and nonlinear stochastic hyperbolic problems.
\end{section}

\begin{appendices}

\section{Supplemental Proofs}

\label{appendixA}

\subsection{Proof of Lemma \ref{lem:num-sol-2D}}
\label{appendix-num-sol-2d}
\begin{proof}
    Let $K = I_i \times J_j$. We first prove the estimate in the $x$-direction. 
    By \eqref{eq:fully-discrete-2-2d} and integration by parts, for any $p_h\in Q^r(K)$, 
    \begin{align}
    \label{eq:num-sol-2d}
        (v_{1,h},p_h)_K &= -(u_h,(p_h)_x)_K + \int_{J_j} ({u}_{h}^- p_h^-)_{i+\frac{1}{2},y} - ({u}_{h}^- p_h^+)_{i-\frac{1}{2},y} \, \mathrm{d}y \notag \\
        &= ((u_h)_x,p_h)_K + \int_{J_j} ([u_h] p_h^+)_{i-\frac{1}{2},y} \, \mathrm{d}y. 
    \end{align}
    Define $r_h(x,y) \in Q^r(K)$ as 
    \[ r_h(x,y) = (u_h)_x - (-1)^r (u_h)_x(x_{i-\frac{1}{2}}^+,y)L_x^r(\xi), 
        \qquad \xi=\frac{2(x-x_i)}{h_{x,i}},\]
     where $L_x^r$ is the standard $r$-th order Legendre polynomial on $[-1,1]$. Clearly $r_h(x_{i-\frac{1}{2}}^+,y) = 0$ for all $y \in J_j$, since $L_x^r(-1) = (-1)^r$. 
    Taking $p_h = r_h$ in \eqref{eq:num-sol-2d}, and using the fact that $L_x^r$ is orthogonal to any polynomials with degrees at most $r-1$, we get 
    \[ (v_{1,h},r_h)_K = ((u_h)_x,r_h)_K = \| (u_h)_x\|_K^2 - (-1)^r \left((u_h)_x, (u_h)_x(x_{i-\frac{1}{2}}^+,y)L_x^r(\xi) \right)_K = \| (u_h)_x\|_K^2. \]
    Thus by Cauchy-Schwarz inequality, 
    \begin{align*}
        \| (u_h)_x\|_K^2 &\leq \| v_{1,h}\|_K \| r_h\|_K \leq \| v_{1,h}\|_K \left(\| (u_h)_x\|_K + \| (u_h)_x(x_{i-\frac{1}{2}}^+,y) L_x^r(\xi)\|_K \right) \\
        &\leq \| v_{1,h}\|_K \left(\| (u_h)_x\|_K + \| (u_h)_x(x_{i-\frac{1}{2}}^+,y) \|_{J_j}  Ch_{x,i}^{\frac{1}{2}} \right) 
        \leq C_{\mu} \| v_{1,h}\|_K \| (u_h)_x\|_K,
    \end{align*}
    where the last inequality is obtained by the inverse inequality. Therefore, $\| (u_h)_x\|_K \leq C_{\mu} \| v_{1,h}\|_K$. 
    
    It remains to estimate the jump in the $x$-direction. We define
    \begin{gather*}
        p_h = u_h(x,y) - u_h(x_{i-\frac{1}{2}}^-,y) = \int_{x_{i-\frac{1}{2}}}^x (u_h)_x(s,y) \, \mathrm{d}s + [u_h]_{i-\frac{1}{2},y}.
    \end{gather*}
    Then $p_h\in Q^r(K)$ and $p_h(x_{i-\frac{1}{2}}^+,y)=[u_h]_{i-\frac{1}{2},y}$.
    Plugging $p_h$ into \eqref{eq:num-sol-2d}, and applying triangle inequality with the result $\| (u_h)_x\|_K \leq C_{\mu} \| v_{1,h}\|_K$, we get
    \begin{align*}
        \int_{J_j} ([u_h])_{i-\frac{1}{2},y}^2 \, \mathrm{d}y &= (v_{1,h} - (u_h)_x,p_h)_K \leq h_{x,i}\| v_{1,h} - (u_h)_x\|_K^2 + \frac{1}{4h_{x,i}}\| p_h\|_K^2 \\
        &\leq C_{\mu} h_{x,i} \| v_{1,h}\|_K^2 + \frac{1}{2h_{x,i}} \left\| \int_{x_{i-\frac{1}{2}}}^x (u_h)_x(s,y) \, \mathrm{d}s \right\|_K^2 + \frac{1}{2h_{x,i}} \left\| [u_h]_{i-\frac{1}{2},y} \right\|_K^2 \\
        &\leq C_{\mu} h_x \| v_{1,h}\|_K^2 + \frac{h_x}{2} \| (u_h)_x\|_K^2 + \frac{1}{2} \int_{J_j} ([u_h])_{i-\frac{1}{2},y}^2 \, \mathrm{d}y.
    \end{align*}  
    Therefore, we reach
    \[ \int_{J_j} ([u_h])_{i-\frac{1}{2},y}^2 \, \mathrm{d}y \leq C_{\mu} h_x \| v_{1,h}\|_K^2. \]
    The estimate in the $y$-direction is proved in the same way. Combining all above completes the proof of the first inequality. The estimate for $\xi_u$ follows from the same argument with the usage of Lemma \ref{lem:superconvergence}.  
\end{proof}

\subsection{Proof of Lemma \ref{lem:nonlinear-convec-2D}}
\label{appendix-nonlinear-convec-2d}
\begin{proof}
Let $K = I_i \times J_j$ and $C_f = C(1+\| u_h\|_{\infty}^p + \| z_h\|_{\infty}^p)$. Throughout the proof, we denote by $C$ a generic constant independent of $u_h$, $z_h$ and $r_h$. Utilizing periodic boundary conditions, and integration by parts, we have
    \begin{align*}
        \tilde{H}(f_1,f_2,u_h,r_h) 
        &= \sum_{i,j} -\left( f_1(u_h)_x + f_2(u_h)_y, r_h \right)_K \\
        &+ \sum_{i,j} \int_{J_j} \left( \left(f_1(u_h^-) - f_1(u_h^+) \right) r_h^- + \left(\hat{f}_1-f_1(u_h^+) \right) [r_h] \right)_{i+\frac{1}{2},y} \, \mathrm{d}y \\
        &+ \sum_{i,j} \int_{I_i} \left( \left(f_2(u_h^-) - f_2(u_h^+) \right) r_h^- + \left(\hat{f}_2 -f_2(u_h^+) \right) [r_h] \right)_{x,j+\frac{1}{2}} \, \mathrm{d}x.  
    \end{align*}
    By hypothesis (iii) and Cauchy-Schwarz inequality, we get 
    $$\left|\sum_{i,j} -\left( f_1(u_h)_x + f_2(u_h)_y, r_h \right)_K \right| \leq C(1+\| u_h\|_{\infty}^p) \| \nabla u_h\| \| r_h\|.$$ 
    Moreover, the local Lax-Friedrichs flux and hypothesis (iii) imply
    $$|\hat f_i(u^-,u^+)-f_i(u^+)| + |f_i(u^-)-f_i(u^+)| \leq C(1+\|u_h\|_\infty^p)|[u_h]|, \qquad i=1,2 .$$
    Using this bound, the trace inverse inequality, and the discrete Cauchy--Schwarz inequality, the interface terms satisfy
    $$|\text{interface terms}| \leq C(1+\|u_h\|_\infty^p)  \mu h^{-\frac12}\|[u_h]\|_{\Gamma_h}\|r_h\|.$$
    Combining the volume and interface estimates gives
    \[ |\tilde{H}(f_1,f_2,u_h,r_h)| \leq C(1+\| u_h\|_{\infty}^p) \left(\| \nabla u_h\| + {\mu} h^{-\frac{1}{2}}\|[u_h]\|_{\Gamma_h} \right) \| r_h\|. \]
    
    To estimate $D(f_1,f_2;u_h,z_h,r_h)$, note that
    \begin{align*}
        D(f_1,f_2;u_h,z_h,r_h) &= \sum_{i,j} \left( \left( f_1(u_h)-f_1(z_h), (r_h)_x \right)_K + \left(f_2(u_h)-f_2(z_h), (r_h)_y \right)_K \right) \\
        &+ \sum_{i,j} \int_{J_j} \bigl( \bigl(\hat{f}_{1}(u_h) - \hat{f}_{1}(z_h)\bigr) [r_h] \bigr)_{i+\frac{1}{2},y} \, \mathrm{d}y + \sum_{i,j} \int_{I_i} \bigl( \bigl(\hat{f}_{2}(u_h) - \hat{f}_{2}(z_h)\bigr) [r_h] \bigr)_{x,j+\frac{1}{2}} \, \mathrm{d}x. 
    \end{align*}
    Similarly, the volume terms are bounded by hypothesis (iii):
    $$\left|\sum_{i,j} \left( \left( f_1(u_h)-f_1(z_h), (r_h)_x \right)_K + \left(f_2(u_h)-f_2(z_h), (r_h)_y \right)_K \right) \right| 
        \leq C_f \| (u_h-z_h)\| \| \nabla r_h\| .$$ 
    For the interface terms, the local Lax-Friedrichs flux is Lipschitz on bounded sets. Hence, for $i=1,2$,
    $$|\hat f_i(u_h^-,u_h^+)-\hat f_i(z_h^-,z_h^+)| \leq C_f\left( |u_h^- - z_h^-|+|u_h^+ - z_h^+| \right).$$
    Therefore, by the trace inverse inequality and the discrete Cauchy--Schwarz inequality,
    $$|\text{interface terms}| \leq C_f\,\mu h^{-\frac12} \|u_h-z_h\|\,\|[r_h]\|_{\Gamma_h}.$$
    Combining the volume and interface estimates yields
    \[ |D(f_1,f_2;u_h,z_h,r_h)| \leq C_f \| u_h-z_h\| \left(\| \nabla r_h\| + {\mu} h^{-\frac{1}{2}} \|[r_h]\|_{\Gamma_h} \right). \]
    This completes the proof.
\end{proof}

\subsection{The choice of $\{ \delta,\epsilon_i: i = 0,1,...,7 \}$ in Theorem \ref{thm:stability-estimate-2D}}
\label{appendix-eps-choice}
We claim that $\alpha > \alpha_0(q) := (2^{2q-1}C_b' + 2^{5q-4}C_b)^{1/q}C_b^{1/q}D_4^2K$ if and only if there exist $\delta,\epsilon_i >0$, $i=0,...,7$, such that $\mathcal{C}_1, \mathcal{C}_2, \mathcal{C}_3$ are positive. The three coefficients are given by 
\begin{align*}
    \mathcal{C}_1 &= \frac{1}{2^q}-\frac{2^q(2^{q-1}+\epsilon_4)}{4\epsilon_6}, \quad \mathcal{C}_2 = \left(\frac{1}{2}-\frac{1}{\epsilon_0}-\frac{1}{\epsilon_3} - \epsilon_1 \right)^q, \\
    \mathcal{C}_3 &= \left(\alpha - 4(C_{\mu}^*)^2\epsilon_0 \delta - \epsilon_2 B_3^2 \right)^q -  (2^{q-1}+\epsilon_4)\epsilon_3^q C_b'C_b K^qD_4^{2q}(1+\epsilon_5) - (2^{q-1}+\epsilon_4) 2^q \epsilon_6C_b^2 K^q D_4^{2q} (1+\epsilon_7).
\end{align*}

\begin{proof} 

We first prove the necessity. Suppose that the above coefficients are positive. The condition $\mathcal C_1>0$ implies $\epsilon_6 > 2^{2q-2}(2^{q-1}+\epsilon_4) > 2^{3q-3}$, while $\mathcal C_2>0$ implies $\epsilon_3>2$. Therefore,
$$
\begin{aligned}
& (2^{q-1}+\epsilon_4)  \epsilon_3^q C_b'C_bK^qD_4^{2q}(1+\epsilon_5) 
+  (2^{q-1}+\epsilon_4)2^q\epsilon_6 C_b^2K^qD_4^{2q}(1+\epsilon_7) \\
& \qquad
>2^{2q-1}C_b'C_bK^qD_4^{2q}+2^{5q-4}C_b^2K^qD_4^{2q}=\alpha_0^q(q).
\end{aligned}
$$
The condition $\mathcal C_3>0$ gives $\alpha^q >\alpha_0^q(q)$, hence $\alpha>\alpha_0(q)$.

To show the other direction, first choose $\epsilon_3>2$ and $\epsilon_6>2^{3q-3}$. Then for $1/\epsilon_0$ and $\epsilon_1$ small, we have $\mathcal{C}_2>0$. By taking $\epsilon_4>0$ sufficiently small we also ensure $\mathcal{C}_1>0$.

Next choose $\delta, \epsilon_2>0$ so that
\[  \delta \leq \delta(q) := \frac{\alpha - \alpha_0(q)}{16(C_{\mu}^*)^2 \epsilon_0}, \qquad
\epsilon_2 B_3^2 = \frac{\alpha-\alpha_0(q)}{4}
\qquad\text{if} \, B_3>0. 
\]
If $B_3=0$, we simply choose any $\epsilon_2>0$, since in that case $\epsilon_2B_3^2=0$. Then
\[
\alpha-4(C_{\mu}^*)^2\epsilon_0 \delta-\epsilon_2B_3^2 \ge \frac{\alpha+\alpha_0(q)}{2},
\]
and hence
\[
(\alpha-4(C_{\mu}^*)^2\epsilon_0 \delta-\epsilon_2B_3^2)^q
\ge \alpha_0^q(q) + \Bigl(\frac{\alpha-\alpha_0(q)}{2}\Bigr)^q.
\]
If $\epsilon_3$ and $\epsilon_6$ are chosen arbitrarily close to their lower bounds $2$ and $2^{3q-3}$, respectively, and if $\epsilon_4,\epsilon_5,\epsilon_7$ are taken sufficiently small, then
\begin{align*}
&(1+\epsilon_5)(2^{q-1}+\epsilon_4)\epsilon_3^q C_b'C_b K^qD_4^{2q}
+ (1+\epsilon_7)(2^{q-1}+\epsilon_4)2^q\epsilon_6 C_b^2K^qD_4^{2q}
\end{align*}
can be made arbitrarily close to
$$2^{2q-1} C_b'C_b K^qD_4^{2q} + 2^{5q-4} C_b^2K^qD_4^{2q} = \alpha_0^q(q),$$
which leads to
\begin{align*}
&(1+\epsilon_5)(2^{q-1}+\epsilon_4)\epsilon_3^q C_b'C_b K^qD_4^{2q}
+ (1+\epsilon_7)(2^{q-1}+\epsilon_4)2^q\epsilon_6 C_b^2K^qD_4^{2q} 
< \alpha_0^q(q) + \Bigl(\frac{\alpha-\alpha_0(q)}{2}\Bigr)^q.
\end{align*}
Consequently, we obtain $\mathcal{C}_3>0$.

Thus there exist $\delta \in (0,\delta(q)]$ and $\{\epsilon_i>0: i=0,\ldots,7\}$ such that $\mathcal{C}_1,\mathcal{C}_2,\mathcal{C}_3$ are all positive. 
\end{proof}

\subsection{Proof of Lemma \ref{lem:error2D-term1-2}}
\begin{proof}
Let $\epsilon_0, \epsilon_1 >0$ be arbitrary constants. We first estimate the term involving $\psi_1$. 
For each $t\in[t_n,t_{n+1}]$, we decompose
\begin{align*}
&\psi_1(\cdot,t,u,v_1,v_2)-\psi_1(\cdot,t_n,u_h^n,v_{1,h}^n,v_{2,h}^n) =
\bigl[\psi_1(\cdot,t,u,v_1,v_2) - \psi_1(\cdot,t,u^n,v_1^n,v_2^n)\bigr] \\
&\qquad\quad +
\bigl[\psi_1(\cdot,t,u^n,v_1^n,v_2^n)-\psi_1(\cdot,t_n,u^n,v_1^n,v_2^n)\bigr] +
\bigl[\psi_1(\cdot,t_n,u^n,v_1^n,v_2^n)-\psi_1(\cdot,t_n,u_h^n,v_{1,h}^n,v_{2,h}^n)\bigr].
\end{align*}
By the elementary inequality
$$\|a+b+c\|^2 \le C\bigl(\|a\|^2+\|b\|^2\bigr) + (1+\epsilon_0)\|c\|^2,$$
we obtain 
\begin{align*}
     &\left\| \psi_1(\cdot,t,u,v_1,v_2) - \psi_1(\cdot,t_{n},u_h^{n},v_{1,h}^{n},v_{2,h}^n) \right\|^2 
     \leq C\left\| \psi_1(\cdot,t,u,v_1,v_2) - \psi_1(\cdot,t,u^{n},v_{1}^{n},v_{2}^n) \right\|^2 \\
     &\hspace{2cm}
     + C\left\| \psi_1(\cdot,t,u^{n},v_1^{n},v_2^n) - \psi_1(\cdot,t_n,u^n,v_1^n, v_2^n) \right\|^2 \\
     &\hspace{2cm}
     + (1+\epsilon_0) \left\| \psi_1(\cdot,t_n,u^{n},v_1^{n},v_2^n) - \psi_1(\cdot,t_n,u_h^n,v_{1,h}^n, v_{2,h}^n) \right\|^2.
\end{align*}
As a result of hypothesis (iv),
\begin{align*}
    &\int_{t_n}^{t_{n+1}}\| \psi_1(\cdot,t,u,v_1,v_2) - \psi_1(\cdot,t,u^{n},v_{1}^{n},v_{2}^n) \|^2 \, \mathrm{d}t \leq 3B_1^2 \int_{t_n}^{t_{n+1}}\| u(t)-u^n \|_{1}^2 \, \mathrm{d}t, 
    \\
    &\int_{t_n}^{t_{n+1}}  \| \psi_1(\cdot,t,u^{n},v_1^{n},v_2^n) - \psi_1(\cdot,t_n,u^n,v_1^n, v_2^n)\|^2 \, \mathrm{d}t \leq 4B_4^2 k^2 \left(1 + \| u^n \|_{1}^2 \right), 
\end{align*}
Using $e_u^n=\xi_u^n-\eta_u^n$, $e_{v_i}^n=\xi_{v_i}^n-\eta_{v_i}^n$, together with Young's inequality and Lemma \ref{lem:proj-property-2d},
\begin{align*}
    &\| \psi_1(\cdot,t_n,u^{n},v_1^{n},v_2^n) - \psi_1(\cdot,t_n,u_h^n,v_{1,h}^n, v_{2,h}^n)\|^2 \leq B_1^2 \left\| |e_u^n| + |e_{v_1}^n| + |e_{v_2}^n| \right\|^2 \\
    &\leq C\| \xi_u^n\|^2 + Ch^{2r+2}(\| u^n\|_{r+1}^2 + \| {v}_1^n\|_{r+1}^2 + \| {v}_2^n\|_{r+1}^2) +  2B_1^2 (1+\epsilon_0) (\| \xi_{v_1}^n\|^2 + \| \xi_{v_2}^n\|^2).
\end{align*}
Combining all above estimates and summing from $n=0$ to $m$, we get
\begin{align*}
    &\sum_{n=0}^{m} \mathbf{1}_{\Omega_{\kappa,n}} \int_{t_n}^{t_{n+1}} \| \psi_1(\cdot,t,u,v_1,v_2) - \psi_1(\cdot,t_{n},u_h^{n},v_{1,h}^{n},v_{2,h}^n) \|^2 \, \mathrm{d}t \\
    &\leq C\sum_{n=0}^{m} \int_{t_n}^{t_{n+1}} \mathbf{1}_{\Omega_{\kappa,n}}\|u-u^n\|_{1}^2 \, \mathrm{d}t + Ck^2\sum_{n=0}^{m}(1+\| u^n\|_{1}^2 ) + Ck\sum_{n=0}^{m} \mathbf{1}_{\Omega_{\kappa,n}} \| \xi_u^{n}\|^2 \\
    &\quad 
    + 2B_1^2(1+\epsilon_0)^2 k\sum_{n=0}^{m} \mathbf{1}_{\Omega_{\kappa,n}} (\| \xi_{v_1}^n\|^2 + \| \xi_{v_2}^n\|^2) + Ch^{2r+2} k \sum_{n=0}^{m} (\| u^n\|_{r+1}^2 + \| {v}_1^n\|_{r+1}^2 + \| {v}_2^n\|_{r+1}^2). 
\end{align*}
Then the inequality \eqref{ineq:error2D-psi} follows from relabeling $2(1+\epsilon_0)^2$ with $2+\epsilon_0$. By the same argument, we can show the estimate for $g$ in \eqref{ineq:error2D-g} using hypothesis (v) in place of (iv). 
\end{proof}

\subsection{Proof of Lemma \ref{lem:error2D-term3}}

\begin{proof}
Let $\epsilon_2 >0$ be arbitrary. We first fix $i,j\in\{1,2\}$ and estimate the following two terms 
\begin{align*}
    T_1 = \max_{0 \leq l \leq m} k\sum_{n=0}^l -\mathbf{1}_{\Omega_{\kappa,n}} \big(a_{ij}^{n+1} - a_{h,ij}^{n+1},\, v_j^{n+1} \xi_{v_i}^{n+1} \big), \quad T_2 = \max_{0 \leq l \leq m} k\sum_{n=0}^l \mathbf{1}_{\Omega_{\kappa,n}} \big(a_{h,ij}^{n+1}\,\eta_{v_j}^{n+1},\, \xi_{v_i}^{n+1} \big). 
\end{align*}
For $T_1$, by Cauchy-Schwarz inequality, hypothesis (ii), Lemma \ref{lem:proj-property-2d} and the definition of subsets \eqref{def:error-subset-2d}
\begin{align*}
    |T_1| &\leq \alpha_1 k\sum_{n=0}^{m} \mathbf{1}_{\Omega_{\kappa,n}} \| v_j^{n+1}\|_{\infty} \| \xi_{v_i}^{n+1}\| \| e_u^{n+1}\| \leq \frac{\epsilon_2}{2} k\sum_{n=0}^{m} \mathbf{1}_{\Omega_{\kappa,n}} \| \xi_{v_i}^{n+1}\|^2 + C(1/\epsilon_2) k\sum_{n=0}^{m} \mathbf{1}_{\Omega_{\kappa,n}} \| v_j^{n+1}\|_{\infty}^2 \| e_u^{n+1}\|^2 \\
    &\leq \frac{\epsilon_2}{2} k\sum_{n=0}^{m} \mathbf{1}_{\Omega_{\kappa,n}} \| \xi_{v_i}^{n+1}\|^2 + C(1/\epsilon_2)\kappa k \sum_{n=0}^{m} \mathbf{1}_{\Omega_{\kappa,n}} \| \xi_u^{n+1}\|^2 + C(1/\epsilon_2)h^{2r+2} \kappa k \sum_{n=0}^{m} \mathbf{1}_{\Omega_{\kappa,n}} \| u^{n+1}\|_{{r+1}}^2. 
\end{align*}
Similarly, for $T_2$, by Young's inequality, hypothesis (ii), and Lemma \ref{lem:proj-property-2d},
\begin{align*}
    |T_2| &\leq \frac{\epsilon_2}{2} k\sum_{n=0}^{m} \mathbf{1}_{\Omega_{\kappa,n}} \| \xi_{v_i}^{n+1}\|^2 + \frac{1}{2\epsilon_2}k\sum_{n=0}^m \mathbf{1}_{\Omega_{\kappa,n}} \| a_{h,ij}^{n+1}\,\eta_{v_j}^{n+1}\|^2 \\
    &\leq \frac{\epsilon_2}{2} k\sum_{n=0}^{m} \mathbf{1}_{\Omega_{\kappa,n}} \| \xi_{v_i}^{n+1}\|^2 + C(1/\epsilon_2)h^{2r+2} k \sum_{n=0}^{m} \| v_j^{n+1}\|_{{r+1}}^2.
\end{align*}
Combining the above estimates of $T_1,T_2$ over $i,j=1,2$, we reach \eqref{ineq:lem-error-term3-2D}. 
\end{proof}

\subsection{Proof of Lemma \ref{lem:error2D-term4}}
\begin{proof}
Since $w_1,w_2$ are smooth exact solutions, applying integration by parts, and Cauchy-Schwarz inequality, we get
\begin{align*}
    \mathcal{I}_4 &\leq \sum_{n=0}^{m} \int_{t_n}^{t_{n+1}} \mathbf{1}_{\Omega_{\kappa,n}} \left|H^+(w_1(t)-w_1^{n+1}, w_2(t)-w_2^{n+1}, \xi_u^{n+1}) \right| \, \mathrm{d}t \\
    &= \sum_{n=0}^{m} \int_{t_n}^{t_{n+1}} \mathbf{1}_{\Omega_{\kappa,n}} \left|((w_1(t)-w_1^{n+1})_x + (w_2(t)-w_2^{n+1})_y, \xi_u^{n+1}) \right| \, \mathrm{d}t \\
    &\leq \sum_{n=0}^{m} \int_{t_n}^{t_{n+1}} \mathbf{1}_{\Omega_{\kappa,n}} \left\| (w_1(t)-w_1^{n+1})_x \right\|^2 + \left\| (w_2(t)-w_2^{n+1})_y \right\|^2 \, \mathrm{d}t + k\sum_{n=0}^{m} \mathbf{1}_{\Omega_{\kappa,n}} \|\xi_u^{n+1}\|^2. 
\end{align*}
We only provide a detailed estimate for $\| (w_1(t)-w_1^{n+1})_x \|^2$, as $\| (w_2(t)-w_2^{n+1})_y\|^2$ can be done similarly. 
Since $w_1(t)=a_{11}(t,u)v_1(t) + a_{12}(t,u)v_2(t)$, we write
\begin{align*}
    (w_1(t) - w_1^{n+1})_x 
    = T_{11} + T_{12}, 
    \qquad T_{1j} := \bigl(a_{1j}(t,u(t))v_j(t) -a_{1j}(t_{n+1},u^{n+1})v_j^{n+1}\bigr)_x,\, j=1,2.
\end{align*}
By hypothesis (vii), $T_{11}$ is bounded by 
\begin{align*}
    T_{11} &= a_{11}(v_1)_x - a_{11}^{n+1}(v_1)_x^{n+1} + (a_{11})_x v_1 - (a_{11}^{n+1})_x v_1^{n+1} \\
    &= a_{11}(v_1 -v_1^{n+1})_x + (a_{11}-a_{11}^{n+1})(v_1^{n+1})_x + (a_{11})_x(v_1 - v_1^{n+1}) + v_1^{n+1} (a_{11} - a_{11}^{n+1})_x \\
    &\leq C|(v_1 -v_1^{n+1})_x| + C|t-t_{n+1}|^{\frac{1}{2}}(1+|u|)|(v_1^{n+1})_x| + C|u-u^{n+1}||(v_1^{n+1})_x|  \\
    &+ C(1 + |u| + |v_1|)|v_1-v_1^{n+1}| + C\left(|u(t)-u^{n+1}| + |v_1-v_1^{n+1}| \right)|v_1^{n+1}| + C|t-t_{n+1}|^{\frac{1}{2}}(1+|u| + |v|)|v_1^{n+1}|. 
\end{align*}
Thus by the definition \eqref{def:error-subset-2d} of $\Omega_{\kappa,n}$, we have
\begin{align*}
    \sum_{n=0}^{m} \int_{t_n}^{t_{n+1}} \mathbf{1}_{\Omega_{\kappa,n}} \| T_{11} \|^2 \, \mathrm{d}t \leq C\kappa \sum_{n=0}^{m} \int_{t_n}^{t_{n+1}} \mathbf{1}_{\Omega_{\kappa,n}}\|u(t)-u^{n+1}\|_{2}^2 \, \mathrm{d}t + C\kappa k^2\sum_{n=0}^{m}\|v_1^{n+1}\|_{1}^2,
\end{align*}
where the $L^{\infty}$-norm of $u,v_1,(v_1)_x$ are absorbed by $\Omega_{\kappa,n}$. By the same argument, for $T_{12}$ we have
\begin{align*}
    \sum_{n=0}^{m} \int_{t_n}^{t_{n+1}} \mathbf{1}_{\Omega_{\kappa,n}} \| T_{12} \|^2 \, \mathrm{d}t \leq C\kappa \sum_{n=0}^{m} \int_{t_n}^{t_{n+1}} \mathbf{1}_{\Omega_{\kappa,n}}\|u(t)-u^{n+1}\|_{2}^2 \, \mathrm{d}t + C\kappa k^2\sum_{n=0}^{m}\|v_2^{n+1}\|_{1}^2.
\end{align*}
Similarly, the other term $\sum_{n=0}^{m} \int_{t_n}^{t_{n+1}} \mathbf{1}_{\Omega_{\kappa,n}} \| (w_2(t)-w_2^{n+1})_y \|^2 \, \mathrm{d}t$ can be bounded with the same bound. Combining these estimates, we get \eqref{ineq:error-estimate-term4-2D}. 
\begin{align*}
    \mathcal{I}_4 \leq k\sum_{n=0}^{m} \mathbf{1}_{\Omega_{\kappa,n}} \|\xi_u^{n+1}\|^2 + C\kappa \sum_{n=0}^{m} \int_{t_n}^{t_{n+1}} \mathbf{1}_{\Omega_{\kappa,n}}\|u(t)-u^{n+1}\|_{2}^2 \, \mathrm{d}t + C\kappa k^2\sum_{n=0}^{m} \bigl(\|v_1^{n+1}\|_{1}^2 + \|v_2^{n+1}\|_{1}^2 \bigr),
\end{align*}
which completes the proof. 
\end{proof}

\subsection{Proof of Lemma \ref{lem:error2D-term5}}
\begin{proof}
Let $\epsilon_3 >0$ be arbitrary. Since
\[ \| u - u_h^n\| = \bigl\| u - u^n + \xi_u^n - \eta_u^n \bigr\| \leq \| u-u^n\| + \| \xi_u^n\| + \| \eta_u^n\|, \]
by Corollary \ref{coro:nonlinear-convec-2D}, Young's inequality and triangle inequality, for $p \geq 1$, we have
\begin{align*}
    \mathcal{I}_5 
    &\leq C\sum_{n=0}^{m} \int_{t_n}^{t_{n+1}} \mathbf{1}_{\Omega_{\kappa,n}} \left(1+\|u(t)\|_{\infty}^p + \|u_h^{n}\|_{\infty}^p \right) \| u(t)-u_h^{n}\| \left(\| \xi_{v_1}^{n+1}\| + \| \eta_{v_1}^{n+1}\| + \| \xi_{v_2}^{n+1}\| + \| \eta_{v_2}^{n+1}\| \right)  \, \mathrm{d}t  \\ 
    &\leq C\sum_{n=0}^{m} \int_{t_n}^{t_{n+1}}  \| \eta_{v_1}^{n+1}\|^2 + \| \eta_{v_2}^{n+1}\|^2 + \mathbf{1}_{\Omega_{\kappa,n}} \epsilon_3 (\| \xi_{v_1}^{n+1}\|^2 + \| \xi_{v_2}^{n+1}\|^2) \, \mathrm{d}t \\
    &\quad + C\sum_{n=0}^{m} \int_{t_n}^{t_{n+1}} 3\left(\frac{1}{2\epsilon_3} +\frac{1}{2} \right) \mathbf{1}_{\Omega_{\kappa,n}} \left(1+\|u(t)\|_{\infty}^p + \|u_h^{n}\|_{\infty}^p \right)^2 \| u(t)-u^{n}\|^2 \, \mathrm{d}t \\
    &\quad + C\sum_{n=0}^{m} \int_{t_n}^{t_{n+1}} 3\left(\frac{1}{2\epsilon_3} +\frac{1}{2} \right) \mathbf{1}_{\Omega_{\kappa,n}} \left(1+\|u(t)\|_{\infty}^p + \|u_h^{n}\|_{\infty}^p \right)^2 \left(\|\xi_u^{n}\|^2 + \|\eta_u^{n}\|^2 \right) \, \mathrm{d}t \\
    &\leq Ch^{2r+2} k\sum_{n=0}^{m}(\| v_1^{n+1}\|_{r+1}^2 + \| v_2^{n+1}\|_{r+1}^2) + \epsilon_3 k \sum_{n=0}^{m}\mathbf{1}_{\Omega_{\kappa,n}} (\| \xi_{v_1}^{n+1}\|^2 + \| \xi_{v_2}^{n+1}\|^2) \notag \\
    &\quad + C\kappa \sum_{n=0}^{m} \int_{t_n}^{t_{n+1}} \|u(t)-u^{n}\|^2 \, \mathrm{d}t 
    + \kappa k\sum_{n=0}^{m} \mathbf{1}_{\Omega_{\kappa,n}} \|\xi_u^{n}\|^2 + C\kappa h^{2r+2} k\sum_{n=0}^{m}\| u^{n+1}\|_{r+1}^2,
\end{align*}
where the last inequality follows from the definition \eqref{def:error-subset-2d} of $\{ \Omega_{\kappa,n}\}$. 
\end{proof}

\subsection{Proof of Lemma \ref{lem:error2D-term6}}

\begin{proof}
We first decompose the source contribution. For each $0\leq l\leq m$, we write
\begin{align*}
&\sum_{n=0}^{l} \int_{t_n}^{t_{n+1}} \mathbf{1}_{\Omega_{\kappa,n}} \left((u_h^n)^\lambda-u(t)^\lambda,\xi_u^{n+1}\right) \,\mathrm{d}t \\
&\quad =
\sum_{n=0}^{l} \int_{t_n}^{t_{n+1}} \mathbf{1}_{\Omega_{\kappa,n}} \left((u^n)^\lambda-u(t)^\lambda,\xi_u^{n+1}\right) \,\mathrm{d}t 
+ k\sum_{n=0}^{l} \mathbf{1}_{\Omega_{\kappa,n}} \left((u_h^n)^\lambda-(u^n)^\lambda, \xi_u^{n+1}-\xi_u^n\right) \\
&\qquad
+ k\sum_{n=0}^{l} \mathbf{1}_{\Omega_{\kappa,n}} \left((u_h^n)^\lambda-(u^n)^\lambda,\eta_u^n\right) 
+ k\sum_{n=0}^{l} \mathbf{1}_{\Omega_{\kappa,n}} \left((u_h^n)^\lambda-(u^n)^\lambda,e_u^n\right).
\end{align*}
The last partial sum is nonpositive for every $l$. Indeed,
\[
    \left((u_h^n)^\lambda-(u^n)^\lambda,e_u^n\right)
    =
    \left((u_h^n)^\lambda-(u^n)^\lambda,u^n-u_h^n\right)
    \leq 0,
\]
because $\lambda$ is odd and $s\mapsto s^\lambda$ is monotone increasing. Therefore,
\begin{align*}
    \mathcal{I}_6 
    &\leq \max_{0 \leq l \leq m}  \sum_{n=0}^{l} \int_{t_n}^{t_{n+1}} \mathbf{1}_{\Omega_{\kappa,n}} \left((u^n)^{\lambda} - u(t)^{\lambda},\xi_u^{n+1} \right) \, \mathrm{d}t  
    + k\max_{0 \leq l \leq m} \sum_{n=0}^{l} \mathbf{1}_{\Omega_{\kappa,n}} \left((u_h^n)^{\lambda} - (u^n)^{\lambda},\xi_u^{n+1}-\xi_u^{n} \right)  \\
    &\qquad 
    + k\max_{0 \leq l \leq m} \sum_{n=0}^{l}  \mathbf{1}_{\Omega_{\kappa,n}} \left((u_h^n)^{\lambda} - (u^n)^{\lambda},\eta_u^{n} \right)   
    := T_1 + T_2 + T_3. 
\end{align*}
For $T_1$, by Cauchy-Schwarz inequality, \eqref{ineq:convex-2} and the definition of subsets \eqref{def:error-subset-2d}, we have
\begin{align*}
    T_1 &\leq k\sum_{n=0}^{m} \mathbf{1}_{\Omega_{\kappa,n}} \| \xi_u^{n+1}\|^2 + \frac{1}{4}\sum_{n=0}^{m} \int_{t_n}^{t_{n+1}}\mathbf{1}_{\Omega_{\kappa,n}} \left\| (u^n)^{\lambda} - u(t)^{\lambda} \right\|^2 \, \mathrm{d}t \\
    &\leq k\sum_{n=0}^{m} \mathbf{1}_{\Omega_{\kappa,n}} \| \xi_u^{n+1}\|^2 + \frac{\lambda^2}{2}\sum_{n=0}^{m} \int_{t_n}^{t_{n+1}}\mathbf{1}_{\Omega_{\kappa,n}} \left(\| u^n\|_{\infty}^{2\lambda-2} + \| u(t)\|_{\infty}^{2\lambda-2} \right) \left\| u^n - u(t) \right\|^2 \, \mathrm{d}t \\
    &\leq k\sum_{n=0}^{m} \mathbf{1}_{\Omega_{\kappa,n}} \| \xi_u^{n+1}\|^2 + \frac{\lambda^2}{2} \kappa \sum_{n=0}^{m} \int_{t_n}^{t_{n+1}}\mathbf{1}_{\Omega_{\kappa,n}} \left\| u^n - u(t) \right\|^2 \, \mathrm{d}t. 
\end{align*}
Similarly, for any $\epsilon_4 >0$ we get
\begin{align*}
    T_2 &\leq \epsilon_4 \sum_{n=0}^{m} \mathbf{1}_{\Omega_{\kappa,n}} \| \xi_u^{n+1}-\xi_u^{n}\|^2 + \frac{k^2}{4\epsilon_4}\sum_{n=0}^{m} \mathbf{1}_{\Omega_{\kappa,n}} \left\| (u_h^n)^{\lambda} - (u^n)^{\lambda} \right\|^2 \\
    &\leq \epsilon_4 \sum_{n=0}^{m} \mathbf{1}_{\Omega_{\kappa,n}} \| \xi_u^{n+1}-\xi_u^{n}\|^2 + Ck^2 \sum_{n=0}^{m} \mathbf{1}_{\Omega_{\kappa,n}} \left(\| u_h^n\|_{\infty}^{2\lambda-2} + \| u^n\|_{\infty}^{2\lambda-2} \right) \left\| e_u^n \right\|^2 \\
    &\leq \epsilon_4 \sum_{n=0}^{m} \mathbf{1}_{\Omega_{\kappa,n}} \| \xi_u^{n+1}-\xi_u^{n}\|^2 + C\kappa k^2 \sum_{n=0}^{m} \mathbf{1}_{\Omega_{\kappa,n}} \left\| \xi_u^n \right\|^2 + Ch^{2r+2} \kappa k^2 \sum_{n=0}^{m} \mathbf{1}_{\Omega_{\kappa,n}} \|u^n \|_{r+1}^2.
\end{align*}
and
\begin{align*}
    T_3 &\leq k\sum_{n=0}^{m} \mathbf{1}_{\Omega_{\kappa,n}} \| \eta_u^{n}\|^2 + k\sum_{n=0}^{m} \mathbf{1}_{\Omega_{\kappa,n}} \left\| (u_h^n)^{\lambda} - (u^n)^{\lambda} \right\|^2 \\
    &\leq Ch^{2r+2}k\sum_{n=0}^{m} \mathbf{1}_{\Omega_{\kappa,n}} \| u^{n}\|_{r+1}^2 + 2\lambda^2 k\sum_{n=0}^{m} \mathbf{1}_{\Omega_{\kappa,n}} \left(\| u_h^n\|_{\infty}^{2\lambda-2} + \| u^n\|_{\infty}^{2\lambda-2} \right) \left\| e_u^n \right\|^2 \\
    &\leq C\kappa k \sum_{n=0}^{m} \mathbf{1}_{\Omega_{\kappa,n}} \left\| \xi_u^n \right\|^2 + Ch^{2r+2}\kappa k\sum_{n=0}^{m} \mathbf{1}_{\Omega_{\kappa,n}} \| u^{n}\|_{r+1}^2.
\end{align*}
Combining the estimates of $T_1$--$T_3$ above, we get \eqref{ineq:error-estimate-term6-2d}. 
\end{proof}

\subsection{Proof of Lemma \ref{lem:w_h-v_h-2D}}
\label{appendix-w_h-v_h-2d}
\begin{proof}
We only prove the bound for $\| \xi_{w_1}^{n+1}\|^2$, as the bound for $\| \xi_{w_2}^{n+1}\|^2$ follows in the same way. By the error equation \eqref{eq:error-3-4-2d}, for any $z_h\in \mathbb V_h$,
\begin{gather*}
    (\xi_{w_1}^{n+1},z_h) = (\eta_{w_1}^{n+1},z_h) + \big(a_{11}^{n+1} - a_{h,11}^{n+1},\, v_1^{n+1} z_h \big) + \big(a_{h,11}^{n+1}\, \xi_{v_1}^{n+1},\, z_h \big) - \big(a_{h,11}^{n+1}\, \eta_{v_1}^{n+1},\, z_h \big) \\
    + \big(a_{12}^{n+1} - a_{h,12}^{n+1},\, v_2^{n+1} z_h \big) + \big(a_{h,12}^{n+1}\,\xi_{v_2}^{n+1},\, z_h \big) - \big(a_{h,12}^{n+1}\,\eta_{v_2}^{n+1},\, z_h \big).
\end{gather*}
Let $z_h = \xi_{w_1}^{n+1}$, by Young's inequality, Cauchy-Schwarz inequality, hypothesis (ii) and Lemma \ref{lem:proj-property-2d}, we have
\begin{align*}
    \| \xi_{w_1}^{n+1}\|^2 &\leq 7\| \eta_{w_1}^{n+1}\|^2 + C\| e_u^{n+1} v_1^{n+1}\|^2 + C\| e_u^{n+1} v_2^{n+1}\|^2  \\
    &\quad + 7\Lambda \| \xi_{v_1}^{n+1}\|^2 + 7\Lambda\|\eta_{v_1}^{n+1}\|^2 + 7\Lambda \| \xi_{v_2}^{n+1}\|^2 + 7\Lambda\|\eta_{v_2}^{n+1}\|^2 \\
    &\leq Ch^{2r+2}(\| w_1^{n+1}\|_{r+1}^2 + \| v_1^{n+1}\|_{r+1}^2 + \| v_2^{n+1}\|_{r+1}^2) + 7\Lambda (\| \xi_{v_1}^{n+1}\|^2 + \| \xi_{v_2}^{n+1}\|^2) \\
    &\quad + C\| \xi_u^{n+1} \|^2 (\|v_1^{n+1}\|_{\infty}^2 + \|v_2^{n+1}\|_{\infty}^2) + C\| \eta_u^{n+1} \|^2 (\|v_1^{n+1}\|_{\infty}^2 + \|v_2^{n+1}\|_{\infty}^2).
\end{align*}
Thus summing from $n=0$ to $m$ and multiplying by the subset indicator function, we reach 
\begin{align*}
    k\sum_{n=0}^m  \mathbf{1}_{\Omega_{\kappa,n}} \| \xi_{w_1}^{n+1}\|^2 &\leq 7\Lambda k\sum_{n=0}^m \mathbf{1}_{\Omega_{\kappa,n}} (\| \xi_{v_1}^{n+1}\|^2 + \| \xi_{v_2}^{n+1}\|^2)  
    + C\kappa k\sum_{n=0}^m \mathbf{1}_{\Omega_{\kappa,n}} \| \xi_{u}^{n+1}\|^2 \\
    &+ C\kappa h^{2r+2} k\sum_{n=0}^m \| u^{n+1}\|_{r+1}^2 + Ch^{2r+2} k\sum_{n=0}^m (\| w_1^{n+1}\|_{r+1}^2 + \| v_1^{n+1}\|_{r+1}^2 + \| v_2^{n+1}\|_{r+1}^2).
\end{align*}
By the same argument,
\begin{align*}
    k\sum_{n=0}^m  \mathbf{1}_{\Omega_{\kappa,n}} \| \xi_{w_2}^{n+1}\|^2 &\leq 7\Lambda k\sum_{n=0}^m \mathbf{1}_{\Omega_{\kappa,n}} (\| \xi_{v_1}^{n+1}\|^2 + \| \xi_{v_2}^{n+1}\|^2)  
    + C\kappa k\sum_{n=0}^m \mathbf{1}_{\Omega_{\kappa,n}} \| \xi_{u}^{n+1}\|^2 \\
    &+ C\kappa h^{2r+2} k\sum_{n=0}^m \| u^{n+1}\|_{r+1}^2 + Ch^{2r+2} k\sum_{n=0}^m (\| w_2^{n+1}\|_{r+1}^2 + \| v_1^{n+1}\|_{r+1}^2 + \| v_2^{n+1}\|_{r+1}^2).
\end{align*}
Therefore, we get \eqref{ineq:w_h-v_h-2D} by combining the two inequalities above. 
\end{proof}

\end{appendices}


 \printbibliography


\end{document}